%% file: Final-ControlPaper.tex
\documentclass[11pt, oneside,reqno]{amsart} %without reqno print equations label at the left
\usepackage[toc]{appendix}
\usepackage{hyperref}
\usepackage{bbm} % for indicator function

\usepackage{relsize} %increase size of subset symbol
\usepackage[showframe=false]{geometry}%. =true. shows frame margin
\include{MyPackages}

\begin{document}

\title{Martingale approach to control for general jump processes}
%\date{\today}
\author{Hern\'andez-Hern\'andez, M. E.$^1$}
\address{$^{1,2,3}$Department of Statistics, University of Warwick, Coventry CV4 7AL UK}
%\email{m.hernandez-hernandez@warwick.ac.uk}
\thanks{$^{1,2,3}$Department of Statistics, University of Warwick, Coventry CV4 7AL UK }
\thanks{$^1$ Email: m.hernandez-hernandez@warwick.ac.uk}
%\curraddr{$^1$Department of Statistics, University of Warwick, Coventry CV4 7AL UK}

%    author one information
\author{Jacka, S. D.$^2$}
%\address{$^2$Department of Statistics, University of Warwick, Coventry CV4 7AL UK} 
%\email{S.Jacka@warwick.ac.uk}
\thanks{$^{2}$ Saul Jacka gratefully acknowledges funding received from the EPSRC grant EP/P00377X/1 and is also grateful to the Alan Turing Institute for their financial support under the EPSRC grant EP/N510129/1. Email: s.jacka@warwick.ac.uk}

%    author two information
\author{Mijatovi\'{c}, A.$^3$ }
%\address{$^3$Department of Statistics, University of Warwick, Coventry CV4 7AL UK} 
%\email{a.mijatovic@warwick.ac.uk}
%\thanks{$^3$ Department of Statistics, University of Warwick, Coventry CV4 7AL UK\\ Email: a.mijatovic@warwick.ac.uk }
\thanks{$^3$  Email: a.mijatovic@warwick.ac.uk}

\begin{abstract} We provide  verification theorems (at different levels of generality) for infinite  horizon stochastic control problems in continuous time for semimartingales. The control framework is given as an abstract "martingale formulation", which encompasses a broad range of standard control problems. Under appropriate conditions we show that the set of admissible controls gives rise to a certain class of controlled special semimartingales. Our results generalise both the standard controlled \ito- and \levy-diffusion  settings as we allow ourselves to locally control not only the drift and diffusion coefficients, but also the jump intensity measure of the jumps. As an illustration, we present a few examples with explicit solutions.% such as an infinite horizon control problem with a running cost function being quadratic in the space variable and in the infinitesimal drift of the underlying dynamics. 
\end{abstract}

%    For articles to be published after 1 January 2010, you may use
%    the following version:
%.   \subjclass[2010]{34A08, 26A33, 34A12, 34A05, 60G52}

\keywords{Control martingale approach, Semimartingale characteristics, Bellman's process}

\maketitle
\section{Introduction}\input{VF7-Introduction}

\section{Notation}\label{S:Notation}\input{VF7-Notation}

\section{Stochastic control setting} \label{S:ControlSetting} \input{VF7-ControlSetting} %1) Control setting, 2) Example of policies, 3) Theorems about underlying dynamics of controlled processes: X being a semimartingale and the identification of its characteristics

%%%%%%%%%%%%% VERIFICATION RESULTS     %%%%%
%%%%%%%%%%%%%%%%%%%%%%%%%%%%%%%%%%%%%%
\section{Control problem and Verification theorems}\label{S:MainResults}
\input{VF7-VerifTheoremInfinite}

\section{Finite horizon case and other extensions}\label{S:Gral}
\input{VF7-FiniteCase} %Sketch  finite horizon case
\subsection{Possible extensions}
\input{VF7-OtherGeneralisations} %Generalisations and final comments

\section{Applications}\label{S:Applications}

\input{FV7-Examples}

%%%%%%%%%%%%%     PROOFS    %%%%%%%%%%%%
%%%%%%%%%%%%%%%%%%%%%%%%%%%%%%%%%
%%%%%%%%%%%%%%%%%%%%%%%%%%%%%%%%%
\appendix

\section{Proofs of Results in Sections \ref{S:ControlSetting} and \ref{S:Dynamics}}\label{SectionP}
\input{V7-Proofs-I} %Proofs: Underlying dynamics
\section{Proofs of Results in Sections \ref{S:MainResults} and \ref{S:classes}}
\input{V7-Proofs-II} %Proofs: Verification results
\section{Proofs of Results in Section \ref{S:Applications}} 
\input{V7-Proofs-III} %Proofs: Examples
\bibliography{bibfile}
\bibliographystyle{plain}

\end{document}

%% file: MyPackages.tex
\usepackage{enumitem,array,etoolbox}
\usepackage{amsmath, amsthm, amsfonts, amssymb,fge,cite,fixmath}
\usepackage[parfill]{parskip} 
\usepackage[applemac]{inputenc}
\usepackage{epsfig,epic,eepic} 
\usepackage[english]{babel}
\usepackage{xcolor} % for colors

\patchcmd{\quote}{\rightmargin}{\leftmargin .5cm \rightmargin}{}{}  %command to modify quote environment - Added

\usepackage{enumitem} %to modify left margin in list enumerate and itemize environments

\usepackage{dsfont} %to use \mathbbm{1} for the indicator function or $\mathds{1}$ 
\usepackage{tensor} %for left subscripts

\usepackage{index,hyperref,fancyhdr,amssymb,enumerate,color}
\usepackage[mathscr]{eucal}
\usepackage{MnSymbol,anysize,float,calligra}
\newtheorem{theorem}{Theorem}[section]
\newtheorem{proposition}[theorem]{Proposition}
\newtheorem{lemma}[theorem]{Lemma}
\newtheorem{corollary}[theorem]{Corollary}

\newtheorem{definition}[theorem]{Definition}

\theoremstyle{definition}
\newtheorem{remark}[theorem]{Remark}
\newtheorem{remarks}[theorem]{Remarks}

\newcommand{\MyRemark}[1]{%\begin{remark} \textcolor{blue}{#1} \end{remark}
}

\newcommand{\MySummary}[1]{%\begin{summary} \textcolor{blue}{#1} \end{summary}
}

\newcommand{\MyRemarkM}[1]{%\begin{remark} \textbf{?} \textcolor{red}{#1} \end{remark}
}

\newcommand{\Myfootnote}[1]{
}
\newcommand{\MyCalc}[1]{
}

\newcommand{\Out}[1]{
}

\numberwithin{equation}{section}

\newcommand*\tr{\mathop{}\!\mathrm{Tr}}
\newcommand*\Hess{\mathop{}\!\mathrm{H}}
\newcommand*\diff{\mathop{}\!\mathrm{d}}

\newcommand{\Ro}{\mathbb{R}^n_0}
\newcommand{\Mn}{\mathrm{M}_{n\times n}(\rr)}
\newcommand{\Matrixn}[1]{(#1_{ij})_{1\le i,j \le n}}

\newcommand{\Rp}{\mathbb{R}_+}
\newcommand{\rp}{\mathbb{R}_+}
\newcommand{\rd}{\mathbb{R}^n}

\newcommand{\lrangle}[2]{\langle\,#1, #2 \, \rangle}

\newcommand{\rr}{\mathbb{R}}
\newcommand{\Mremark}[1]{\begin{remark}\textcolor{black}{#1}\end{remark}
}
\newcommand{\MremarkB}[1]{%\begin{remark}\textcolor{brown}{#1}\end{remark}
}
\newcommand{\MremarkO}[1]{%\begin{remark}\textcolor{violet}{#1} \end{remark}
}
\newcommand{\Mcomment}[1]{\textcolor{black}{#1} 
}
\newcommand{\McommentO}[1]{%\textcolor{violet}{#1} 
}
\newcommand{\Mqo}[1]{%\textcolor{red}{\textbf{Q0:} #1} 
}
\newcommand{\Fa}{\mathbb{F}^{\alpha}}
\newcommand{\Pa}{\mathbb{P}^{\alpha}}
\newcommand{\FPa}{(\Fa, \Pa)}
\newcommand{\F}{\mathbb{F}}
\newcommand{\Prb}{\mathbb{P}}

\newcommand{\V}[1]{\mathbf{#1}}
\newcommand*\Msigma{\mathbold{\sigma}}
\newcommand*\Mmu{\mathbold{\mu}}
\newcommand*\Malpha{\mathbold{\alpha}}
\newcommand*\Mbeta{\mathbold{\beta}}

\newcommand{\cad}{c\`adl\`ag }
\newcommand{\ito}{It\^o }
\newcommand{\levy}{L\'evy }
\newcommand{\holder}{H\"{o}lder }
\newcommand{\xpa}{\V{X}^{\Malpha}}

\newcommand{\xp}{\V{X}^{\Malpha^{r,\xi}}}
\newcommand{\ap}{\Malpha^{r,\xi}}

\newcommand{\Eu}{\mathbb{E}_{\V{x}}^{\V{u}}}

\newcommand{\ypa}{\V{Y}^{\Mbeta_t^{\alpha}}}
\newcommand{\bpa}{\Mbeta_t^{\alpha}}
\newcommand{\etat}{\eta_t^{\alpha}}
\newcommand{\Apya}{\mathcal{A}^p_{t,\etat}}

\newcommand{\upx}{\V{u}^{\xi,\oplus_t}}
\newcommand{\uxp}{\V{u}^{\V{x},\oplus_t}}

\newcommand{\upxo}{\V{u}^{\oplus_t}}

\newcommand{\zpx}{\V{Z}^{\V{u}^{\xi,\oplus_t}}}
\newcommand{\zxp}{\V{Z}^{\V{u}^{\V{x},\oplus_t}}}

\newcommand{\zpxo}{\V{Z}^{\V{u}^{\oplus_t}}}

\newcommand{\apg}[2]{\Malpha^{#1,\V{#2}}}
\newcommand{\Ag}[2]{\mathcal{A}^p_{#1,\V{#2}}}

\newcommand{\xprx}{\V{X}^{\Malpha^{r,\V{x}}}}
\newcommand{\aprx}{\Malpha^{r,\V{x}}}
\newcommand{\xptx}{\V{X}^{\Malpha^{t,\V{x}}}}
\newcommand{\aptx}{\Malpha^{t,\V{x}}}

\newcommand{\xpx}{\V{X}^{\Malpha^{\V{x}}}}
\newcommand{\apx}{\Malpha^\V{x}}
\newcommand{\xpz}{\V{X}^{\Malpha^{\V{z}}}}
\newcommand{\apz}{\Malpha^\V{z}}

\newcommand{\Xc}{\V{X}^{\apx,c}}
\newcommand{\Xd}{\V{X}^{\apx,d}}
\newcommand{\nx}{\eta^{X,\apx}}
\newcommand{\ny}{\eta^{Y,\beta}}

\newcommand{\Ap}{\mathcal{A}^p_{r,\xi}}

\newcommand{\Arx}{\mathcal{A}^p_{r,\V{x}}}

\newcommand{\Ax}{\mathcal{A}^p_{\V{x}}}
\newcommand{\Az}{\mathcal{A}^p_{\V{z}}}

\newcommand{\apo}{\Malpha}
\newcommand{\xpo}{\V{X}^{\Malpha}}

\newcommand{\apoT}{\tilde{\Malpha}}
\newcommand{\ypo}{\V{Y}^{\tilde{\Malpha}}}
\newcommand{\apoTx}{\tilde{\Malpha}^{\V{x}}}
\newcommand{\ypox}{\V{Y}^{\tilde{\Malpha}^{\V{x}}}}

\newcommand{\xpi}{X^{i}}

\newcommand{\Xic}{X^{i,c}}

\newcommand{\Y}{\V{Y}}

\newcommand{\AxT}{\tilde{\mathcal{A}}_{\V{x}}}

\newcommand{\Atx}{\mathcal{A}_{t,\V{x}}^p}
\newcommand{\Ex}{\mathbb{E}_{\V{x}}^{\alpha}}
\newcommand{\Eob}{\mathbb{E}_{0}^{\beta}}

\newcommand{\Etx}{\mathbb{E}_{t,\V{x}}^{\alpha}}

\newcommand{\Jx}{J^{\alpha}(\V{x})}

\newcommand{\Lasx}{L^{\apx_s}}

\newcommand{\Las}{L^{\ap_s}}
\newcommand{\as}{\mathbb{P}^{\alpha}-a.s.}
\newcommand{\Ft}{\mathcal{F}_t}

%% file: VF7-Introduction.tex
As far as we know, to date the most general setting for controlled Markov processes in continuous time is that in \cite{OKSulem}. 
In this paper %studies the characterisation of the value function  for an infinite  horizon control problem in continuous time for $\rd-$valued stochastic processes with continuous action and state space.  
  we provide sufficient conditions of optimality for the control problem of minimising  the cost functional
      \begin{equation}\label{J-infinite}
  J(\xpx,\apx):=  \Ex \left [   \int_0^{\infty} e^{- \gamma_t^{\Malpha}}f\left (\xpx_t,\,\apx_t\right)\diff t\right ],
     \end{equation}
  over all \textit{admissible} state-control processes $\left (\xpx,\,\apx \right ) \in \Ax$, where $\xpx_0 = \V{x}$. We assume that each control process $\apx_{\cdot} = (\Msigma_{\cdot}, \nu_{\cdot}, \Mmu_{\cdot})$ determines the volatility, the jump intensity and the drift of the  system $\xpx$. We stress that, unlike most of the literature including \cite{OKSulem}, the action space allows us to choose the jump measure $\nu$ in a position-dependent way and from an arbitrary class.  The  local dynamics of the controlled process $\xpx$ are described by nonlocal operators of the form 
  \begin{align}\label{D:LaIntro}
   (L^\V{a} g)(\cdot) := (\V{u} +\Mmu)^T  \nabla g(\cdot) + \frac{1}{2} \tr( \Msigma^T \Hess g\, \Msigma)(\cdot)  +\int_{\Ro} \left ( g(\cdot+\V{y}) - g(\cdot) -  \V{y}^T \nabla g (\cdot ) \right ) \nu(\diff \V{y}),
   \end{align}
  where each $\V{a} = (\Msigma, \nu, \Mmu)$ is an element in the (properly defined) set of  available actions $A$.    
 Here, $f$ is the given running cost function 
 and $ \gamma_t^{\Malpha}$ denotes a discounting process depending on the admissible pair $\left (\xpx,\,\apx \right ) $.
  
The contribution of this paper to the subject  of stochastic control is twofold: firstly, we set up a general (and abstract) martingale formulation  for multidimensional controlled semimartingales whose differential characteristics are controlled continuously in time.  Secondly, we explore different characterisations of the value function at different levels of generality, including the standard verification theorems which are based on the corresponding HJB type equation. We also  show the interplay between the different assumptions and the relationship with the classical  stochastic differential equation (SDE) setting (see, for example,  \cite{Pham}, \cite{OKSDE} and references therein).  In particular,  our control  formulation encompasses the controlled \ito diffusion case and, further, generalises the controlled jump diffusion case presented in \cite{OKSulem}.

The basic structure of a stochastic control problem involves the description of the  system dynamics and the corresponding class of admissible controls. Most of the literature dealing with verification theorems is set up via either  a \textit{strong} or a \textit{weak} SDE formulation with a \emph{Markovian} structure. That is,  the evolution of each  controlled system is governed by an SDE whose relevant coefficients are functions of the type $h (t,\V{x}, u(t,\V{x}))$, where $u$ is the control policy.  In such a setting  the definition of the \lq\lq smallest'' class of admissible controls  relies then on  well-posedness results (existence and uniqueness of solutions) of the corresponding SDE's.  
  Along these lines,   the study of continuous time stochastic control of Markov processes, and more generally of Ito-diffusions,  has been widely researched.    For a quick overview of this formulation, see the surveys \cite{Borkar2005} and \cite{APham2005}. For a more detailed study  we refer to \cite{FlemingR1975}, \cite{FlemingSoner}, \cite{YongZ} and references therein. Control settings allowing discontinuous processes  include those  of stochastic control for Markovian jump systems (also known as Markov decision processes) as given, for example, in  \cite{Miller1968}, \cite{Stone1973}, \cite{Pliska1975}; and, more generally, controlled jump diffusions (also known as \levy diffusions) as  in \cite{OKSulem}.   

When compared to the controlled SDE approaches, either for  (continuous) \ito diffusions or for (discontinuous) \levy diffusions,  the amount of research outside these settings, such as the \emph{control martingale approach}, is less. Some works in this direction are \cite{ElKaroui1981}, \cite{ElKaroui1987}, \cite{EK1987}, \cite{Karatzas2006}, \cite{JM0}.  
   In this paper we also consider a control setting outside the SDE formulation. The main features of our model are the following: 
 \begin{itemize}[leftmargin=0.75cm]
 \item [i)]  Our control problem is based on a \textit{martingale approach}  similar to the one given in \cite{JM0} (see also  \cite{Davis1979}, \cite{EK1987}), wherein the underlying local dynamics, i.e. each admissible pair $(\xpx,\,\apx)$, is characterised  as a solution to a   \textit{control martingale problem}. This  approach is also similar to the one introduced by Stroock and Varadhan for SDE's. This formulation allows very general underlying dynamics for the controlled system where neither is the dynamics of the system assumed to be Markovian  nor is the class of admissible policies restricted to be \textit{Markovian or state-feedback} policies. We assume a  multidimensional state space with  no state constraints. The general setting  presented here encompasses typical stochastic control problems for \ito difusions.  
 \item [ii)] The action set  $A$ is an open abstract space for which each action $\V{a} \in A$ is a triplet which determines (locally) the diffusion coefficient, the jump intensity and  the drift  of the underlying process. We are thus outside the standard assumption of taken $A \subset \mathbb{R}^k$, $k\in \mathbb{N}$. The local dynamics of the controlled system $\xpx$ is  described by means of the family of \levy operators $\{L^{\V{a}}\,:\, \V{a}\in A\}$ as defined in \eqref{D:LaIntro}.    The use of these operators allows us to study a general class of controlled processes without any \emph{a priori} restriction to the Markovian class.  
\end{itemize}
Our main results are the following:
\begin{itemize}[leftmargin=0.75cm]
\item [i)]   \emph{Controlled semimartingale dynamics.} We prove that  our  controlled processes are special semimartingales whose  \textit{differential semimartingale characteristics} are related to the associated control policy $\apx$ (Corollary \ref{X-charact}). The generality of our setting imposes stronger integrability conditions for the control $\apx$ so as to guarantee the  finiteness of the $p$th moments of $\xpx$ (Proposition \ref{qth-moments}), which are crucial when dealing with value functions of polynomial growth.  In the standard SDE approach, where polynomial growth is also a natural assumption, the corresponding integrability conditions are usually granted by the \ito conditions. This is discussed in Section \ref{S:SDEcase}.  
\item [ii)] \emph{Dynamic Programming Principle (DPP).} In order to use the dynamic programming approach to study our control problem, our starting point is to prove that the  DPP (also known as Bellman's principle of optimality) holds true in our abstract martingale setting (Lemma \ref{L:DPP} and Lemma \ref{P0-S}). Such a principle is well-known for continuous controlled Markov processes. However, under more general frameworks, its validity is either assumed true because of \emph{its clear intuitive meaning} or the reader is referred to references which, in many cases, deal with slightly different models. 
\item [ii)] \emph{Verification theorems.} Under different assumptions on the initial control data (running cost function) and considering different classes of admissible policies (including the Markovian class),  we provide different results to characterise the value function including the standard \emph{verification theorems} (Theorem \ref{VT0}).   These results provide sufficient conditions to characterise the optimal payoff function (or value function) as a solution to the corresponding Hamilton-Jacobi-Bellman (HJB) equation and, at the same time, they allow us to determine optimal feedback controls via the pointwise optimisation of the HJB equation.  The proofs of these  theorems rely on the probabilistic counterpart of \textit{the dynamic programming principle}. Essentially, one needs to show  that the so-called \textit{Bellman process} associated with a control process (see  definition in \eqref{Sax00}) is a submartingale for any admissible policy, and it is a martingale whenever the policy is optimal. In particular, we cover in detail the case where the running cost function $f$ has polynomial growth.% of degree $q \ge 1$. 
  \item[iii)]  \emph{Examples}. As applications of our  results, we provide a few examples with explicit solutions. In one of them the optimal controlled process is a Brownian motion that is jumped to the origin whenever it leaves a certain region. We also give the explicit solution for a control problem with a quadratic running cost function. For this case we  determine the value function and exhibit an optimal policy  whose associated optimal controlled process is an  Ornstein-Uhnlenbeck  type process.  There is a close  connection between  linear-quadratic (LQ) optimal control problems  and the quadratic case presented here. Indeed, although  the controlled system associated with each admissible control  is not assumed to be linear,  we still obtain a linear, optimally controlled process. Furthermore, we get an optimal control whose drift component is linear in the state variable and whose  diffusion component and jump intensity are related to the solution to an algebraic Riccati equation (see \eqref{Riccati}-\eqref{D:delta} in Section 4).  
   \end{itemize}
 Needless to say, we are aware of the limitations arising when  considering  verification results in the context of \emph{classical} (regular enough) solutions.  It is well-known that (even for standard controlled diffusions) the regularity of the value function cannot be guaranteed in general. Nevertheless, even though our verification results are based on the assumption of a smooth solution to the HJB equation,  our control setting gives a promising starting point to handing very general stochastic control problems in continuous time and with abstract control sets. The generalisation of our results to the context  of \textit{viscosity solution} is left as part of our  future research.  A brief discussion of the finite horizon case and other possible extensions are given in Section \ref{S:Gral}. %\Mcomment{In particular, the verification theorems for the finite horizon case can be obtained without much difficulty.}
 
  \McommentO{ In general, proving the \emph{a priori} regularity  for the value function $V$ is a difficult task that can be handled with the notion of viscosity solutions. In fact, as many examples show it, one cannot expect for $V$ to be a smooth function. We thus leave the study of this problem  for  future research. Classical verification theorems allow to solve examples of control problems where one can find, or at least there exists, a smooth solution to the associated HJB equation. They apply successfully for control problems where the diffusion term does not depend on the control and is uniformly elliptic, since in this case the HJB equation is semilinear  and so classical existence results for smooth solutions exist. They also apply for some specific models with control on the diffusion term, such as some typical Merton?s portfolio selection problem, and more generally to extensions of Merton?s model with stochastic volatility, since in this case, the HJB equation may be reduced after a suitable transformation to a semilinear equation. However, in the general case of nondegeneracy of the diffusion term and in a number of applications, the value function might be not smooth or it is not possible to obtain a priori the required regularity. Moreover, for singular control problems, the value function is in general not continuous at the terminal date. Then, the classical verification approach does not work and we need to relax the notion of solution to the HJB equation.}

 The rest of the paper is organised as follows. Section \ref{S:Notation} introduces some standard  notation.  Section \ref{S:ControlSetting}  describes  the control setting of interest.  In Section \ref{S:Dynamics} we study the underlying dynamics of our controlled process.  Here we show the semimartingale structure of our formulation.  %Here we show that, under suitable conditions for the admissible controls, each controlled process is a semimartingale for which we also obtain the associated semimartingale characteristics. 
  Then, in Section \ref{S:MainResults} we define the cost structure of the problem and state the main results of the paper: the dynamic programming principle and its (sub-)martingale formulation (Lemma \ref{P0-S}), as well as the characterisation results of the value function (Lemma \ref{VT0-SubPhi}, Theorem \ref{VT0} and   Theorem \ref{VT-EFinite}). In Section \ref{S:classes} we consider different classes of admissible controls such as  the Markovian case and its connection with the standard SDE setting.  The finite horizon case and other possible generalisations are briefly discussed  in Section \ref{S:Gral}.  We provide some applications in Section \ref{S:Applications}. Finally, for the sake of clarity, the proofs of all our results are collected in the  Appendices.

%% file: VF7-Notation.tex
Let $\rd$ be the $n$-dimensional Euclidean space where each point $\V{x}$ is expressed by a column vector $\V{x} =(x_1, \ldots, x_n)^T$. As usual,  the superscript $"T"$ denotes the transpose of a vector (or a matrix). The set of nonnegative real numbers is denoted by $\rp$.  For any two vectors $\V{x},\V{y}\in \rd$, the inner product and the Eucliden norm are denoted by $\V{x}\cdot \V{y} :=\sum_{i=1}^n x_i y_i $ and  $|\V{x}| := \sqrt{\V{x}\cdot \V{x}}$, respectively.   
 
  Notation $M(\rd)$, $B(\rd)$ and $C (\rd)$  denote the  spaces  of real-valued measurable functions, bounded measurable functions and continuous functions on $\rd$, respectively.  These spaces are endowed with the sup-norm $||f||  = \sup_\V{x} |f (\V{x})|$. We denote by $C^{k} (\rd)$, $k \in \mathbb{N}$, the  space  of real-valued $k$-times continuously differentiable functions  defined on $\rd$. \McommentO{with the norm given by $||f||_{C^K}:= \sum_{j=0}^k ||\partial_j f||$}  An additional subscript $c$  will be used to denote the corresponding space of functions with compact support and by $C_c^{\infty}(\rd) := \bigcap_{k=1}^{\infty} C_c^k(\rd)$ we denote the space of infinitely often continuously differentiable functions on $\rd$ with compact support. 
  
 Given a metric space $(E,d)$, $\mathcal{B}(E)$ denotes the Borel $\sigma-$algebra compatible with the metric $d$.   Given the measure space $(E, \mathcal{B}(E),\rho)$, $L^q (E, \mathcal{B}(E),\rho)$ (in brief  $L^q(\rho)$ if there is no risk of confusion), $1 \le q < \infty$, stands for the Banach space of all equivalence classes of mappings $f:E \to \rd$ which agree a.e. with respect to $\rho$ and for which $||f||_q < \infty$, where the $q$-norm $||\cdot||_q$ is given by $||f||_q := \left (  \int_{E} | f(x)|^q \diff \rho \right )^{1/q}$.

Let $\Mn$ be the set of real-valued $n\times n$-matrices endowed with the Frobenius norm\McommentO{\footnote{Since all norms on a finite dimensional space are equivalent, choosing this one is just a matter of taste.}} $||\V{B}|| := \tr (\V{B} \V{B}^T)^{1/2}$, $\V{B} = \Matrixn{B}  \in \Mn$, where $\tr(\V{C}):= \sum_{i=1}^n C_{ii}$ denotes the trace  of  the matrix $\V{C}\in  \Mn$. The $n\times n$ identity matrix will be denoted by $\mathbf{I}$. %Notation $\mathcal{B}(\Mn)$ stands for the Borel $\sigma$-algebra  in $\Mn$.

Let $\Ro := \mathbb{R}^n\fgebackslash \{0\}$ and fix a $p \ge 2$. Define
\begin{equation}\label{D:Lm}
\mathcal{M}_p := \left \{   \text{ measures $\nu$  on  $(\Ro, \mathcal{B}(\Ro))$  such that $ \int_{\Ro} |\V{y}|^2 \vee |\V{y}|^p \nu(\diff \V{y}) < + \infty$} \right \}.\end{equation}

Since $1 \wedge |\V{y}|^2 \le |\V{y}|^2$ for all $\V{y} \in \rd$, it follows that each $\nu \in \mathcal{M}_p$ is a L\'evy measure, i.e. \begin{equation}\label{LevyCondition}
 \int_{\Ro} (1 \wedge |\V{y}|^2) \nu (\diff \V{y}) < +\infty,
 \end{equation} and, further, $\nu$ satisfies $ \int_{\Ro} |\V{y}|^2 \nu (\diff \V{y}) < + \infty$.  Here, we used the notation $a \wedge b:= \min \{a,b\}$  and $a \vee b:= \max \{a,b\}$,  for any $a,b \in \rr$.  The space $\mathcal{M}_p$ is endowed with a suitable weak convergence topology and its Borel $\sigma$-algebra is denoted by $\mathcal{B} (\mathcal{M}_p)$. \McommentO{The Borel $\sigma$-algebra is also the smallest $\sigma$-algebra such that the mappings $\nu \mapsto \int  f(\V{y}) \nu (\diff \V{y})$ are measurable for $f \in C_c (\Ro)$. We can introduce a filtration $(\mathcal{F}^{\mathcal{M}_p}_t)$  where $\mathcal{F}^{\mathcal{M}_p}_t$ is the $\sigma$-algebra generated by  the family of measures $\{ 1_{(0,t]} \nu \, : \, \nu \in \mathcal{M}_p\}$.}

For any $\Msigma \in \Mn$, the matrix $ \Msigma^T \Msigma $  is symmetric positive semidefinite.\Myfootnote{That is, $<\V{x}, \Msigma \Msigma^T \V{x}> = <\V{x}^T \Msigma \Msigma^T \V{x}> \ge 0$ for all vectors $\V{x}$. Indeed, let $y = \Msigma^T \V{x}$, then $<\V{x}, \Msigma \Msigma^T \V{x}> = \V{x}^T \Msigma \Msigma^T \V{x} = y^T y = \sum_{k=1}^n y^2_k = || \Msigma T \V{x}||^2 = || y|| \ge 0$.  Clearly $\Msigma \Msigma ^T$ is symmetric, and a symmetric matrix is positive-definite if and only if it has only positive eigenvalues.} By  Theorem~8.1 in \cite[Chapter 8, p. 37]{Sato1999}  and Theorem~7.10 in \cite[Chapter 7, p. 35]{Sato1999}, given  $\Msigma \in \Mn$, $ \Mmu \in \rd$ and $\nu \in \mathcal{M}_p$, there exists a unique in law $\rd$-valued process $\V{X}^\V{x} := (\V{X}_t^\V{x})_{t \in \mathbb{R}_+}$, started at $\V{x}\in\rd$, having stationary independent increments and satisfying
\begin{equation}\nonumber 
\mathbb{E}_\V{x} \left [ e^{i \V{u} \cdot (\V{X}_t^\V{x} - \V{x})}\right ] = e^{ t\,\Psi (\V{u})}, \quad \text{ for all } \V{u}\in \rd, \,\, t \in \Rp,
\end{equation}
where
\begin{equation}\label{D:Psi}
 \Psi (\V{u}) := -\frac{1}{2} \V{u} \cdot \Msigma^T \Msigma \V{u} + i \V{u} \cdot \Mmu + \int_{\Ro} \left( e^{i\V{u}\cdot \V{y}} - 1 - i\V{u}\cdot \V{y} \right) \nu(\diff  \V{y}),\quad  \V{u} \in \rd. 
\end{equation}
The  infinitesimal generator of the $n-$dimensional process $\V{X}^\V{x}$ is the operator $L$,
 with domain $\text{Dom} (L)$, defined for any $g \in C_c^2(\rd)$ by
\begin{align}\label{D:L-levy}
   (L g)(\cdot) := \Mmu^T  \nabla g(\cdot) + \frac{1}{2} \tr (\Msigma^T\Hess g\,\Msigma)(\cdot)  +\int_{\Ro} \left ( g(\cdot+\V{y}) - g(\cdot) -  \V{y}^T \nabla g (\cdot ) \right ) \nu(\diff \V{y}),
   \end{align}
   see  \cite[Theorem 31.5, p. 208]{Sato1999}. Here, $\nabla g$ and $\Hess g$ denote the gradient and the Hessian  in the variable $\V{x}$  with components $\partial_i g:=\frac{\partial }{\partial x_i}g$, $1 \le i \le n$, and $\partial_{ij}^2 g := \frac{\partial^2}{\partial x_i \partial x_j}g$, $1 \le i,j \le n$,  respectively.

%% file: VF7-ControlSetting.tex
We are interested in studying a stochastic control problem in a continuous setting:  continuous time and continuous action and state space. We shall consider a  \emph{martingale control formulation}  to describe the dynamics of each controlled system, or rather its law, as determined by the  \emph{martingale property} of a certain class of processes. For each controlled system, the local dynamics are described by a given family of operators defined on a certain class of test functions.

\emph{Notation. Given a metric space $(E,d)$, $\mathcal{P}(E)$ denotes the set of probability measures on $(E,\mathcal{B}(E))$.  The set of \cad functions (right-continuous  with left-limits) on $\rp$ with values in $E$ is denoted by $D(\rp; E)$ and is endowed with the Skorohod topology \cite[Chapter 3, Section 5]{EK}.   Given a stochastic basis $(\Omega, \mathcal{F}, \F, \Prb)$, all equalities and inequalities between random variables are understood to hold $\Prb-$almost surely, unless stated otherwise.%We say that a $\rp$-valued process is \cad if its sample paths lie in $D(\rp; E)$.
}

\subsection{Control setting} \label{CSetp}
A \emph{controlled system}  with state space $\rd$ and distribution $\xi\in \mathcal{P}(\rd)$ at an initial time $r\in \rp$ is described by the following elements.
 \begin{itemize}[leftmargin=0.29in]
 \item [(i)]   \emph{Action set.} 
 It is given by an open subset $A \,\subseteq \, \Mn \times \mathcal{M}_p \times \rd$ and denotes the set of possible actions available at each instant time and it is endowed with the corresponding product Borel $\sigma-$algebra denoted by $\mathcal{B}(A)$ and derived from the Borel $\sigma$-algebra on $M_{n\times n}(\rr) \times \mathcal{M}_p \times \rd$. \Mqo{or subspace topology?}

\item [(ii)]  \emph{Instantaneous underlying dynamics.} The local dynamics of the system  are  specified  by a family $\{ (L^\V{a}, D^\V{a}):\,  \V{a} \in A\}$ of operators  $L^\V{a} : D^\V{a} \subset M(\rd) \to M(\rd)$, such that for each $\V{a}= (\Msigma, \nu,\Mmu) \in A$ and for a fixed vector $\V{u} \in \rd$, the operator $L^\V{a}$ with domain $D^\V{a}$ is defined, for each $g \in C_c^2(\rd) \subset D^\V{a}$, by
\begin{align}\label{D:La}
   (L^\V{a} g)(\cdot) := (\V{u} +\Mmu)^T  \nabla g(\cdot) + \frac{1}{2} \tr( \Msigma^T \Hess g\, \Msigma)(\cdot)  +\int_{\Ro} \left ( g(\cdot+\V{y}) - g(\cdot) -  \V{y}^T \nabla g (\cdot ) \right ) \nu(\diff \V{y}).
   \end{align}
   \item [(iii)]  \emph{Control policies.} A control policy  is an $A$-valued process $\Malpha = (\Malpha_t)_{t\ge 0}$  determined by its probability law $\mathbf{P}_{\Malpha}$, i.e. $\mathbf{P}_{\Malpha}$ is a probability measure on $(D(\rp;A))$. The set of all such control processes is denoted by $\mathcal{U}$. \Mqo{Or $A$-valued predictable process $\Malpha$? or $A$-valued process with measurable sample paths?}
\item [(iv)]  \emph{Admissible control policies.} The set of admissible policies 
$ \Ap  \subseteq  \mathcal{U}$ with initial condition $(r,\xi) \in \rp\times \mathcal{P}(\rd)$ is defined  as the class of control  processes $\ap = \left(\ap_t\right)_{t\ge r} \in \mathcal{U}$ satisfying the following:
\begin{itemize}

\item [(\textbf{H1})] There exists a filtered, complete probability space\footnote{As usual, we will always assume that any probability space satisfies the \emph{usual conditions} \cite[Chapter I]{Protter}:  $(\Ft)_{t\in\rp}$ is a filtration of sub-$\sigma-$algebras of $\mathcal{F}$ such that  $\mathcal{F}_0$  contains all the $\mathbb{P}-$null sets in $\mathcal{F}$, and $(\Ft)_{t\in\rp}$ is right continuous, i.e. $\Ft = \bigcap_{s >t} \mathcal{F}_s =:\mathcal{F}_{t+}$.}  $(\Omega^{\alpha}, \mathcal{F}^{\alpha}, \Fa=(\Ft^{\alpha}), \mathbb{P}^{\Malpha})$   which supports both an $\Fa$-adapted \cad copy of $\ap$ (denoted again by $\ap$)  and  an  $\Fa$-adapted $\rd$-valued \cad process  
 $\xp= (\xp_t)_{t\ge r}$ with $\mathbb{P}^{\Malpha} \circ \left (\xp_s \right)^{-1} = \xi$ for all $0\le s\le r$,   such that the pair $(\xp,\ap)$ is unique in law and, further, for each $h \in C_c^2(\rd)$,  $\int_r^t \left |(\Las h\left (\xp_{s-} \right) \right |\diff s < \infty $, for $t> r$ $\as$, and   the process $M^{h,\ap} = \left (M_t^{h,\ap} \right )_{t\ge r}$ 
defined by  
\begin{equation}\label{D:Mphi}
M_t^{h,\ap} := h \left (\xp_t \right ) - \int_r^t  \left (\Las h \right )\left (\xp_{s-} \right )\diff s, \quad t\ge r,
\end{equation}
 is an $\FPa$-local martingale (i.e. an $\Fa$-adapted local martingale under  $\Pa$). 
 \item [(\textbf{H2})] If $\ap_s = ( \Msigma_s, \nu_s, \Mmu_s )$,  $s \ge r$, then for any $t\ge r$ 
 \begin{equation}\label{A3}
 \int_r^t Q_s^{p,\ap} \diff s  \,< \,\infty, \quad \text{where }\,\,
 Q_s^{p,\ap} := |\Mmu_s|  + ||\Msigma_s||^2 + \int_{\Ro} |\V{y}|^2 \vee |\V{y}|^p \nu_s(\diff \V{y}), \quad \as.
 \end{equation}
   \end{itemize}
   \end{itemize}
     Hereafter, any pair $\left (\xp, \ap \right)$ with $\ap \in \Ap$ will be referred to as an \emph{admissible pair}.

     \emph{Notation. For any $\V{x}\in \rd$,  the  initial conditions $(r,\delta_{\V{x}})$ will be denoted by $(r,\V{x})$, and thus $\mathcal{A}^p_{r,\delta_{\V{x}}} \equiv \Ag{r}{\V{x}}$ and $ (\V{X}^{\Malpha^{r,\delta_{\V{x}}}}, \Malpha^{r,\delta_{\V{x}}} )\equiv \left (\xprx, \aprx \right) $, whereas  if $r=0$, then $\mathcal{A}^p_{0,\delta_{\V{x}}} \equiv \Ax$ and $ (\V{X}^{\Malpha^{0,\delta_{\V{x}}}}, \Malpha^{0,\delta_{\V{x}}} )\equiv \left (\xpx, \apx \right) $. % When more clarity is needed, we will use notation $\Pa_{r,\xi}$ to indicate the probability measure $\Pa$ associated with the initial state $(r,\xi)$.
       }
     
\begin{definition}[\textbf{Control Martingale Problem}] 
 Let $\xi \in \mathcal{P}(\rd)$ and $r\in \rp$. The $\rd\times A$-valued admissible pair $\left (\xp, \ap \right )$  or, more precisely, the sextuplet $(\Omega^{\alpha}, \mathcal{F}^{\alpha}, \Fa, \Pa, \xp, \ap)$, with  $\Pa \circ \left (\xp_r \right)^{-1} = \xi$, is called a solution to the \emph{control martingale problem} for $ (\{L^a: a \in A\},\, \,C_c^2 (\rd),\,(r,\xi))$, whenever (\textbf{H1}) holds. Here  $\{L^a: a \in A\}$ is the family of operators defined in \eqref{D:La} and $\xi$ is the distribution of $\xp_r$. \end{definition}

\Mremark{The set of admissible policies $\Ap$ can then be thought of as the set of laws of the pair $(\xp,\ap)$ viewed as a random element with values in $D(\rp;A)\times D(\rp; \rd)$. }     
     
  \emph{Convention}. In order to extend the martingale problem \eqref{D:Mphi} to $\rp$ and given the assumption $\Pa \circ (\xp_s )^{-1} = \xi$ for all $0\le s\le r$, we set $\int_r^t \left | \left (\Las h \right )(\xp_{s-}) \right |\diff s = 0$ for $t\le r$.

%\begin{remark}\label{AT-admissible}
%When dealing with the finite horizon case on $[0,T]$, the class of admissible policies is define similarly as a restriction on the corresponding time interval. For such a case, the class of admissible policies will be denoted by $\mathcal{A}_{\V{x}} [0,T]$.  %  admissible policies $\ap$ corresponding to a control process $\xp$ with initial state $\V{x}$ -time point $(t,x)$, notation 
%\end{remark}

 \begin{remarks}
 \end{remarks}
 \begin{itemize}[leftmargin=0.75cm] %[leftmargin=*] 
 \item  [i)]Since the stochastic basis $(\Omega^{\alpha}, \mathcal{F}^{\alpha}, \Fa, \Pa)$ in (\textbf{H1}) fulfils  the usual conditions, by \cite[Theorem 2.9, Chapter II, p. 65]{RYor} %\cite[Theorem 4, Chapter IV]{DMeyerB} 
  any martingale defined on it admits  a   \cad  version which will be the one we always work with.
\McommentO{The assumption of the local Lebesgue integrability of the mapping $s \mapsto (\Las h)(\xp_{s-}) $ guarantees that \eqref{D:Mphi} is well-defined.}  
\item   [ii)]By \eqref{D:Mphi} in (\textbf{H1}), the process  $h(\xp)$ is a special semimartingale for each $h\in C_c^2 (\rd)$.  \Mcomment{Observe that no further restrictions, such as  assuming the Markov property, are imposed neither on the process $\xp$ nor on $\ap$.}
 \item  [iii)]It is feasible to consider a  larger set of test functions in  \eqref{D:Mphi}, for instance  the space $C^2 (\rd)$. However,  enlarging this set  reduces the size of the set of admissible policies $\Ap$.  \McommentO{One can  also show that the \emph{local} martingale problem \eqref{D:Mphi} is equivalent to the corresponding \emph{true} martingale problem due to the regularity (smoothness and compact support) of the test functions.}
 \item [iv)] Using the standard definition of (L\'evy) generating triplet with respect to the truncation function $h(\V{y}) = \V{y}1_{\{|\V{y}| \le 1\}}$ (see  \cite[Definition~8.2, p. 38]{Sato1999}), one can see that  the operator $L^{\V{a}}$ in \eqref{D:La} is the infinitesimal generator of a \levy process with triplet $(\Msigma^T\Msigma, \gamma, \nu)$ w.r.t  $h$, where $\gamma := (\V{u} + \Mmu)^T + \int_{|\V{y}| >1} |\V{y}|^2  \nu(\diff \V{y})$ (for details, see also \cite[Chapter 2, Section 7 and Theorem~31.5, p. 208]{Sato1999}).
\end{itemize}

\MremarkO{Note that, for each $h\in C_c^2(\rd)$, the martingale $M^{f,\ap}$ has the starting point  $M^{f,\ap}_r =h(\xi)$, for each $\ap \in \Ap$.}

 \MremarkO{In the context of \ito diffusions, solutions to martingale problems are equivalent to weak solutions of the corresponding SDE \cite[Karatzas, Proposition 4.6]{Karatzas} . Hence, a martingale approach is more convenient for mathematical reasons (because there are SDE which have  weak but no strong solutions such as the Tanaka equation). Furthermore, in a strong setting the probability space needs to be specified beforehand which, from a modelling point of view, makes  one to specify the explicit form of the white noise. Conditions for existence and uniqueness of  solutions to the martingale problem for \ito diffusions are well-known.}
 
\subsection{Concatenation property and examples of admissible policies}
 
Since any solution to a martingale  problem  can only determine the law of its solution process and due to the  \cad requirement for each admissible pair  $\left (\xp,\ap \right )$, without loss of generality we may (and we will) consider  solutions to the control martingale problem in the corresponding  canonical space.    More precisely, the sextuplet $(\Omega^{\alpha}, \mathcal{F}^{\alpha}, \Fa, \Pa,\xp,\ap)$ takes the form: $\Omega^{\alpha} := \Omega^{\alpha,1} \times \Omega^{\alpha,2}$ where $\Omega^{\alpha,1}:=D(\rp;\rd)$ and $\Omega^{\alpha,2}:= D(\rp;A)$ are the space of \cad functions on $\rp$ with values in $\rd$ and $A$, respectively, both of them endowed with the corresponding Skorohod topology; $ \Fa := \mathcal{F}^{\alpha,1} \otimes \mathcal{F}^{\alpha,2}$, where $\mathcal{F}^{\alpha,i} := \mathcal{B} (\Omega^{\alpha,i})$, $i=1,2$.  The solution process $(\xp_{\cdot},\ap_{
 \cdot})$ corresponds to the coordinate pair $(\xp_t(\omega), \ap_t(\omega)) := ( \omega^1(t), \omega^2(t))$, for each $\omega = (\omega^1, \omega^2) \in \Omega^{\Malpha}$, $t\ge 0$.  The filtration $\Fa = (\Ft^{\alpha})$ is defined by $\Ft^{\alpha,1}\otimes \Ft^{\alpha,2}$, where $\Ft^{\alpha,1} :=\sigma (\xp_s; s\le t)$ and $\Ft^{\alpha,2} :=\sigma (\ap_s; s\le t)$ are the natural filtration generated by the coordinate processes $\xp$ and $\ap$, respectively.   Finally, the  measure $\Pa$ is a probability measure on $(\Omega^{\alpha},\Fa)$. 

In the following result we prove a concatenation property for the admissible policies. This property  is fundamental for the validity of the dynamic programming principle as will be shown in the next section.
\begin{lemma}\label{L:Concatenation}
Let $(r,\xi) \in  \rp \times \mathcal{P}(\rd)$ and   let $\ap \in \Ap$. Denote by $\etat $  the law of the r.v. $\xp_t$, i.e. $\etat := \mathbb{P}^{\alpha} \circ (\xp_t)^{-1} \,\in\, \mathcal{P}(\rd)$. Then,  given any admissible policy $\bpa \in \mathcal{A}_{t,\etat}^p$, the concatenation of $\ap$ and $\bpa$ at time $t \ge r$, denoted by $\ap \oplus_t \bpa$ and defined by
\begin{equation}\label{D:oplus}
\ap \oplus_t \bpa\,:=\, 
\ap(s)1_{[r,t)}(s) \,+\, \bpa(s)1_{[t,\infty)}(s).
\end{equation}
 belongs to the set of admissible controls $\Ap$.
\end{lemma}
\begin{remark}
All the proofs of our results are given in the Appendices. See \ref{P:Concatenation} for the proof of Lemma \ref{L:Concatenation}. 
\end{remark}
\MremarkO{In the standard SDE setting (see, for example \cite{Pham}), each admissible process  is  assumed to be progressively measurable.  Therein, the measurability of admissible processes of the type defined in \eqref{D:oplus}  is justified by measurable selection theorems  (some references cited: Chapter 7, Bertsekas D. and S. Shreve (1978): Stochastic optimal control; the discrete-time case, Math. in Sci. and Eng., Academic Press). In our setting, the measurability follows by the \cad property of the concatenated processes.}
 \subsection*{Examples of admissible policies.} \label{CteP}
 \begin{itemize}[leftmargin=0.75cm]
 \item [i)] Since each measure $\nu \in \mathcal{M}_p$,  $p \ge 2$, has finite $p$th moments outside the unitary ball, for each \emph{constant policy} the corresponding controlled process $\xpx$ is a \levy martingale with drift component and, further, it also has finite $p$th moments  (see  \cite[Theorem~3.3.3, p. 163; p.133]{a}, \cite[Theorem 31.5, p. 208; p. 39]{Sato1999}).  \Mcomment{ In this case, it is also known that there exists a unique solution to the martingale problem for each initial condition $X_0 = x$, and further, such a solution is the unique strong solution to the corresponding SDE.}
 \item [ii)]  Since constant policies are admissible, Lemma \ref{L:Concatenation} implies that (by pasting appropriately) the set of step controls (or piecewise constant controls) $\mathcal{A}_{\V{x}}^{step}$ also belongs to $\Ax$. We say that a control $\apx$ belongs to $\mathcal{A}_{\V{x}}^{step}$  if there exists $N\in \mathbb{N}$ and  $\alpha_k \in A$, $k =1, \ldots, N$, such that  $\apx_t = \alpha_k$ on $[t_{k-1}, t_{k})$ for all $k =1, \ldots, N$, where $t_0 = 0$ and $t_N = \infty$. %Roughly, the way to construct such controls is by taking the admissible pairs $(X^{\alpha_k},\alpha_k)$ with constant policies $\alpha_k$, such that $\alpha_1 \in \mathcal{A}_{r,\V{x}}^p$ and $\alpha_k \in \mathcal{A}_{t_{k-1}, \xi_k}^p$ with the initial distribution  $X^{\alpha_k}_{t_{k-1}} = \xi_k$ where $\xi_k := \mathbb{P}^{\alpha_{k-1}} \circ X^{\alpha_{k-1}}_{t_{k-1}}$ for $k = 2,\ldots, N$. The process $\xpx$ controlled under the step policy $\apx$ is then constructed by pasting together the consecutive (coordinate) processes  $X^{\alpha_{k-1}}$ under $\mathbb{P}^{\alpha_{k-1}}$ and the process $X^{\alpha_{k}}$ under $\mathbb{P}^{\alpha_{k}}$ at time $t_{k-1}$. The corresponding probability measure $\mathbb{P} \in \mathcal{P} (D(\rp;\rd) \times D(\rp;A))$ is the law of the process $\apx$. 
 \item [iii)] The set $\Ax$ also contains Markov policies of the form $\apx_t  := ( \,\sigma(\V{X}_t^\V{x}), \,\nu(\V{X}_t^\V{x},\cdot),\,  \mu (\V{X}_t^\V{x}))$, $\V{x} \in \rd$, with continuous functions  $\mu: \rd \to \rd$, $\sigma: \rd \to \Mn$ and $\nu: \rd \to \mathcal{M}_p$ satisfying 
\[\sup_\V{x} \left(  |\mu(\V{x})| + ||\sigma(\V{x})||^2 + \int_{\Ro} |\V{y}|^2\vee |\V{y}|^p \nu(\V{x}, \diff \V{y})\right ) < \infty.\]
The associated controlled process  $\xpx$ is a  \emph{\levy-type process $\V{X}^\V{x}$  with bounded coefficients} (see \cite[Chapter 3]{a}, \cite{Schilling2010})  and whose infinitesimal generator  $L^* $ is defined,  on functions $g \in C_c^{\infty} (\rd)$,  by 
\begin{equation}\label{D:Ltype}
( L^* g )(\V{x}) := (\V{u} +\mu(\V{x}) )^T \nabla g(\V{x}) + \frac{1}{2} \tr(\sigma^T(\V{x}) \Hess g(\V{x})\, \sigma(\V{x})) + \int_{\Ro } (g(\V{x}+\V{y}) - g(\V{x})- \V{y}^T \nabla g(\V{x})\nu(\V{x},\diff \V{y}).
\end{equation}
\MremarkO{Note that, the boundedness of the coefficients $\mu$, $\sigma$ and $\nu$ ensure the validity of \eqref{A3}.  Solutions to martingale problems associated with  Markov generators are very well-known in the literature.}
\end{itemize}

%%%%%%%%%%%%
%%%%%%%%%%%%. UNDERLYING DYNAMICS
%%%%%%%%%%%%

\section{First results: underlying control dynamics}\label{S:Dynamics}

%We recall that the \emph{predictable semimartingale characteristics} of a semimartingale (relative to a truncation function $h$) are given by a triplet $(\V{B}(h), \V{C}, \xi)$ (see Appendix  for a brief remainder of these concepts or see, e.g., \cite[Chapter II]{JacodS1987}  for a detailed study). Roughly, $\V{B}(h)$ describes the drift, $\V{C}$ the diffusion component, and $\xi$ the jump component of the process.

We shall prove that, for each policy $\apx \in \Ax$ the corresponding controlled process is an $\rd$-valued  (special) semimartingale whose characteristics depend on $\apx$. 

\Mcomment{Recall that, given a filtered probability space $(\Omega,\mathcal{F}, \F:=(\mathcal{F}_t),\Prb)$, an $\F$-adapted \cad process $X$ is said to be a classical $(\F,\Prb)$-\emph {semimartingale} if it admits a (not necessarily unique) decomposition $X = X_0 + M + A$ $\Prb$-a.s., where $X_0$ is finite-valued and $\mathcal{F}_0-$measurable, $M$ is an $(\F,\Prb)$-local martingale with  $M_0 =0=A_0$, and $A$ is an $\F$-adapted \cad processes with paths of (locally) finite variation $\Prb$-a.s.   If additionally $A$ is predictable,  then $X$ is called a \emph{special semimartingale} \cite[Definition 4.21, I.4c, p. 43]{JacodS1987}. For \emph{special semimartingales} its decomposition is unique (up to indistinguishability) and is known as the \emph{canonical decomposition of} $X$ (\cite[Definition 4.22, I.4c, p. 43]{JacodS1987}). An adapted stochastic process is said to be \emph{predictable} (resp. \emph{optional}) if it is measurable with respect to the \emph{predictable} $\sigma$-algebra $\Prb$ (resp. \emph{optional} $\sigma$-algebra $\mathcal{O}$), i.e. the $\sigma$-algebra on $\Omega \times \rp$ generated by all the left-continuous (resp. c\'adl\'ag) adapted processes $Z$ considered as mappings $(\omega,t) \mapsto  Z_t (\omega)$ on $\Omega \times \rp$. A function $W : \Omega \times\rp \times \rd \to \rr$ is called \emph{predictable} (resp. \emph{optional}) if it is measurable with respect to the $\sigma$-algebra of \emph{predictable} sets  (resp. \emph{optional} sets) in $\Omega \times \rp \times \rd$, given by $\mathcal{P} \otimes \mathcal{B} (\rd)$ (resp. $\mathcal{O} \otimes \mathcal{B} (\rd)$).} 

Another explicit representation for semimartingales can be given in terms of  the so-called \emph{semimartingale characteristics relative to a truncation  function} $h$. Such a decomposition is known as the \emph{canonical representation relative to $h$}. Let us  briefly recall these concepts (see \cite[Chapter II]{JacodS1987}  for a detailed study). Let $X=(X^i)_{1\le i \le n}$ be an $n$-dimensional semimartingale. Fix a truncation function $h : \rd\to \rd$ (i.e.,  a bounded measurable function with compact support such that $h(x) = x$ in a neighbourhood of the origin \cite[Definition 2.3]{JacodS1987}). Define 
\[
\begin{array}{rcl}
\tilde{X}(h)_{\cdot} &:=&\sum_{s\le \cdot} [ \Delta X_s + h(\Delta X_s)] \\
X(h)_{\cdot} & :=& X_{\cdot} -\tilde{X}(h)_{\cdot},
\end{array}
\]
then $\tilde{X}(h)$ is an $n$-dimensional finite variation process and $X(h)$ is a process with uniformly bounded jumps (and thus is a special semimartingale) and differs from $X$ by a finite variation process.   

Now, recall  also that, for a given state space $S$, a \emph{random measure} $\mu$ on $\rp \times S$ is a family  $\{\mu(\omega)\,:\, \omega \in \Omega\}$ of measures $\mu (\omega)$ on $(\rp \times S, \mathcal{B}(\rp) \otimes \mathcal{B} (S))$ satisfying $\mu (\omega; \{0\} \times S) = 0$ for each $\omega \in \Omega$. For any random measure $\mu$ and any optional function $W$ on $\Omega \times \rp \times S$, we denote by $W \ast \mu$ the integral process
\begin{equation}\label{Def:Int-ast}
\int_0^{\cdot} \int_S W (\omega;t,y) \mu (\omega ; \diff t, \diff y).
\end{equation}

Given  an adapted \cad $\rd$-valued process $X$, the measure associated to  its jumps is defined as the integer-valued \emph{random measure} $\eta^X$ on $\mathbb{R}^+ \times \rd$ given by  
\[ \eta^X (\omega; \diff t, \diff y) := \sum_{s} \V{1}_{\{\Delta X_s(\omega) \neq 0\}} \delta_{(s, \Delta X_s(\omega))} (\diff t, \diff y),\]
see \cite[II.1b, Definition 1.13, Proposition 1.16, pp. 68-69]{JacodS1987}. 
Thus, for each path $\omega$, $\eta^X(\omega;[0,t]\times D)$ gives the number of jumps whose sizes fall in the measurable set $D \subset \rd$, during the time interval $[0,t]$. Since $X$ is \cad, the random measure $\eta^X (\omega; \diff t, \diff y)$ takes only finite values for any Borel subset $D$ of $\rd$ bounded away from zero.  
Moreover, by \cite[Theorem 1.8, Chapter II.1a, p. 66]{JacodS1987}, there exists an unique (up to indistinguishability) \emph{predictable} random measure   $\eta$ on $\mathbb{R}^+ \times \mathbb{R}^n$ for which, in particular, the integral process $W\ast (\eta^X -\eta)$ given by 
\[ W \ast (\eta^X -\eta)_t(\omega) :=\int_0^t \int_{\rd} W(\omega,s,x) (\eta^X  - \eta)(\omega;\diff s,\diff x))\]
is a local martingale for each predictable function $W$ on $\Omega \times \rp \times \rd$ for which $|W| \ast (\eta^X -\eta)_t (\omega)$ is finite. The random measure $\eta(\omega; \diff s, \diff y)$ is called the \emph{predictable compensator}, or \emph{dual predictable projection} of the jump measure $\eta^X$. 
 \MremarkO{ Let $\V{X}$ be an $n-$dimensional semimartingale, i.e. an $\rd$-valued process, each of whose coordinates $X_i$ is a real-valued semimartingale. Notation $\V{X}^c$ stands for the continuous part of $\V{X}$ with components $X^{i,c}$, $i\in \{1,\ldots, n\}$, and the $n\times n$ matrix-valued processes $\left [ \V{X},\,\V{X}\right ] = ( [X^i,\,X^j])_{1\le i,j\le n}$ and $\lrangle{\V{X}}{\V{X}} = (\lrangle{X^i}{X^j})_{1\le i,j\le n}$ denote the \textit{quadratic variation process} and the \textit{predictable quadratic variation process} (whenever it exists) with entries $[X^i,\, X^j]$ and  $\lrangle{X^i}{X^j}$, for all $i,j \in \{1,\ldots, n\}$.  We will also use notation  $[\,X\,] :=[X,\, X]$ and $\langle \, X\, \rangle := \lrangle{X}{X}$.    We recall that,  for each $i,j \in \{1,\ldots, n\}$, if $[X^i,\, X^j]$ is locally integrable, then the process $ \lrangle{X^i}{X^j}$ is the unique (up to null sets) adapted process such that the paths $ t \mapsto \lrangle{X^i}{X^j}_t$ have finite variation over compact intervals, $\lrangle{X^i}{X^j}_0 = 0$ and   $ X^i X^j - \lrangle{X^i}{X^j} $ is a local martingale  \cite[IV.26]{RogersW}.}

\begin{definition}[\textbf{Semimartingale characteristics}]
The \emph{characteristics of the semimartingale} $X$  w.r.t  $\mathbb{P}$ and \emph{relative to} $h$, are given by a triplet $(\V{B}^h, \V{C}, \eta)$, where  \emph{(i)}  $\V{B}^h = (B^i)_{i \le n}$ is the $\rd-$valued predictable process of finite variation  in the canonical representation of the special semimartingale $X(h)$; \emph{(ii)}    $\V{C} = (C^{ij})_{1\le i,j\le n}$, $C^{ij} := \lrangle{X^{i,c}}{X^{j,c}}$, is the  $M_{n\times n}(\rr)$-valued continuous process with entries $<X^{i,c},\,X^{j,c} >$, $1\le i,j\le n $, where $X^c$ denotes the continuous martingale part of $X$ (with components $X^{i,c}$); and \emph{(iii)} $\eta =\eta(\omega; \diff t, \diff y)$ is  a random measure on $\rp \times \Ro$ and is the predictable compensator of the integer-valued random measure $\eta^X =\eta^X(\omega;\diff t, \diff x)$ associated with the jumps of $X$. 
\end{definition}
\Mremark{Notice that the characteristics are unique up to indistinguishability. Furthermore, as notation indicates, $\V{B}^h$ depends on the chosen truncation function $h$, whereas $\V{C}$ and $\eta$ are independent of $h$. }
\begin{remark}\label{R:Notation-x}
To simplify notation we shall fix the initial state $(0, \delta_{\V{x}}) \in \rd \times \mathcal{P} (\rd)$ and work with the class of controls $\Ax$. Our results can easily be extended to the general case  $\Ap$ with $(r,\xi) \in \rp \times \mathcal{P} (\rd)$.
\end{remark}

We now state the relationship between each policy $\apx \in \Ax$ and the semimartingale characteristics of its  corresponding process $\xpx$.

\begin{proposition}\label{X-charact}
Let $h:\rd \to \rd$ be the function  $h(\V{y}) := \V{y} \V{1}_{\{ |\V{y}| \le 1\}}$,  $\V{y}\in \rd$. For each admissible policy $\apx = (\Msigma, \nu,\Mmu) \in \Ax$, $p \ge 2$,  the process  $\xpx$ is an $\FPa$-semimartingale with the predictable semimartingale characteristics (relative to $h$) $(\V{B}^h, \V{C}, \eta)$  given by
\begin{itemize}[leftmargin=0.75cm]
\item [\emph{i)}]  $\V{B}_{\cdot}^h$ is the process  $\V{B}^h(t,\omega) = \int_0^t(\V{u} + \Mmu_s (\omega)) \diff s + \int_0^t \diff s \int_{\Ro} (\V{y}- h(\V{y})) \nu_s(\omega;\diff \V{y})$, for $t\ge 0$.
\item [\emph{ii)}] $\V{C}_{\cdot} = (C_{ij}(\cdot))_{1\le i,j\le n}$, where $C_{ij} (t,\omega)= \int_0^t \diff s  ( \Msigma_s^T \Msigma_s)(\omega)_{ij}$, for $t \ge 0$. 
\item [\emph{iii)}] $\eta$ is the random measure on $\rp \times \Ro$ given by $\eta (\omega;\diff s, \diff \V{y}) := \diff s\, \otimes \nu_s (\omega;\diff \V{y})$, $\omega \in \Omega^{\alpha}$.
\end{itemize}
\McommentO{Here, $ ( \Msigma_s^T \Msigma_s)(\omega)_{ij} = \sum_{k=1}^n \left(\sigma_{ik}\sigma_{kj}\right)(s,\omega)$ for each $(s,\omega)\in \rp \times \Omega^{\alpha}$.}
\end{proposition} 
\Mremark{$\mathbf{B}^h$ and $\mathbf{C}$ are the local drift coefficient and the local covariance matrix, respectively. We also recall that the predictable characteristics do not characterise the law of the process, however they provide useful information related to its jumps as we will see in Proposition \ref{qth-moments}.}
\Mqo{For the  integrals $\int_0^t \mu_s \diff s$, $\int_0^t \sigma_s \diff s$ and $\int_0^t \diff s\int_{\rd} y \nu_s( \diff y)$ to be well-defined in the Lebesgue-Stieltjes sense we need, e.g.,  $\mu$ and $\sigma$  to be predictable processes; and $\nu_{\cdot}(A)$ to be  a predictable process for all $A \in \mathcal{B} (\rd)$. OR these processes to be progressively measurable and locally bounded.}
\begin{remarks}
 \end{remarks}
\begin{itemize}[leftmargin=0.75cm]
\item [i)] The process $\V{C}$ takes values on the set of all positive semidefinite symmetric $n\times n$-matrices. By definition, $\V{C}$ is the unique (up to null sets) adapted continuous process, starting at $\V{C}_0 = 0$, with paths $t \mapsto C_{ij} (t)$ having finite variation over compact intervals and such that $ X^{i,c}_t\,X^{j,c}_t- C_{ij} (t)$ is a local martingale (see \cite[IV.26]{RogersW}).
\item [ii)]  By Proposition \ref{X-charact}, for each admissible $\apx$, the process  $\xpx$ is an \emph{\ito semimartingale},\McommentO{\footnote{\textcolor{red}{Examples of \ito semimartingales are the \ito processes:
 \[X_t = X_0 + \int_0^t b_s \diff s + \int_0^t \sigma_s \diff W_s,\]
where $X_0$  is an $\mathcal{F}_0$-measurable r.v. and $b$ and $\sigma$ are progressively measurable processes valued in $\rd$ and $M_{n\times n} (\rr)$ respectively, and such that
 \[\int_0^t |b_s| \diff s + \int_0^t |\sigma_s|^2 \diff s < \infty,  \text{ a.s. for all } t\in \rp.\]}}} in the sense that its characteristics are absolutely continuous with respect to the Lebesgue measure \cite[Definition 1.16, p. 45]{JacodPress2014}. Hence, the control process $\apx$ determines the \emph{differential semimartingale characteristics} of $\xpx$. 
\end{itemize}
 \MremarkO{\levy processes with first moments are special semimartingales with \emph{deterministic}  characteristics given by $\V{B}_{t}^h := t \left (\mu +  \int_{|y|>1}\right ) \nu(\diff y)$, $\V{C}_{t} = t  \Msigma^T \Msigma$ and $\eta (\diff t, \diff \V{y}) = \diff t\otimes \nu(\diff y)$, where $\mu \in \rd$, $\Msigma \in M_{n\times n} (\rr)$ and $\nu$ is a given \levy measure on $\Ro$. Another example is a \emph{diffusion with jumps}, that is a solution to the SDE
 \begin{equation}
 \diff Y_t = b(t,X_t)\diff t + \sigma (t, X_t)\diff B_t + \gamma(t, X_{t-}) \left ( N(\diff t, \diff y) - \lambda(\diff t, \diff y) \right ), \quad Y_0 = y,
 \end{equation}
 where $B$ is a $d$-dimensional Brownian motion and $N$ is a Poisson random measure on $\rp \times \Ro$ with intensity measure $\lambda(\diff t, \diff y) :=\diff t \otimes \nu(\diff y)$, where $|y|^2 \wedge |y| \ast \nu \in A_{loc}$. Its semimartingale  characteristics take the form $\V{B}_{t}^h := t \left (\mu +  \int_{|y|>1}\right ) \nu(\diff y)$, $\V{C}_{t} = t  \Msigma^T \Msigma$ and $\eta (\diff s, \diff \V{y}) = \int_0^t \gamma(s,y) \nu(\diff y)$, where $\mu \in \rd$, $\Msigma \in M_{n\times n} (\rr)$ and $\nu$ is a given \levy measure on $\Ro$. }
 \Mremark{If $X$ is a \levy process with generating triplet $(b,c,F)$, then its characteristics in the semimartingale sense are the not random functions:
 \begin{equation} \textbf{B}_t (\omega) := bt,\quad\quad
 \textbf{C}_t (\omega) := ct,\quad \quad
\eta (\omega, \diff t, \diff x) := \diff t \otimes F (\diff x).
 \end{equation}
 }
 
Proposition \ref{X-charact} identifies $\diff s\, \otimes \nu_s (\omega,\diff \V{y})$ as the compensator of the random  measure  associated to the jumps of  $\xpx$, hereafter denoted by $\eta^{X,\apx}$.  This plays an important role in obtaining estimates of the running maximum of $|\xpx|^q$ for $q >2$, which are fundamental to  deal with value functions  of super-quadratic order. We now state conditions on the control process  $\apx \in \Ax$ which guarantee the finite expectation of the running maximum of $\left | \xpx \right |^q$, for  $q \ge 2$.

\emph{Notation.  $\Ex$ will denote the expectation with respect to the measure induced by the process $\xpx$ started at $\V{x}$. }

\begin{proposition}\label{qth-moments} 
For each admissible pair $(\xpx, \apx)$, $\apx \in \Ax$, $p \ge 2$, the following holds.
\begin{itemize}[leftmargin=0.75cm]
\item [\emph{i)}] The semimartingale $\xpx$ admits the \emph{canonical representation} 
\begin{equation}\label{X-decomp}
\xpx_t = \V{x} + \Xc_t + \int_0^t (\V{u} + \Mmu_s) \diff s + \int_0^t \int_{\Ro} \V{y} \left (\nx -\eta \right )(\diff s, \diff \V{y}),
\end{equation}
where $\Xc $ is the continuous martingale part of $\xpx$, $\nx$ is the jump measure associated with $\xpx$  and $\eta$ is its predictable compensator.
\item [\emph{ii)}]If for each $t\ge 0$ the process $\int_0^t ||\Msigma_s||^2 \diff s \in L^{q/2} (\mathbb{P}^{\apx})$ for $q \in [2,p]$,  then there exists a constant $C>0$ such that 
\begin{equation}\label{Xc-estim}
\Ex \left (  \sup_{0 \le s \le t} \left | \Xc_s \right |^q\right ) \le \,C\,\Ex \left (  \int_0^t ||\Msigma_s||^2 \diff s\right )^{q/2} < \infty, \quad t\ge 0.
\end{equation}
\item [\emph{iii)}] Let $\Xd$ be the discontinuous martingale part of the process $\xpx$ and let $q \in [2,p]$. Suppose that for each $t \ge 0$ 
\begin{equation}\label{G-H}
 \V{G}_t^{\apx} := \int_0^t \diff s\int_{\Ro}   |\V{y}|^2   \nu_s (\diff \V{y}) \in L^{q/2} (\mathbb{P}^{\apx}) \quad \text{and}\quad   \V{H}_t^{\apx} := \int_0^t \diff s\int_{|y| >1}   |\V{y}|^q   \nu_s (\diff \V{y}) \in L^{1} (\mathbb{P}^{\apx}) 
\end{equation}
Then, there exists a constant $C_1 > 0$ such that
\begin{equation}\label{Xd-estim}
 \Ex \left ( \sup_{0\le s \le t} |\Xd_s|^q\right ) \,\le \, C_1  \,\Ex\left [ \left(\V{G}_t^{\apx}\right )^{q/2} \right ] \,+ \,C_1\,\Ex ( \V{H}_t^{\apx})\, < +\infty, \quad t \ge 0.
 \end{equation} 
\end{itemize}
\end{proposition}
\begin{remarks}\label{R:qthMoments}
\end{remarks}
\begin{itemize}[leftmargin=0.75cm]
 \item [i)] In general, the \emph{canonical representation}  of a (not necessarily special) semimartingale depends on function $h$ associated with its characteristics (see \cite[II.2c, Theorem 2.34, p.84]{JacodS1987}). However, for the special semimartingale  $\xpx$ its representation \eqref{X-decomp} is independent of $h$. This is because of the finite first moments of the measures $\nu$ outside $B_1$, which allows one to compensate the \emph{big} jumps. In fact,  the representation  \eqref{X-decomp} coincides with the \emph{canonical decomposition} of $\xpx$ (by \cite[II.2c, Corollary 2.38, p. 85]{JacodS1987}). Indeed, $\xpx$ is an $\FPa$-\emph{special} $\rd$-valued semimartingale with compensator   \[\V{A}^{\apx}_{\,\cdot \,}:= \int_0^{\,\cdot\,} (\V{u} + \Mmu_s) \diff s.\]
\Mcomment{Also observe that, under constant policies we recover the corresponding \levy-\ito decomposition\McommentO{\footnote{\textcolor{red}{\levy-\ito decomposition of $X$:
 \begin{equation}
 X_t = bt + \sigma W_t + y 1_{\{ |y| \le 1\}}\ast (\eta^X -\eta)_t + y 1_{\{ |y| > 1\}}\ast \eta_t^X,
 \end{equation}
 which states that $X$ can be decomposed as a sum of independent terms: a pure drift term, a continuous martingale, a purely discontinuous martingale of compensated "small" jumps and the sum of "big" jumps. In this case, $\eta^X$ is a Poisson random measure with intensity $\eta (\diff t, \diff y) = \diff t \otimes \nu (\diff y)$. In this case, $\eta^X$ is random, but its compensator is not.}
 }} for \levy processes.}
  \item [ii)] Proposition \ref{qth-moments} establishes a relationship between the existence of $q$th-moments of $\xpx$ and the finiteness of  $\int_{|y| > 1} |y|^q\nu(\diff y) < +\infty$ for each measure  $\nu$. In particular, by considering constant policies we recover the well-known results for the existence of moments of \levy processes \cite[Theorem 25.3]{Sato1999}:   a \levy process with generating triplet $(b,c,F)$ has finite $q$th-moments whenever $\int_{|y| > 1} |y|^qF(\diff y) < +\infty$. That is, the finiteness of the moments depends on the tail behaviour of $F$ (the big jumps). % It is not difficult to see that the previous condition guarantees that the corresponding process $\int_0^t \hat{Q}_s^{p,\apx}(\omega) \in L_{loc}^{p/2} (\mathbb{P}^{\alpha})$. 
  \item [iii)]  Since $\xp$ is a semimartingale, the  Meyer-\ito formula ensures that, for any $f\in C^2 (\rd)$,  $f(\xp)$ is also a semimartingale  whose decomposition can be given explicitly \cite[Chapter II, Theorem 33, p. 81]{Protter}. In fact, convex functions are the most general functions that take semimartingales into semimartingales \cite[Chapter IV, Theorem 67, p. 215]{Protter}. % \item A semimartingale $\xp$ is said to belong to $\mathcal{H}^1$ if its decomposition $\xp_t = \xp_0 + M_t + A_t$ are such that $\xp_0$, $[M,M]_{\infty}^{1/2}$ and $\int_0^{\infty} |\diff A_s|$ are all in $L^1$. $\xp$ is locally in $\mathcal{H}^1$ if there exists a localising sequence ${T^n}$  such that $\xp{t\wedge T^n}1_{\{T_n > 0\}} \in \mathcal{H}^1$ for each $n$.  We observe that if $\xp$ is a continuous semimartingale then $\xp$  automatically belongs to $\mathcal{H}^1_{loc}$.  
 \end{itemize}

\MremarkO{In our setting the compensator of the random measure $\eta^X$ is another \emph{random} measure: $\diff s\times \nu_s (\omega;\diff y)$. For a \levy process $X$ with \levy measure $\nu$ the compensator of $\eta^X$ (usually denoted by $N(\diff t, \diff y)$) is a deterministic function $\diff s\otimes \nu (\diff y)$. In particular, when $\nu(A)< \infty$, then $N([0,t]\times A)$ is a Poisson process with parameter $\nu (A)$.}

  \Mremark{ We can also prove that    the \emph{predictable quadratic variation} of $\xpx=(X^1,\ldots, X^n)^T$ is given by
\begin{equation}\label{P:QV}
   \langle X^i ,X^j \rangle_{t} = \int_0^t \diff s \left (  ( \Msigma_s^T \Msigma_s)_{ij} + \int_{ \Ro}  y_i y_j \nu_s (\diff \V{y})\right), \quad i,j \in \{1,\ldots, n\};%,
  \end{equation}
\McommentO{ Hence,  the quadratic variation-covariation process $ \langle X^i ,X^j \rangle$ is 
absolutely continuous with respect to the Lebesgue measure }
and, thus, 
  \begin{align} \nonumber
  \tr \lrangle{\xpx}{ \xpx}_t %&= \int_0^t  \diff s \left (\sum_{i=1}^n (\Msigma_t^T \Msigma_t)_{ii} + \int_{\Ro} y_i^2\nu_s (\diff \V{y}) \right ) \\ \nonumber
  %&= \int_0^t  \diff s \left (  \sum_{i=1}^n \sum_{k=1}^n \sigma_{ik}^2(t) + \int_{\Ro} |\V{y}|^2\nu_s (\diff \V{y}) \right ) \\
  &= \int_0^t  \diff s \left (  ||\Msigma_t||^2 + \int_{\Ro} |\V{y}|^2\nu_s (\diff \V{y}) \right ).
  \end{align}
}

%% file: VF7-VerifTheoremInfinite.tex
Let us now describe the cost structure of the control problem we are interested in.
\subsection{Infinite horizon case}
Consider the infinite horizon stochastic control problem with  \textit{running cost} (or \emph{instantaneous payoff}) given by a measurable function $f: \rd \times A \to \mathbb{R}^+$,   and \textit{payoff function} $J$ defined, for each admissible pair $(\xpx,\apx)$, by
  \begin{equation}\label{J}
  J(\xpx,\apx):=  \Ex \left [   \int_0^{\infty} e^{- \gamma_t^{\apx}} f (\xpx_t,\,\apx_t)\diff t\right ],
     \end{equation}
 where   the discount process $\gamma^{\apx} = (\gamma_t^{\apx})$  is given by $\gamma_{t}^{\apx} := \int_0^t q (\xpx_s, \, \apx_s) \diff s$, for some given bounded, measurable and nonnegative function $q : \rd \times A \to \rp$.  % and such that $\int_0^{\infty} q \diff s = \infty$. 

The aim is to minimize \eqref{J} over all admissible policies $\apx \in \Ax$. In order to solve this problem, the \emph{dynamic programming approach} focuses on studying   the associated family of optimisation problems indexed by the initial data $\V{x}$. The solution of the control problem consists then in finding both
 \begin{itemize}[leftmargin=0.75cm]
 \item [(a)] the   \textit{value function}  (or  \textit{optimal payoff function})   
  $V: \rd \to \mathbb{R}^+$, given by    
\begin{equation}\label{V}
  V(\V{x}) := \inf_{\{(\xpx,\apx)\,: \, \apx \in \Ax\}} J(\xpx,\apx),\quad \V{x} \in \rd,\end{equation}
 \item [(b)]  an  \textit{optimal policy} for each  initial state $\V{x}\in \rd $ (whenever it exists); that is a family  $\{\hat{\Malpha}^\V{x}  \in \Ax : \V{x} \in \rd\}$  such that  the corresponding admissible pairs $(\V{X}^{\hat{\Malpha}^\V{x}},\hat{\Malpha}^\V{x})$ satisfy  \begin{equation}\label{D:alphaO}
    V(\V{x})  = J (\V{X}^{\hat{\Malpha}^\V{x}},\hat{\Malpha}^\V{x}),\quad \text{ for each } \V{x} \in \rd.
    \end{equation}
\end{itemize}

\Mqo{Can we reduce the control problem to a minimisation of the payoff functional but under the most restrictive set of Markov controls? i.e. when can we say that $V_{Markov} (x) = V(x)$? See the case for controlled diffusions in \cite[Theorem 11.2.3, p. 174]{OKSDE} }

\emph{Notation. For convenience we will write $J^{\alpha}(r,\V{x})  \equiv   J (\V{X}^{\Malpha^{r,\V{x}}},\Malpha^{r,\V{x}})$ and $\Jx =J^{\alpha}(0,\V{x})$. We will also write $(\V{X}^{\Malpha},\Malpha)$ instead of $(\V{X}^{\Malpha^{r,\V{x}}},\Malpha^{r,\V{x}})$ whenever these processes are inside the operator $\mathbb{E}^{\alpha}_{r,\V{x}}$. }

\Mremark{For many applications the cost functional \eqref{J} is considered a suitable model to analyse the long-time behaviour of a controlled system. The corresponding finite horizon case will be treated in   Section \ref{S:FiniteCase}.}

 It is well-known that in the previous form the infinite horizon formulation is time-invariant: it does not vary over time as long as the initial state is the same. This is stated in the following lemma.
 \begin{lemma}\label{L:TimeHom}
 Define
 \begin{align}
 J^{\alpha}(r,\V{x}) :&=  \mathbb{E}_{r,\V{x}}^{\alpha} \left [ \int_r^{\infty} e^{- \int_t^s q (\xpo_l, \, \apo_l) \diff l} f(\xpo_s, \apo_s) \diff s \right ],\\
 v(r,\V{x}) :&= \inf_{\apg{r}{x} \in \Ag{r}{x}} J^{\alpha}(r,\V{x}).
 \label{E:v}
 \end{align}
 Then   $ v(r,\V{x}) = v(0,\V{x}) = V(\V{x})$, for all $(r,\V{x}) \in \rp \times \rd$.
\end{lemma}
%Let $\mathcal{A}_{\V{x}}^F$ denote the subset of admissible controls with finite payoff, i.e., $\Jx < +\infty$. 

We now prove the validity of the \emph{dynamic programming principle} (DPP) for our control setting. To do this, we will need the following definitions.
\begin{definition}[\textbf{$\epsilon$-optimal controls}]\label{D:eOptimal}
Let $\epsilon > 0$ and  $\V{x}\in \rd$. An admissible pair  $\left (\xpx,\apx\right)$ with $\apx \in \Ax$   is said to be an \emph{$\epsilon$-optimal control}  if
\begin{equation}
J^{\alpha} (\V{x}) \,< \,V(\V{x}) + \epsilon.
\end{equation}
An admissible pair $\left ( \V{X}^{\Malpha^{\xi}}, \Malpha^{\xi}\right )$ with $\Malpha^{\xi} \in \mathcal{A}^p_{\xi}$  and $\xi \in \mathcal{P} (\rd)$ is said to be an \emph{$\epsilon$-optimal control} if
\begin{equation}
\int_{\rd} J \left (   \xpx, \apx \right ) \xi (\diff \V{x}) \,\, \le \,\, \int_{\rd} V(\V{x})\xi (\diff \V{x})  + \epsilon.
\end{equation}
\end{definition}
\begin{remark}
Notice that, by definition of the value function $V$, the existence of $\epsilon$-optimal controls is always guaranteed.
\end{remark}
\begin{lemma}\label{L:DPP}(\textbf{DPP})
Let $V$ be a continuous  function solving \eqref{V}-\eqref{D:alphaO}. Then, for each $\V{x} \in  \rd$, the function $V$ satisfies 
\begin{equation}\label{E:DPP}
V(\V{x}) = \inf_{\apx \in \mathcal{A}^p_{\V{x}}} \mathbb{E}^{\alpha}_{\V{x}} \left [   \int_0^t e^{-\gamma_{s}^{\apo} } f \left (\xpo_s, \apo_s \right ) \diff s + V \left ( \xpo_t \right ) e^{-\gamma_{t}^{\apo}}\right ].
\end{equation}
\end{lemma}

%\begin{remark}
%To ease notation we will write
%For notational clarity, hereafter we will use notation
%If there is no risk of confusion, we will use the notation 
%\[ J^{\alpha} (t,\V{x}) :=   \Ext \left [   \int_0^{\infty} e^{- \gamma_t^{\apo}} f (\xpo_t,\,\apo_t)\diff t\right ] \equiv J (\xpg{t}{x},\apg{t}{x}),\] for any admissible pair $(\xpg{t}{x},\apg{t}{x})$ and $(t,\V{x}) \in \rp\times\rd$. As usual, the subscript $t$ is omitted whenever $t=0$ and thus $J^{\alpha} (t,\V{x}) \equiv J^{\alpha} (\V{x})$
%\end{remark}
\Mqo{Continuity of $V$?}
\MremarkO{The continuity assumption on the value function $V$ guarantees that $V$ is measurable and thus the right-hand side in the DPP is well-defined. Measurability of the value function is usually deal with by means of measurable section theorems, see for example  \cite{DMeyerA}.}

We want to establish conditions which help us to determine if a given function $\phi :\rd \to \rp$ is the value function of the control problem \eqref{J}-\eqref{D:alphaO}. The following lemma provides a \textit{necessary} condition satisfied by the optimal cost function $V$. This result is a consequence of the DPP.
\begin{lemma}\label{P0-S}
Let  $V: \rd \to \mathbb{R}^+$ be the value function of the  control problem  \eqref{J}-\eqref{D:alphaO}. For each admissible policy $\apx \in \Ax$, $p\ge 2$, such that $\Jx< +\infty$,  the process   $S^{V,\apx}$ defined by
\begin{equation}
S_t^{V,\apx} := \int_0^t e^{-\gamma_s^{\apx}} f(\xpx_s, \apx_s) \diff s + e^{-\gamma_{t}^{\apx}} V(\xpx_t),\quad t\ge 0
\end{equation}
is a positive $\Pa$-submartingale. Furthermore, if $\apx$ is optimal, then $S^{V,\apx}$ is a (true) $\Pa$-martingale.
\end{lemma}

\begin{remarks}
\begin{itemize}[leftmargin=0.75cm]\item []
\item [i)]The process  $ S^{V,\apx}$ is called  \textit{the Bellman process}. At each time $t$, it can be thought of as the minimum expected total cost, given the evolution of the process  up to time $t$ under the policy $\apx$ and then changing to an optimal control afterwards. The submartingale property tells us then that by using an arbitrary control $\apx$ for a longer time, the expected cost keeps on increasing and such an  increase is zero whenever the policy is optimal.
\item [ii)] Lemma \ref{P0-S} provides a first necessary (but not sufficient) condition to characterise the value function.  As in the standard diffusion case, some sufficient conditions can be given in terms of the so-called \textit{transversality condition} (see  \eqref{E:TC} below). 
\item [iii)] \Mcomment{Notice that policies with infinite payoff were  not excluded from the definition of admissible policies. However, when solving the optimisation problem we will only focus on policies with finite payoff, as otherwise it is clear that such a policy cannot be an optimal one.}
\end{itemize}
\end{remarks}

%The previous implies that, for the tractability of the control problem, we need further   constraints on the set of admissible controls:
Let us give some further assumptions for a given candidate value function $\phi \in C^2(\rd)$:
\begin{itemize}
\item [(\textbf{SC})] For any admissible policy $\apx \in \Ax$  with $\Jx <+\infty$, $S_t^{\phi,\apx}$ is a positive submartingale. 
\item [(\textbf{MC})] There exists a family of  admissible pairs $\{\,(\V{X}^{\hat{\Malpha}^{\V{x}}},\hat{\Malpha}^{\V{x}})\,: \V{x} \in \rd\,\}$ such that,  for each $\hat{\Malpha}^{\V{x}}$,  $S_t^{\phi,\hat{\Malpha}^{\V{x}}}$ is a martingale. 
\item [(\textbf{TC})]  Each admissible policy $\apx$ is such that either  $\Jx = \infty$ or
 \begin{align}
 \liminf_{t\to \infty}\Ex \left [ e^{-\gamma_{t}^{\apo}} \phi \left (\xpo_t\right) \right ] = 0,\label{E:TC}
 \end{align}
\item [(\textbf{nC})] For each $\V{x} \in \rd$, there exists a sequence $\{\apx_n\}_{n\in \mathbb{N}} \subset \Ax$ such that $J^{\alpha_n} (\V{x}) < \phi (\V{x}) + \frac{1}{n}$.
\end{itemize}

\begin{remarks}
\begin{itemize}[leftmargin=0.75cm] \item []
\item [i)]The submartingale and martingale conditions (\textbf{SC})-(\textbf{MC}) are the core feature in the standard \textit{martingale approach} for stochastic control problems (see \cite{Davis1979}). \Mcomment{By Lemma \eqref{P0-S}, these conditions are only necessary conditions for optimality.} 
\item [ii)] The \textit{transversality condition} \eqref{E:TC} \McommentO{which is a sufficient condition for optimality,}  implicitly prescribes some kind of growth condition on $V$.  Essentially, it ensures that the value function $V(\V{x})$ does not  growth too rapidly for large $|\V{x}|$. This condition plays a similar role than a terminal condition does in the finite horizon case.  \McommentO{An example  illustrating this is given in Remark \ref{R:TC-Ex}.} 
\end{itemize}
\end{remarks}
\begin{remark}\label{R:Max1}
If our control problem were a maximisation problem then the condition "\emph{positive submartingale}" in (\textbf{SC}) would be replaced by "\emph{positive supermartingale}", whereas the inequalities  \eqref{E:TCa}-\eqref{E:TCb} would be reversed. 
\end{remark}

\begin{lemma}[\textbf{Verification Result 1}]\label{VT0-SubPhi}
Let $p\ge2$ and let $\phi \in C^2(\rd)$ be a nonnegative function satisfying (\textbf{SC}) and (\textbf{TC}). Then the following holds.
\begin{itemize}[leftmargin=0.75cm]
\item [\emph{i)}] For any admissible policy $\apx \in \Ax$, $\phi(\V{x}) \le \Jx$, $\V{x} \in \rd$.
\item [\emph{ii)}] If condition (\textbf{MC}) holds, then $\phi(\V{x}) = J^{\hat{\Malpha}} (\V{x})$ and, thus, $\phi$  is the value function of the control problem \eqref{J}-\eqref{D:alphaO} and $\{\hat{\Malpha}^{\V{x}}\,:\, \V{x}\in \rd\}$ is a family of optimal admissible policies. 
\item [\emph{iii}] If condition  (\textbf{nC}) holds, then  $\phi$  is the value function and, for each $n$, $\{\apx_n\,:\, \V{x}\in \rd\}$ is a family of $\epsilon_n$-optimal admissible policies with $\epsilon_n := \frac{1}{n}$.
\end{itemize}
\end{lemma}
\Mqo{Probably: If $(SC)$ holds (or (HJB)), then \begin{itemize}
\item if $\liminf_{t\to\infty} e^{-qt} \Ex [\phi (X_t)] \le 0$, then $\phi < J^{\alpha}$.
\item If exists (MC) (or existence of minimiser to $HJB$ equation) and $\lim_{t\to\infty} e^{-qt} \Ex [\phi (X_t)] = 0$, then $\phi = J^{\alpha^*}$.
\end{itemize}}

\begin{remarks}
\end{remarks}
\begin{itemize} [leftmargin=0.75cm]
\item [i)]Under conditions (\textbf{SC}) and (\textbf{TC}), Lemma \ref{VT0-SubPhi} implies that the candidate value function $\phi$ is a lower bound for the value function $V$. However,  to apply this lemma one needs to verify both conditions  \emph{for each  policy} $\apx \in \Ax$. \McommentO{This explains why in some papers of controlled diffusions the transversality condition (\textbf{TC}) is included in the definition of the class of admissible policies $\Ax$.} 
\item [ii)] \Mcomment{The transversality condition (\textbf{TC}) is not much of a problem in the maximisation case with nonnegative running cost as such a condition becomes  \[\limsup_{t\to \infty} e^{-qt} \mathbb{E} \left [ \phi \left ( X_t^x\right )\right ] \ge 0.\] }
\item [iii)] Condition (\textbf{SC}) is usually dealt with via the corresponding \textbf{HJB} equation related to the so-called verifications theorems. 
\end{itemize}
 \textit{Verification theorems} provide sufficient conditions for optimality when the value function is smooth enough. Such conditions are given  in terms of the solution to an integro-differential equation and require the concept of \textit{Markov policies}.

%Among the class of controls, the subset of Markov policies play an important role as they arise as the solutions of the HJB equations associated with the control problem of interest. 

\begin{definition}[\textbf{Markov controls}]\label{MarkovC}
Let $(\xpx, \apx)$ be an admissible pair with $\apx \in \Ax$, $p\ge 2$. The process $\apx =(\sigma,\nu,\mu)$ is called a \emph{Markov control} if there exists an $\Fa$-adapted \cad process $X^{\V{x}}$ starting at $\V{x}$ and there exist measurable mappings $\bar{\mu}: \rp \times \rd \to \rr $, $\bar{\sigma}: \rp \times \rd \to M_{n\times n} (\rr) $ and $\bar{\nu}: \rp \times \rd \to  \mathcal{M}_p$  such that $\as$ 
$\mu_s = \bar{\mu} (s,X^\V{x}_s)$, $\sigma_s = \bar{\sigma} (s,X^{\V{x}}_s)$ and $\nu_s = \bar{\nu} (s,X^{\V{x}}_s)$, for all $s\ge 0$, and, furthermore, if $\bar{\alpha} := (\bar{\sigma},\bar{\nu},\bar{\mu})$ then $(\xpx_{\cdot},\apx_{\cdot}) = (X^{\V{x}}_{\cdot}, \bar{\alpha}^{\V{x}} (\cdot, X^\V{x}_{\cdot}))$ have the same law. If all the mappings in $\bar{\alpha}$ are independent of time, then the corresponding control is said to be a \emph{stationary Markov control}. 
\end{definition}
 
To state the verification theorem of our control problem, we introduce the following conditions:
\begin{itemize}
\item [(\textbf{HJB})]  $\phi$ satisfies the integro-differential equation   
\begin{equation}\label{D:HJB00}
 \inf_{\V{a} \in A} \{  L^\V{a} \phi (\V{x}) -q(\V{x},\V{a}) \phi(\V{x}) + f(\V{x},\V{a}) \} = 0, \quad \text{ for all } \V{x} \in \rd.
 \end{equation}
\item [(\textbf{UI})] For each $\V{x} \in \rd$ and for each admissible policy $\apx$, the family $\{ e^{-\gamma_{t\wedge \tau}^{\apx}}\phi \left (\xpx_{\tau \wedge t} \right )\}_{\tau \in \mathcal{T}}$ is uniformly integrable (UI) for each $t \ge 0$.  Here $\mathcal{T}$ denotes the  family of all stopping times.
 \item [(\textbf{OC})] There exists an admissible policy  $\hat{\Malpha}^{\V{x}} \in \Ax$ with corresponding controlled process $\V{X}^{\hat{\Malpha}^{\V{x}}}$, defined on a filtered probability space $(\hat{\Omega},\hat{\mathcal{F}}, \hat{\mathbb{F}}:=(\hat{\mathcal{F}}_t), \hat{\mathbb{P}})$, such that 
\begin{equation}\label{HJB-op}
  L^{\hat{\Malpha}_s^{\V{x}}} \phi (\V{X}^{\hat{\Malpha}^{\V{x}}}_s) -q(\V{X}^{\hat{\Malpha}^{\V{x}}}_s,\hat{\Malpha}^{\V{x}}_s)\, \phi(\V{X}^{\hat{\Malpha}^{\V{x}}}_s) + f(\V{X}^{\hat{\Malpha}^{\V{x}}}_s,\hat{\Malpha}^{\V{x}}_s)  = 0, \quad \hat{\mathbb{P}}-a.s. \quad \text{for all } s\ge 0,
  \end{equation}
\end{itemize}
A few comments about the previous assumptions:
\begin{itemize}[leftmargin=0.75cm] 
\item [(i)]Equations of the type \eqref{D:HJB00} are usually referred to as \textit{Hamilton-Jacobi-Bellman (HJB) equations} \Mcomment{(or the Dynamic Programming Equation)} and they are the infinitesimal version of the dynamic programming principle. By  assuming sufficient regularity conditions for the value function, 
 this equation can be  derived formally by a standard limiting procedure. 
\McommentO{One of the key differences between equation \eqref{D:HJB00} and the classical HJB equation for controlled diffusions is the additional integral term appearing in the operator $L^{\V{a}}$. Such a term allows us to consider jumps in our model. Our HJB equation is given in terms of an integral-differential equation, whereas in the diffusion setting one gets a  fully nonlinear PDE of second order. }
  \item [ii)]Assumption  (\textbf{UI}) is important for the case when the Bellman process $ S^{\phi, \apx} $  is only a \emph{local} submartingale. \Mcomment{This condition  implies that the nonnegative local submartingale $S^{\phi, \apx} $ is of class (DL) and thus it is a true submartingale.} \McommentO{With this we   can then deal with a stopped version of the transversality condition \eqref{E:TC} that appears in the proof.}
\item [iii)]Assumption (\textbf{OC}) allows one to identify and construct an optimal (stationary) Markov control policy via a pointwise minimisation of the associated HJB equation.
\end{itemize}

 The main result of this paper is the following verification theorem.    
  This theorem characterises the value function $V$ as a solution to  the integro-differential equation \eqref{D:HJB00} and  it also  identifies optimal Markov controls.
 
 \begin{theorem}[\textbf{Verification Result 2}]\label{VT0}
Let $p\ge 2$ and suppose that $\phi \in C^2(\rd)$ is a nonnegative function satisfying $|\phi(\V{x})| \le C (|\V{x}|^q + 1)$ for some $C>0$ and $q \in [2,p]$. Under conditions (\textbf{HJB}), (\textbf{UI}) and (\textbf{TC}), the following   holds.
\begin{itemize}[leftmargin=0.75cm]
\item [\emph{i)}] For any admissible policy $\apx \in \Ax$, $\phi(\V{x}) \le \Jx$, $\V{x} \in \rd$.
\item [\emph{ii)}] If, additionally, condition (\textbf{OC}) holds,
then $\phi(\V{x}) = J^{\hat{\Malpha}} (\V{x}) $,  the family $\{ \hat{\Malpha}^{\V{x}} \,:\, \V{x} \in \rd\}$ is a family of optimal (stationary) Markov policies, and $\phi (\V{x})$ is the value function of the control problem \eqref{J}-\eqref{D:alphaO}. 
\end{itemize}
\end{theorem}
\MremarkO{Theorem \eqref{VT0} states that our control problem  can be reduced to finding a smooth solution to the nonlinear equation \eqref{D:HJB00} (with the corresponding transversality condition) and then minimising $ L^\V{a} \phi (\V{x}) -q(\V{x},\V{a}) \phi(\V{x}) + f(\V{x},\V{a})$ over $A$ for each $\V{x} \in \rd$. }
\begin{remark}\label{R:Max2}
\begin{itemize}[leftmargin=0.75cm] 
\item []
\item [i)]The proof of Theorem \ref{VT0} is given in the appendix  \ref{P:VT0} and follows a localised version of  the standard \textit{martingale approach} for stochastic control problems.  Therein  we first show  that the Bellman process $S^{\phi,\apx}$ is a \emph{local}  submartingale for any arbitrary admissible policy and, further, it is a \emph{local} martingale whenever the policy is optimal. Then, thanks to condition (\textbf{UI}) we conclude the proof by standard localising arguments.
\item [ii)] Since $\phi \in C^2$,  we are seeking smooth  solutions to the HJB equation. This regularity  also justifies the use of  the \ito-Meyer formula in the corresponding proof.  However, by Theorem 71 in \cite[Chapter IV, p. 221]{Protter}, the smoothness condition can be relaxed, at least for the one-dimensional case, by considering $\phi \in C^1$ with an absolutely continuous derivative $f'$. 
\Mcomment{\item [iii)]For a maximisation problem the \emph{infimum} in \eqref{D:HJB00} should be replaced by a \emph{supremum}, so that statement $i)$ in Theorem \ref{VT0}  would imply that $\phi$ is an upper bound for the value function, that is $\phi(\V{x}) \ge \Jx$, for each admissible policy.}
\end{itemize}
\end{remark}
\McommentO{As in the standard case, this theorem assumes the existence of a sufficiently regular and well-behaved solution to the nonlinear equation \eqref{D:HJB00}. This result gives  sufficient conditions for the existence of an optimal policy.}
\begin{remark}\label{R:TC-Ex}
It is worth recalling the importance of  (\textbf{TC}) at prescribing  growth conditions for the candidate value function $\phi$.  In the diffusion setting, a standard example is the following (see \cite[Example 3.1, p. 130]{FlemingSoner}): consider the  equation $\frac{1}{2} \phi''(x) - q \phi(x) + f(x) = 0$, whose   general solution $\phi$ is given by
\begin{equation} \label{Eq:HJB-ExTC}
\phi(x) = x^2 + 1 + c_1 \exp (\sqrt{2} x) + c_2 \exp(-\sqrt{2} x).\end{equation}
It can be proved that the corresponding transversality condition  $\lim_{t\to \infty} e^{-qt} \mathbb{E}_x (\phi (B_t^x)) = 0$ is satisfied only when $c_1 = 0 =c_2$.  Here  $B^x$ is a standard Brownian motion started at $x$. On the other hand, the other solutions given by \eqref{Eq:HJB-ExTC} grow exponentially as $x \to \infty$ or $x\to -\infty$. 
\end{remark}
\Mremark{For any fixed $a \in A$, let $Y^{x,a}$ be an $\rd$-valued  \levy process started at $x$  whose infinitesimal generator coincides  with the operator $L^a$ on $C_c^2 (\rd)$. If $\phi\in C^2(\rd)$ has polynomial growth  of degree $p\ge 2$ and, further,  $\phi$ solves the integro-differential equation
\begin{align}
 L^a \phi(x) -q(x,a)\phi(x) + f(x,a) &= 0, \quad x\in \rd,\\
 \lim_{t\to \infty} \mathbb{E}\left [e^{-\int_0^t q(X_r^x,a)\diff r} f(X_t^x, a) \right ] &= 0, \quad x \in \rd,
 \end{align}
then, by Theorem \ref{VT0}, such a solution admits the probabilistic representation:
\[ \phi(\V{x}) =\mathbb{E} \left [ \int_0^{\infty} e^{-\int_0^t -q\left ( X_r^x, a)\right)\diff r} f\left ( X_t^x, a\right) \diff t \right ].\]
 }
In general, it can be difficult to verify the validity of conditions   (\textbf{UI}) and (\textbf{TC}).  However, if the function $\phi$ is bounded, both conditions follow straightforwardly. In particular,   if the running cost function $f$ is  bounded, then the value function is also bounded and so is any candidate value function $\phi$. We thus obtain the following result for the bounded case.
\begin{corollary}(\textbf{Bounded case})
Suppose that $f:\rd\times A \to \rp$ is a measurable bounded function.  Let $\phi \in C^2(\rd)$ be a bounded nonnegative function satisfying condition (\textbf{HJB}), then   the conclusions of Theorem \ref{VT0}  are valid.
\end{corollary}

%\begin{theorem}\label{VT0-Join}
%Let $\phi \in C^2(\rd)$ be a nonnegative function. Suppose that (\textbf{TC}) holds and either of the following conditions is satisfied by $\phi$.
%\begin{itemize} 
%\item [1.]  (\textbf{A1}) holds, or  
%\item [2.] (\textbf{A2}) holds,  $\phi$ is bounded by a function of polynomial growth of degree at most two and $\phi$ satisfies 
%\begin{equation}\label{D:HJB00}
 %\inf_{\V{a} \in A} \{  L^\V{a} \phi (\V{x}) -q(\V{x},\V{a}) \phi(\V{x}) + f(\V{x},\V{a}) \} = 0, \quad \text{ for all } \V{x} \in \rd.
 %\end{equation}
 %\end{itemize}
 %Then  the following holds.
%\begin{itemize}[leftmargin=0.75cm]
%\item [(a)] For any admissible policy $\ap \in \mathcal{A}_{\V{x}}$, $\phi(\V{x}) \le J(\xp, \ap)$, $\V{x} \in \rd$.
%\item [(b)] If (\textbf{OP}) holds, then
 %$\phi(\V{x}) = J(\V{X}^{\hat{\Malpha}^{\V{x}}},\hat{\Malpha}^{\V{x}} )$, and thus the family $\{ \hat{\Malpha}^{\V{x}} \,:\, \V{x} \in \rd\}$ is a family of optimal policies and $\phi (\V{x})$ is the value function of the control problem \eqref{J}-\eqref{D:alphaO}. 
%\end{itemize}
%\end{theorem}
Thanks to the finiteness of the $p$th moment of each measure $\nu$ outside the ball $B_1$, the growth condition of  $\phi$ in Theorem \ref{VT0} ensures that the non-local term in $L^\V{a} \phi$ (and so equation  \eqref{D:HJB00}) is well-defined.  We are now interested in providing conditions on the running cost function $f$ that guarantee that the corresponding value function is  of  polynomial growth. 

\begin{lemma}\label{P:lemma00}
Consider the control problem \eqref{J}-\eqref{D:alphaO} with admissible policies $\Ax$ for a fixed $p\ge 2$. If there exist a positive constant $c> 0$, an action control $\V{a}_0 \in A$ and  $q  \in [2,p]$ such that the running cost $f: \rd \times A \to \rp$ satisfies $|f(\V{x},\V{a}_0)| \le c (1 + |\V{x}|^q)$ for all $\V{x} \in \rd$, then there exists a positive constant $C> 0$ such that the  value function $V:\rd \to \rd$  satisfies $V(\V{x}) \le C (1+ |\V{x}|^q)$, for all $\V{x} \in \rd$. 
\end{lemma}

\begin{remarks}\label{RpM}
\end{remarks}
\begin{itemize} [leftmargin=0.75cm]
 \item [i)] The key fact in the proof of this lemma (see section \ref{Pr:lemma00}) is to show that for the constant policy $\apx = \V{a}_0$,  with associated controlled process $\V{X}^{\V{a}^\V{x}}$, the following holds \[\int_{0}^{\infty} e^{-\gamma_{s}^{\V{a}^\V{x}}}  \mathbb{E}_\V{x} |\V{X}^{\V{a}^\V{x}}_t|^q   \diff t  \le C(1 + |\V{x}|^q).\] 
Observe that the latter inequality is valid as the process  $\V{X}^{\V{a}^\V{x}}$ is a \emph{\levy martingale with drift} whose $q$th-moments estimates are well-known and satisfy the inequality above.
\item [ii)] In the \ito diffusion setting another standard condition for the running cost function $f$  is the (uniform) growth condition $|f(\V{x},\V{a})| \le C (1 + |\V{x}^q|)$. In our framework, such a case is covered by Lemma \ref{P:lemma00} which also implies that the value function  $V$ has polynomial growth of degree at most $q$. % whenever $f(\cdot, \V{a})$ has  polynomial growth of degree $q$ uniformly in $a$. \Mcomment{ Observe that Lemma \ref{P:lemma00} does not require any  $q$th order growth for $f(\cdot, \V{a})$ (uniformly in $\V{a}$). }
\item [iii)] It is not difficult to see that to allow a value function with exponential growth one needs, in principle, an exponential moment for the measures $\nu$ outside the unitary ball $B_1$. Appropriate assumptions on the control $\apx$ are required to guarantee the finite expectation of the exponential moments of $\xpx$. The latter will then ensure the validity of the corresponding transversality condition.
\end{itemize}
We now state  conditions on the running cost function $f$ that imply  (\textbf{TC}).
\begin{lemma}\label{P:lemma01}
Consider the control problem \eqref{J}-\eqref{D:alphaO} with admissible policies $\Ax$ for a fixed $p\ge 2$. 
Let $f: \rd \times A \to \rp$ be the running cost function.  If there exists a positive constant $c> 0$ such that $|f(\V{x},\V{a})| \ge c (1 + |\V{x}^p|)$ for all $\V{a} \in A$, then the transversality condition \eqref{E:TC} holds for any nonnegative function $\phi \in C^2 (\rd)$ with polynomial growth of degree $p$. 
\end{lemma} 
As a direct consequence of Lemma \ref{P:lemma00} and Lemma \ref{P:lemma01}, we can now give conditions on the running cost  $f$ under which the value function is of polynomial growth and, further, satisfies the transversality condition.
\begin{corollary}\label{C:TC-cond}
Consider the control problem \eqref{J}-\eqref{D:alphaO} with admissible policies $\Ax$ for a fixed $p\ge 2$.
Suppose that there exist $C> 0$ and $\V{a}_0 \in A$  such that the running cost $f: \rd \times A \to \rp$ satisfies $|f(\V{x},\V{a}_0)| \le C (1 + |\V{x}|^p)$ for all $\V{x} \in \rd$ and, further, there exists $c> 0$ such that $|f(\V{x},\V{a})| \ge c (1 + |\V{x}^p|)$ for all $\V{a} \in A$. Then the value function $V$ is of polynomial growth of degree $p$ and satisfies (\textbf{TC}). 
\end{corollary}
\begin{remark}
In particular, this corollary is valid when  $f$ is of the same order than a polynomial function of degree $p\ge 2$. Recall that the parameter $p$ is related to the moments of each $\nu$  outside of the ball $B_1$ (see \eqref{D:Lm}).
\end{remark}

\section{Different classes of admissible controls}\label{S:classes}

Apart from the case where the running cost function is bounded, we have not explored yet conditions that guarantee the validity of  (\textbf{UI}). This assumption, as was pointed out before, is not easy to verify in practical applications.  To fill in this gap, we provide the following three classes of admissible controls for which condition (\textbf{UI}) is no longer needed in the corresponding verification theorems. %Note, however, that such results require further restrictions on  the set of admissible controls.  

\textbf{Case 1.  Integrability conditions on the control processes $\apx$}. 

\begin{definition}\label{Adm-Tilde}
An admissible policy $\apx = (\Msigma, \nu, \Mmu) \in \Ax$ is said to belong to the class $\AxT^p(q)$, $q\in [2,p]$, whenever it satisfies the additional condition:
 \begin{itemize}%[leftmargin=0.29in]
\item [(\textbf{H3})] For all $t \in \Rp$, 
\begin{equation}\label{Q-Fin} 
\int_0^t|\Mmu_s| + ||\Msigma_s||^2 \diff s \in L^{q/2} (\mathbb{P}^{\apx}),\quad   \V{G}_t^{\apx} \in L^{q/2} (\mathbb{P}^{\apx}), \quad  \text{and} \quad    \V{H}_t^{\apx} \in L^{1} (\mathbb{P}^{\apx}),
\end{equation}
 where $ \V{G}_t^{\apx}$ and $ \V{H}_t^{\apx}$ are as defined in \eqref{G-H}. 
 \end{itemize} 
\end{definition}

\begin{theorem}[\textbf{Verification Result 3}]\label{VT-EFinite}
Let $p\ge 2$ and $q \in [2,p]$. If $\phi \in C^2(\rd)$ is a nonnegative function bounded by a polynomial function of order $q$ for which both conditions (\textbf{TC}) and $(\textbf{HJB})$ hold,  then  the conclusions of  \emph{Theorem \ref{VT0}} are also valid for the corresponding minimisation problem over the set  $\AxT(q)$. 
\end{theorem}

\begin{remark}
Condition (\textbf{H3}) seems to be quite restrictive, however, we will see that due to the generality of our set-up is fairly natural to impose integrability conditions of this type.  We will also see that in the standard  SDE framework the corresponding integrability assumptions are implied by the very-well known \ito conditions.
\end{remark}

\Mremark{
Thanks to the assumptions on the growth of the candidate function $\phi$, the proof of Theorem \ref{VT-EFinite} follows easily by observing that  \eqref{Q-Fin} guarantees the finiteness of the expectation of the running maximum of  $\left |\xpx \right |^q$. %This allows one to deal with the localised version of the condition (\textbf{TC}).
}

\textbf{Case 2.  Growth conditions on the control processes $\apx$}. 
\begin{definition}\label{Adm-Hat}
An admissible policy $\apx = (\Msigma, \nu, \Mmu) \in \Ax$ is said to belong to the class $\hat{\mathcal{A}}^p_{\V{x}}$, $p\ge 2$, if  the following conditions hold:
\begin{itemize}
\item [(\textbf{H4})] There exist measurable functions 
\begin{align*}
\hat{\mu}^i :  \rd \times\Omega \to \rd, &\quad 1\le i \le n,\quad \quad
 \hat{\sigma}_{ij} : \rd\times \Omega \to M_{n\times n} (\rr),\quad  1\le i,j\le n,\\
 \hat{\nu} &:\rd\times\Omega \to \mathcal{M}_p,
 \end{align*}
  such that, for each $\hat{g}\in\{\hat{\sigma}, \hat{\nu},  \hat{\mu}\}$, $\hat{g}\left(\xpx_t,\,\cdot \, \right )$ is $\mathcal{F}_t^{\alpha}$-adapted and, further, \[\apx_t = ( \hat{\sigma} (\xpx_t (\omega),\omega), \hat{\nu} (\xpx_t (\omega),\omega), \hat{\mu} (\xpx_t (\omega),\omega)), \quad \as.\] 

\item [(\textbf{GC})]There exist a deterministic positive constant $K$   and a real-valued process $\kappa = (\kappa_t)_{t\ge 0}$ such that for all $t\ge 0$ and $\V{x} \in \rd$
\begin{align}\label{C:Gcond} 
%|\bar{\mu} (t,\V{x}, \omega) - \bar{\mu} (t,\V{y}, \omega)| + |\bar{\sigma} (t,\V{x}, \omega) - \bar{\sigma} (t,\V{y}, \omega)| \,\,+ \quad \quad \quad \quad &\\
 % +\,\, | \int_{\Ro} |\V{z}|^2 \bar{\nu} (t,\V{x}, \omega, \diff \V{z}) -  \int_{\Ro} \big |\V{z}|^2 \bar{\nu} (t,\V{y}, \omega,\diff \V{z} )| &\,\,\,\le\,\,\, K | \V{x} - \V{y} \big|,\\
   |\hat{\mu} (\V{x},\omega) |^p + ||\hat{\sigma} (\V{x},\omega)||^p   +  \int_{\Ro} |\V{z}|^2\vee |\V{z}|^p \hat{\nu} (\V{x}, \omega, \diff \V{z})  &\,\,\,\le\,\,\,|\kappa_t (\omega)|^p  + K  |\V{x} | ^p,\,\, \as
\end{align}
with $ \int_0^t |\kappa_s|^p \diff s \in L^{p/2} (\mathbb{P}^{\alpha})$.
\end{itemize}
%$\hat{\mu}^i :  \rd \times\Omega \to \rd$, $1\le i \le n$, $\hat{\sigma}_{ij} : \rd\times \Omega \to M_{n\times n} (\rr)$, $1\le i,j\le n$, and $\hat{\nu} :\rd\times\Omega \to \mathcal{M}_p$ such that $\apx_s = ( \hat{\sigma} (\xpx_s (\omega),\omega), \hat{\nu} (\xpx_s (\omega),\omega), \hat{\mu} (\xpx_s (\omega),\omega))$, for each $s\ge 0$.
\end{definition}

\Mremark{The previous definition of admissible controls can be thought of as a generalisation of the \ito SDE setting as presented in \cite{Pham}.  In this reference, under an   additional Lipschitz condition,  the process $\kappa_t$ can be given explicitly in terms of the associated drift and diffusion (stochastic) coefficients, see \cite[Section 1.3, p. 23]{Pham}). 
}

\begin{theorem}[\textbf{Verification Result 4}]\label{VT-GC}
Let $p\ge 2$ and $q \in [2,p]$. If $\phi \in C^2(\rd)$ is a nonnegative function bounded by a polynomial function of order $q$, for which both conditions (\textbf{TC}) and $(\textbf{HJB})$ hold,  then  the conclusions of  \emph{Theorem \ref{VT0}} are also valid for the corresponding minimisation problem over the set of admissible controls $\hat{\mathcal{A}}^p_{\V{x}}$. 
\end{theorem}

\begin{remarks}
 The key fact in the proof of Theorem \ref{VT-GC} is to show that, for each admissible policy $\apx \in \hat{\mathcal{A}}_{\V{x}}^p$,    $\Ex \left [ \sup_{0<s\le t} |\xpo_s|^p\right ]  < +\infty$, for each $t\ge 0$.  The latter follows directly from the estimates given in Proposition \ref{qth-moments}. Notice that, apart from the growth condition \eqref{C:Gcond}, the assumption $ \int_0^t |\kappa_s|^p \diff s \in L^{p/2} (\mathbb{P}^{\alpha})$ is crucial. The rest of the proof follows the  same arguments used in the proof of Theorem \ref{VT-EFinite}, so that we omit the details. 
 \end{remarks}

\textbf{Case 3.  Markov conditions on the control processes $\apx$}. 

As a particular case of the class of admissible controls $\hat{\mathcal{A}}^{\V{x}}_p$ given in Definition \ref{Adm-Hat}, we can now restrict our attention to the class of stationary Markov controls, i.e. when $\apx$ is of the form 
\[ \apx_s (\omega) = (\hat{\sigma} (\xpx_s (\omega)), \hat{\nu} (\xpx_s (\omega)), \hat{\mu} (\xpx_s (\omega)) ),\]
for some (deterministic) measurable functions $\hat{\sigma}$, $\hat{\nu}$ and $\hat{\mu}$ (recall Definition \ref{MarkovC}).
\begin{definition}\label{Adm-Markov}
An admissible policy $\apx = (\Msigma, \nu, \Mmu) \in \Ax$ is said to belong to the class $\mathcal{A}^{M,p}_{\V{x}}$, $p\ge 2$, if $\apx$ is a stationary Markov control process and  the following condition is satisfied:
\begin{itemize}
\item [(\textbf{GM})] If $\apx = ( \hat{\sigma} (\xpx_s (\omega)), \hat{\nu} (\xpx_s (\omega)), \hat{\mu} (\xpx_s (\omega)))$, then there exists a deterministic positive constant $K$  such that for all $t\ge 0$ and $\V{x} \in \rd$, the functions $\hat{\sigma}$, $\hat{\nu}$ and $\hat{\mu}$ satisfy
\begin{align}\label{C:Gcond} 
   |\hat{\mu} (\V{x}) |^p + ||\hat{\sigma} (\V{x})||^p   +  \int_{\Ro} |\V{z}|^2\vee |\V{z}|^p \hat{\nu} (\V{x}, \diff \V{z})  &\,\,\,\le\,\,\, K (1+ |\V{x} | ^p).
\end{align}
\end{itemize}
%$\hat{\mu}^i :  \rd \times\Omega \to \rd$, $1\le i \le n$, $\hat{\sigma}_{ij} : \rd\times \Omega \to M_{n\times n} (\rr)$, $1\le i,j\le n$, and $\hat{\nu} :\rd\times\Omega \to \mathcal{M}_p$ such that $\apx_s = ( \hat{\sigma} (\xpx_s (\omega),\omega), \hat{\nu} (\xpx_s (\omega),\omega), \hat{\mu} (\xpx_s (\omega),\omega))$, for each $s\ge 0$.
\end{definition}

\begin{theorem}[\textbf{Verification Result 5}]\label{VT-GM}
Let $p\ge 2$ and $q \in [2,p]$. If $\phi \in C^2(\rd)$ is a nonnegative function bounded by a polynomial function of order $q$, for which both conditions (\textbf{TC}) and $(\textbf{HJB})$ hold,  then  the conclusions of  \emph{Theorem \ref{VT0}} are also valid for the corresponding minimisation problem over the set of admissible (stationary) Markov controls $\hat{\mathcal{A}}^{M,p}_{\V{x}}$. 
\end{theorem}
\begin{remark}
Since this result is just a particular case of Theorem \ref{VT-GC}, we omit its proof.
\end{remark}

\subsection{About the standard SDE settings}\label{S:SDEcase}
Most of the literature dealing with optimal control of \levy-\ito diffusions share two important characteristics in the definition of admissible policies: 
\begin{itemize}[leftmargin=0.75cm]
\item [a)] The optimisation  is usually done over the smaller class of Markov controls $u$. Hence, each process $X^u$  is defined as a solution to a controlled SDE with a Markovian structure: 
\begin{align}\nonumber
X^u_t \,= \,x\, +\, \int_0^t &b(X^u_r, u_r) \diff r \,+ \int_0^t s (X^u_r, u_r) \diff W_r \,+ \\&+ \,\,\int_0^t \int_{|\gamma| \le 1} \gamma(X^u_{r-},z,u_r) \tilde{N} (\diff r, \diff z) +  \int_0^t \int_{|\gamma| > 1} \gamma(X^u_{r-},z, u_{r-}) N (\diff r, \diff z),\label{SD-2}
\end{align}
i.e., the drift  and diffusion coefficients $b$ and $s$, as well as the function $\gamma$ (which determines the size of the jumps), are assumed to be  functions of both the space and the control variable. Here $\tilde{N} (\diff s, \diff z) := N (\diff s, \diff z) - \diff s \nu (\diff z)$ is a compensated Poisson random measure (independent of the Brownian motion $W$), whose mean measure  $\nu$ is fixed and satisfies that $\int_{\Ro} 1\wedge \gamma^2 (\cdot,z,\cdot) \nu(\diff z) < \infty$. 
\item [b)] The drift and diffusion coefficients as well as $\gamma$ satisfy appropriate \ito-type conditions: Lipschitz and linear growth conditions in the space variable and uniformly on the control variable.\McommentO{\footnote{\textcolor{red}{There exists $K>0$ such that for all $t\ge 0$, $a\in A$ and $x,y\in \rd$, the following holds:
\begin{align*}
|b(t,x,a) - b(t,y,a)| + ||\sigma(t,x,a) - \sigma (t,y,a)|| &\le K |x-y| \\
|b(t,x,a)|^2 + ||\sigma (t,x,a)||^2 &\le K^2 (1+|x|^2).
\end{align*} 
or 
\begin{align*}
|b(t,x,a) - b(t,y,a)| + ||\sigma(t,x,a) - \sigma (t,y,a)|| &\le K (|x-y| - |t-s|).
\end{align*}}
}}
\end{itemize}
The importance of such conditions is that they guarantee both (i) the existence and uniqueness of a strong solution to the corresponding SDE,\McommentO{\footnote{\textcolor{red}{Globally (or locally) Lipschitz continuity in the space variable for drift and diffusion coefficients in \ito-diffusions guarantee strong uniqueness \cite[Karatzas,Theorem 2.5, p.287]{Karatzas}}}} and  (ii) the square integrability of the associated control process (see \cite[Chapter 4]{FlemingR1975}, \cite[Section 1.3, p.22]{Pham} for the (continuous) \ito diffusion case, or \cite[Theorem 1.19, p. 10]{OKSulem} for the (jump) \levy diffusion case).\McommentO{\footnote{\textcolor{red}{Globally Lipschitz and linear growth conditions  in the space variable for drift and diffusion coefficients in $\rd$-valued \ito-diffusions guarantee existence of a  square-integrable strong solution \cite[Karatzas,Theorem 2.9, p. 289]{Karatzas}.}}} 
\begin{remarks}
\end{remarks}
\begin{itemize}[leftmargin=0.29in]
\item [i)] Although in our setting condition (\textbf{H3}) seems to be quite strong when compared to the \ito or \levy diffusion cases,  this assumption is crucial  to guarantee the finiteness of the $q$th moments of each process $\xpx$. Nevertheless, by Theorem \ref{VT-GC}, the $q$th moments are also finite whenever the growth condition \eqref{C:Gcond} holds. 
 \item [ii)] %Observe also that the condition (\textbf{GC}) is only a growth condition for $\apx$. 
 By imposing an additional Lipschitz condition we can also guarantee the existence of a strong solution to the corresponding SDE. We then see that our framework encompasses the controlled \ito SDE case. \McommentO{\footnote{\textcolor{red}{In the \ito SDE's setting, it is well-known  that globally Lipschitz continuity for the coefficients $\mu$ and $\sigma$ implies existence of a strong solution and pathwise uniqueness for the corresponding SDE. Under the relaxed assumption of locally Lipschitz continuity, the same results follow by an additional sublinear growth condition \cite[Chapter 5, Karatzas]{Karatzas}}}} 
 \item [iii)] Using Definition \ref{Adm-Hat}, we can recover the  \ito diffusion case by setting $\nu \equiv 0$ and by proceeding as follows: given an admissible control $\alpha_t^x (\omega) : = (\bar{\sigma} ,\bar{\mu}) (\cdot, \omega)$, where $\bar{\sigma} : \rd\times \Omega \to M_{n\times n}(\rr)$ and $\bar{\mu}: \rd\times \Omega \to \rd$ satisfy (\textbf{H4}),  define  $s (\cdot, \alpha_t (\omega)):=\bar{\sigma}(\cdot, \omega)$ and $b (\cdot, \alpha_t (\omega)):=\bar{\mu}(\cdot, \omega)$ as  the corresponding diffusion and drift coefficients.
%\item [ii)] We can recover the standard \ito diffusion framework by setting $\nu \equiv 0$ and taking  the random coefficients $\hat{b}(\V{x},\omega)$ and $\hat{\sigma} (\V{x},\omega)$ in the form $\hat{b} (\xpx_t, u_t (\omega)):= \hat{b}$ and $\hat{\sigma} (\xpx_t, u_t (\omega))$, where $\hat{b}$ and $\hat{\sigma}$ are some given deterministic measurable functions on $ \rd \times A$ and $u = (u_t)_{t\ge 0}$ is an $U$-valued  progressively measurable process. 

\item [iv)]  In the \levy diffusion \eqref{SD-2}, the jump intensity measure $\nu$  is fixed and deterministic. Hence, the control affects only the jump sizes determined by the function  $\gamma(x,z,u)$. Since our  martingale approach allows us to control the jump intensity measure, our setting is more general than the jump case in \cite{OKSulem}.
\item  [v)] We can reformulate the controlled SDE \eqref{SD-2} in terms of our martingale approach as follows. Define the action control set $\Gamma'$ as a subset of $M_{n\times n} (\rr) \times \rd \times F(\rd;\rd)$, where $F( \rd;\rd)$ is the set of measurable functions on $\rd$ with values on $\rd$. Fixed  a jump intensity measure $\nu$ and  replace the operator $L^{\V{a}}$ in \eqref{D:La} by the operator 
\begin{align}\label{D:La-Mod}
   (\mathcal{G}^\V{a} h)(\cdot) := \Mmu^T  \nabla h(\cdot) + \frac{1}{2} \tr( \Msigma^T \Hess h\, \Msigma)(\cdot)  +\int_{\Ro} \left ( h(\cdot+ \V{y}) - h(\cdot) - \V{y}^T \nabla h (\cdot ) \right ) \diff\nu \circ \theta^{-1} (\V{y}),
   \end{align}
  where each action $\V{a} =(\Msigma,\Mmu,\theta) \in \Gamma'$. The set of admissible controls can now be defined (with the appropriate changes) as was done in  Section \ref{CSetp}. 
\end{itemize}

%% file: VF7-FiniteCase.tex
 \subsection{Finite Horizon Case}\label{S:FiniteCase}

We will sketch briefly the formulation for the case when the planning horizon is a finite (deterministic) time interval $[0,T] \subset \rp$.  We take the control setting of Section \ref{CSetp} except for the fact that all processes are defined on the time interval $[0,T]$. Keeping this restriction in mind, we shall use the same notation for the class of admissible controls (recall also Remark \ref{R:Notation-x}).      
Let $p\ge 2$ be fixed throughout this section. For each $\V{z} \in  \rd$, define the cost functional 
  \begin{equation}\label{J-Finite}
  \mathcal{J}\,:\, (\xpz,\apz)\,\,\mapsto\,\, \mathbb{E}_{\V{z}} \left [   \int_0^{T} e^{- \gamma_t^{\apz}} f (t,\xpz_t,\,\apz_t)\diff t + g(\xpz_T)\right ],
     \end{equation}
   where  $f: [0,T]\times \rd \times A \to \mathbb{R}^+$ and $h:\rd \to \mathbb{R}^+$ are   the \textit{running cost} and the \textit{terminal cost} functions, respectively.  As before,  the discount process  is given by $\gamma_{\cdot}^{\apz} := \int_0^{\,\cdot\,} q (s,\xpz_s, \, \apz_s) \diff s$ with $q: [0,T] \times \rd \times A \to \rp$  being a bounded measurable function.
   
The stochastic control problem on the interval $[0,T]$ consists in solving the optimisation problem:
  \begin{equation}\label{J2}
 \inf_{(\xpz,\apz)\,:\,\apz \in \Az} \mathcal{J}(\xpz,\apz) 
      \end{equation}
for some initial state $\V{z} \in \rd $.  

The solution to \eqref{J-Finite}-\eqref{J2} consists then in finding the optimal value of $\mathcal{J}$ and  an optimal policy  (whenever it exists).  A policy $\hat{\Malpha}^{\V{z}} \in \Az$ is called \textit{optimal} if  the corresponding admissible pair $(\V{X}^{\hat{\Malpha}^{\V{z}}},\hat{\Malpha}^{\V{z}})$ satisfies
\begin{equation}\label{FH-2p}
J(\V{X}^{\hat{\Malpha}^{\V{z}}},\hat{\Malpha}^{\V{z}}) \,= \,\inf_{\apz \in \Az} J(\xpz,\apz).\end{equation}

Unlike the infinite horizon, the cost functional (and thus the value function) in the finite horizon case does depend on both the initial state and the initial time of the system. Hence, in order to solve    \eqref{J-Finite}-\eqref{FH-2p} via the \textit{dynamic programming approach}, we consider  the associated family of control problems indexed by the initial time-state points. 

Given   $(t,\V{x}) \in  [0,T] \times \rd$, for each  
admissible pair $(\xptx,\aptx)$, we define 
 \begin{equation}\label{J-tx}
  %J(T-t,\V{x};\xpa,\Malpha)
  \mathcal{J}^{\alpha} (t,\V{x}):=  \Etx \left [   \int_t^{T} e^{- \gamma_{s}^{\apo}} f (s,\xpo_s,\,\apo_s)\diff s + h(\xpo_T)\right ],
     \end{equation}
     where $\Etx$ stands for the mathematical expectation conditional to  $\xptx_t = \V{x}$. As before, we have omitted the superscripts $t,\V{x}$ when appearing inside the operator $ \Etx$. We have that $\mathcal{J}^{\alpha} (t,\V{x})$ is the expected cost of using the control policy $\aptx$ over the time interval $[t,T]$ given the  initial time-state point $(t,\V{x})$. If $t= 0$, we  write $  \mathcal{J}^{\alpha} (\V{x}) \equiv  \mathcal{J}^{\alpha} (t,\V{x}) $. 

   The optimal cost function $V: [0,T]\times \rd \to \mathbb{R}^+$ is then defined  by
     \begin{equation}\label{V-tT}
    % V(T-t,\V{x}) := \inf_{\{(\xpa,\Malpha)\,: \, \Malpha \in \AxtT\}} J(T-t,\V{x};\xpa,\Malpha),
    V(t,\V{x}) := \inf_{\left \{(\xptx,\aptx)\,: \, \aptx \in \Atx\right \}}  V^{\Malpha} (t,\V{x}).
     \end{equation}
 Hence, $V(t,\V{x})$ gives the minimum \emph{cost-to-go}, starting at time $t$ from state $\V{x}$. 
 
   \MremarkO{To indicate that $V$ is a function of the remaining time (and not of the time elapsed), $V(t,\V{x}x)$ is sometimes written as $V(T-t,\V{x})$.
}
Following similar arguments than those used in Section \ref{S:MainResults}, we can obtain the finite horizon counterpart of our previous results. We thus  omit the repetition and only present the following verification theorem.

\begin{theorem}\label{VT0-Finite-bounded}
 Let $f:\rp \times \rd \times A \to \rp$ and $h:\rd \to \rp$ be measurable functions. Let $\phi \in C([0,T]\times\rd)\cap C^{1,2}([0,T)\times \rd)$ be a nonnegative function with polynomial growth in $\V{x}$ of degree $p$ (uniformly in $t$), which solves   
\begin{equation}\label{D:HJB00}
 %\partial_t \phi(t,\V{x})  + \inf_{\V{a} \in A} \{  L^\V{a} \phi (t,\V{x})  -q(t,\V{x},\V{a}) \phi (t,\V{x}) + f(t,\V{x},\V{a}) \} = 0, \quad \text{ for all } \V{x} \in \rd,
 \partial_t \phi(t,\V{x})  + \inf_{\V{a} \in A} \{  L^\V{a} \phi (t,\V{x})  -q(t,\V{x},\V{a}) \phi (t,\V{x}) + f(t,\V{x},\V{a}) \} = 0, \quad \text{ for all } (t,\V{x}) \in [0,T)\times \rd,
 \end{equation}
  with the boundary condition $\phi(T,\cdot) = h(\cdot)$. 
  Suppose that, for each $\aptx \in \Atx$,  the family $\left \{ \phi \left (\tau \wedge T,\,\xptx_{\tau \wedge T} \right )\right\}_{\tau\in \mathcal{T}}$ is uniformly integrable.\footnote{As before, $\mathcal{T}$ denotes the  family of all stopping times.}  
Then the following holds.
\begin{itemize}[leftmargin=0.75cm]
\item [(a)] For any admissible policy $\aptx \in \Atx$, $\phi(t,\V{x}) \le \mathcal{J}^{\alpha}(t,\V{x})$, $(t,\V{x}) \in [0,T]\times \rd$.
\item [(b)] If there exists an admissible policy  $\hat{\Malpha}^{t,\V{x}} \in \Atx$ with corresponding controlled process $\V{X}^{\hat{\Malpha}^{t,\V{x}}}$, defined on a filtered probability space $(\hat{\Omega},\hat{\mathcal{F}}, (\hat{\mathcal{F}}_t), \hat{\mathbb{P}})$, such that $\hat{\mathbb{P}}$-a.s. for all $s\ge 0$
\begin{equation}\label{HJB-opF}
  \partial_s \phi (s,\V{X}^{\hat{\Malpha}^{t,\V{x}}}_s)  + L^{\hat{\Malpha}_s^{t,\V{x}}} \phi (s,\V{X}^{\hat{\Malpha}^{t,\V{x}}}_s) -q(s,\V{X}^{\hat{\Malpha}^{t,\V{x}}}_s,\hat{\Malpha}^{t,\V{x}}_s)\, \phi(s,\V{X}^{\hat{\Malpha}^{t,\V{x}}}_s) + f(s,\V{X}^{\hat{\Malpha}^{t,\V{x}}}_s,\hat{\Malpha}^{t,\V{x}}_s)  = 0,   \end{equation}
with the boundary condition $\phi(T,\cdot) = h(\cdot)$. 
Then, $\phi(t,\V{x}) = J(\V{X}^{\hat{\Malpha}^{t,\V{x}}},\hat{\Malpha}^{t,\V{x}} )$,  the family $\{ \hat{\Malpha}^{t,\V{x}} \,:\, (t,\V{x})  \in [0,T] \times \rd\}$ is a family of optimal policies, and $\phi (t,\V{x})$ is the value function of the control problem \eqref{J-Finite}-\eqref{J2}. 
\end{itemize}
\end{theorem}
\begin{remark}
The proof follows the same arguments used in the proof of Theorem \ref{VT0}, so the details are omitted. \end{remark}

\MremarkO{We can rewrite the finite horizon problem as an infinite horizon control problem with state constraints in the following way: given an admissible pair $(\xp,\ap)$, define the process $Z_t = (\xp_t, T-t)$ with values in $\rd\times [0,T]$ and the operator $\mathcal{G}^a = L^a -\frac{\partial}{\partial t}$.
}

%% file: VF7-OtherGeneralisations.tex
Let us now comment on three possible generalisations.
\subsubsection{Local dynamics and action set $A$}. We can generalise our framework by, for example, replacing $\mathcal{M}_p$  by the set
\begin{equation}
\mathcal{M}_p' := \left \{ \text{measures } \nu \text{ on } \rd \text{ such that } \int_{\rd} \left ( 1\wedge |\V{y}|^2\right ) \vee |\V{y}|^p \nu(\diff \V{y})<\infty \right \}, \quad p \ge 1.
\end{equation}
Hence, to define the local dynamics of a control process, we can replace \eqref{D:La} by the more general \levy operator
\begin{align}\label{D:La-p}
   (L^\V{a} g)(\cdot) := (\V{u} +\Mmu)^T  \nabla g(\cdot) + \frac{1}{2} \tr( \Msigma^T \Hess g\, \Msigma)(\cdot)  +\int_{\Ro} \left ( g(\cdot+\V{y}) - g(\cdot) -  \V{y}^T \nabla g (\cdot )1_{\{|\V{y}| \ge 1\}} \right ) \nu(\diff \V{y}).
   \end{align}
  
\subsubsection{Running cost function $f$.} 
Notice that we have assumed that $f\ge 0$. This condition can be relaxed, for instance, by considering that $f$ is bounded below, say $f \ge m$. In such a case, the previous results can be applied to the modified running cost $\tilde{f} := f -m$. Indeed, under the following assumptions: 
\begin{itemize}
\item [(\textbf{TCa})]  Each admissible policy $\apx$ is such that either  $\Jx = \infty$ or
 \begin{align}
 \liminf_{t\to \infty}\Ex \left [ e^{-\gamma_{t}^{\apo}} \phi \left (\xpo_t\right) \right ] \le 0,\label{E:TCa}
 \end{align}
 \item [(\textbf{TCb})] Each admissible policy $\apx$ is such that  \begin{align}
 \limsup_{t\to \infty}\Ex \left [ e^{-\gamma_{t}^{\apo}} \phi \left (\xpo_t\right) \right ] \ge 0.\label{E:TCb}
 \end{align}
\end{itemize}
Lemma \ref{VT0-SubPhi} becomes:
\begin{lemma}[Verification Result 1']\label{VT0-SubPhi2}
Let $p\ge2$ and the running cost function $f$ be bounded below. Let $\phi \in C^2(\rd)$ be a function bounded below and  satisfying (\textbf{SC}). Then the following holds.
\begin{itemize}[leftmargin=0.75cm]
\item [\emph{i)}] If for every admissible policy $\apx \in \Ax$ conditions  (\textbf{SC}) and (\textbf{TCa}) hold, then $\phi(\V{x}) \le \Jx$, $\V{x} \in \rd$.
\item [\emph{ii)}] If (\textbf{MC}) and (\textbf{TCb}) hold, then  $\phi(\V{x}) \ge J^{\hat{\Malpha}} (\V{x})$. The equality holds when, for every admissible policy with finite payoff, 
\begin{align}
 \lim_{t\to \infty}\Ex \left [ e^{-\gamma_{t}^{\apo}} \phi \left (\xpo_t\right) \right ]= 0.\label{E:TCab}
 \end{align}
 In the latter case, $\{\hat{\Malpha}^{\V{x}}\,:\, \V{x}\in \rd\}$ is a family of admissible optimal policies and $\phi$ equals the value function.
\end{itemize}
\end{lemma}
 In a similar way, we can extend all our results for a lower bounded running cost. For more general functions $f$, additional constraints are required to ensure that the payoff $\Jx$ is well-defined.%, for instance by assuming that $\Ex \int_0^{\infty} f^- (\xpx_s,\apx_s) \diff s < +\infty$, where $f^- := - (f \wedge 0)$.
 
\subsubsection{Weak solutions to HJB equation}
The assumption of having a $C^2$ solution to the HJB equation guarantees that such a solution is regular enough for the integro-differential equation to make sense. %, but  it also gives us the  regularity needed for the application of the generalised \ito's formula. The latter is one of the main tools when proving verification theorems. However, by Theorem 71 in \cite[Chapter IV, p. 221]{Protter}, the $C^2$ assumption  can be relaxed when using \ito's formula. In the one-dimensional case, it is sufficient to take  $\phi \in C^1$ with an absolutely continuous derivative $f'$.  
However, since the existence of $C^2$ solutions is difficult to guarantee, a natural approach to deal with this issue is to introduce the concept of \emph{viscosity solutions}, as has been done in other settings. We leave the study of this issue to future research. 
 
 \MremarkB{Our control problem is also a generalisation of the one considered in \cite{DiTanna2009}. Therein, the author studied a stochastic control problem (on a finite time interval) of a $\rr$-valued  stochastic process $X= (X_t)_{t\in [0,T]}$,  given by $X_t^{\mu} := x + \int_0^t \mu_s \diff s + L_t$, $t\ge0$,   where $\mu = (\mu_s)_{s\in [0,T]}$ is an  admissible control process and $L = (L_t)_{t\in [0,T]}$ is a \levy process whose coefficients in both the diffusion and the jump components are constant.  In our case,  we extend the control to the infinitesimal variance of the continuous part as well as the  jump intensity of the discontinuous component of the system dynamics. }
 

%% file: FV7-Examples.tex
\subsection{Example 1.}\label{S:Example2}
Let $f:\mathbb{R} \to \mathbb{R}_+$ be a symmetric convex function with polynomial growth of degree $p \ge 2$.  Define
 \begin{equation}
 \mathcal{M}_{\le 1} : = \left \{ \nu  \in \mathcal{M}_p \text{ such that } \nu(\rr) \le 1\right \}.
 \end{equation}
Let $A''$ be the  action control set given by
\begin{equation}\label{Action-f}
A'':= \left \{\,a =(\sigma, \nu,\mu)\, \Big | \,\sigma =1; \, \nu \in \mathcal{M}_{\le 1}, \,\mu = \int y \nu (\diff y)\right \},
\end{equation} 
and let $\mathcal{A}_{x}$ be the set of $A''$-valued admissible control processes $\alpha^x = (\alpha^x_s)_{s\ge 0}$ (as defined in Section \eqref{CSetp} with $p\ge 2$ and $u = 0$ for the operator $L^a$). 

   We seek the optimal function\begin{equation}
V \, : \, x \mapsto \inf_{\alpha^x \in \mathcal{A}_{x}} \mathbb{E}_x \left [ \int_0^{\infty} e^{-qt} f(X_t^{\alpha^x}) \diff t\right ].
\end{equation}
 Let $B$ be a standard Brownian motion started at zero and define
\begin{equation}\label{E:Phi-Ex2}
 \psi : x\mapsto \mathbb{E} \left [ \int_0^{\infty} e^{-(q+1)}f (x+B_t)\diff t\right ].
\end{equation} 

\begin{lemma}\label{Example2}
Suppose that
 \begin{equation}\label{UniformB2}
\beta := \sup_{\nu\in \mathcal{M}_{\le 1}} \int_{|y|> 1} |y|^p  \nu (\diff y) < \infty.\end{equation}
Then the following holds:
\begin{itemize}[leftmargin=0.29in]
\item [i)]$\psi$ is symmetric and convex and its global minimum is attained at zero.
\item [ii)]The optimal payoff  is given by 
\begin{equation}
V = \psi + c,\quad \text{ where } \quad c = \frac{\psi (0)}{q},
\end{equation} 
and  $\{\,\hat{\alpha}_{\cdot}^x = (1,-X_{\cdot}^x, \delta_{-X_{\cdot}^x}),\,: \, x \in \rr\}$ is a family of admissible optimal policies, %. policy  is to choose $\hat{\nu}_t = \delta_{-X_t}$, the unit point mass at $-X_t$, 
 where $X^x$ is the Markov process, started at $x$, whose infinitesimal generator defined on $C^2_c (\rd)$ is given by
\begin{equation}
G h(\cdot) := \frac{1}{2} h''(\cdot) +\int \left ( h(\cdot+y) - h(\cdot)\right ) \hat{\nu}(\,\cdot\,;\diff y),
\end{equation}
with  $\hat{\nu} (x; \diff y) := \delta_{-x}(\diff y)$, the unit point mass at $-x$. Thus, the optimal control is to always jump $X$ to zero at maximal (i.e. unit) rate.
\end{itemize}
\end{lemma}
\begin{remark}
In this example, conditions (\textbf{TC}) and (\textbf{UI}) are easily verified  thanks to the definition of the action control set and condition \eqref{UniformB2}. 
\end{remark}
%%%%%%%%%%%%%%%%%%%%%%%
%%%%%%%%%%%%% APPLICATIONS
%%%%%%%%%%%%%%%%%%%.  EXAMPLE 3
%%%%%%%%%%%%%%%%%%%%%%%

\subsection{Example 2.}\label{S:Example3} Let us  consider the same control setting as  in Example 1, but with a slightly different running cost function. Let $f:\mathbb{R} \to \mathbb{R}_+$ be a polynomial of degree $p\ge 2$. Suppose that $f$ is a symmetric, $C^2$,  convex function increasing on $\rp$. %We ignore the trivial case where $f$ is constant and so may assume that $f(x)  \to \infty$ as $x \to \infty$.
 Given $\kappa > 0$, we seek 
   \begin{equation}\label{Eq:V3}
V \, : \, x \mapsto \inf_{\alpha^x \in \mathcal{A}_{x}} \mathbb{E}_x \left [ \int_0^{\infty} e^{-qt} \left [  f(X_t^{\alpha^x})  + \kappa \nu_t (\rr)\right ]\diff t\right ].
\end{equation}
%where $\mathcal{A}_{x} \subset \mathcal{A}_{x}^p$  is the set of $A$-valued admissible control processes $\alpha^x = (\alpha^x_s)_{s\ge 0}$ (as defined in Section \eqref{CSetp} with  $u = 0$ for the operator $L^a$  in \eqref{D:La}). 

The HJB equation for the control problem \eqref{Eq:V3} is now given by 
\begin{equation}
\inf_{a \in [0,1]} \left \{ \frac{1}{2} g'' - q g + f + \kappa a + \int \left ( g(x+y) - g(x)\right ) \nu(\diff y) \right \} = 0.
\end{equation}

\begin{theorem}\label{Example3}
Given $b\ge 0$, define $\phi_b : \rr \to \rp$ by
\begin{equation}\label{Eq:Phi-b}
\phi_b (x) = \mathbb{E}_x \left [  \int_0^{\infty} e^{-qt} \left ( f(B_t^{b,x}) + \kappa 1_{|B_t^{b,x}| \ge b} \right ) \diff t\right ],
 \end{equation}
with $B^{b,x}$ being a controlled BM, started at $x$, which is jumped to the origin at rate $1$ whenever $|B^{b,x}| \ge b$ and is otherwise uncontrolled.
Then, the value function $V$ defined in \eqref{Eq:V3} is given by $\phi \equiv \phi_{\hat{b}}$, where $\hat{b}$ solves $\phi_{\hat{b}} (\hat{b}) - \phi_{\hat{b}} (0) = \kappa$, with corresponding optimal control.
\end{theorem}
\begin{remark}
Unlike Example 1  wherein condition  \eqref{UniformB2} is key to guarantee the transversality condition, in this second example such a condition is a consequence of Corollary \ref{C:TC-cond} and the polynomial form of the running cost $f$.
\end{remark}
%%%%%%%%%%%%%%%%%%%%%%%
%%%%%%%%%%%%% APPLICATIONS
%%%%%%%%%%%%%%%%%%%.  EXAMPLE 4 (Quadratic control)
%%%%%%%%%%%%%%%%%%%%%%%

\subsection{Example 3 (Quadratic Control).}

We now consider the case when the running cost function $f(\V{x}, (\Msigma, \nu, \Mmu))$ is a quadratic form as a function of $\V{x}$ and $\Mmu$. We will see that the  associated payoff function turns out to be a quadratic form as well and we obtain an explicit solution to the stochastic problem  \eqref{J}-\eqref{D:alphaO}.  
 
Let $\Lambda$ and $\Theta$  be positive definite symmetric  matrices in $\Mn$. Consider  the running cost function $f: \rd \times A \to \mathbb{R}^+$ defined as  the quadratic form $f(\V{x}, \V{a}) := \V{x}^T \Lambda \V{x} + \Mmu^T \Theta \Mmu$ for each  $\V{a} = (\Msigma,\nu,\Mmu)\in A$. Let $\V{B}$ be a   symmetric  positive definite matrix solving  the algebraic Riccati equation
\begin{equation}\label{Riccati}
\V{B}^T \Theta^{-1} \V{B} + q\mathbf{B} - \Lambda \,=\,0.
\end{equation}
Let $D$ be an open subset of $M_{n\times n} (\rr) \times \mathcal{M}_2$ and set $\Gamma = D\times \rd$ as the action set $A$. Define 
\begin{align}\label{D:delta}
\hat{\delta} := \inf_{(\Msigma, \nu) \in D} \left ( \text{Tr} (\Msigma^T \mathbf{B}\, \Msigma) + \int_{\Ro} \V{y}^T \mathbf{B}\, \V{y} \,\nu (\diff \V{y}) \right ), 
\end{align}
 and suppose that the infimum in \eqref{D:delta}
is attained at $(\hat{ \Msigma},\hat{\nu}) \in D$. 
\begin{theorem}\label{P:1}
Consider the control problem  \eqref{J}-\eqref{D:alphaO} over the class of admissible controls $\hat{\mathcal{A}}_{\V{x}}^2$ (see Definition \ref{Adm-Tilde}) and with the quadratic running cost $f$ given above.
Define  $\V{Q} := \Theta^{-1} \V{B}$ and $\V{v} := \,-\, \Theta^{-1} \V{P}\V{u}$, where $\V{P} := \, \V{B} \Lambda^{-1} \V{B}$, and $\V{u}\in \rd$ is the drift term in \eqref{D:La}. Let $\hat{\mu}: \rd \to \rd$ be defined by $\hat{\mu} (\V{x}) := -\V{Q} \V{x} + \V{v}$.
 Then the following holds:
\begin{itemize}[leftmargin=0.8cm]
\item [(i)] The family $\{ \hat{\Malpha}^\V{x} : \V{x} \in\rd\}$ defined by 
\begin{equation}\label{alpha*}
\hat{\Malpha}^\V{x}_t :=  \left (\hat{\Msigma}, \hat{\nu}, \hat{\mu} (\hat{\V{X}}_t^{\V{x}})\right ),
\end{equation}
 is a family of optimal policies, where the associated controlled process $\V{X}_t^{\hat{\Malpha}^\V{x}} \equiv \hat{\V{X}}_t^{\V{x}}$ is the   $\rd$-valued % Ornstein-Uhlenbeck-type 
   process with infinitesimal generator $\hat{L}$ defined on functions  $f \in C_c^2(\rd)$ by
\begin{equation}\label{D:Lhat}
( \hat{L} g )(\cdot) :=( \V{u} + \hat{\mu}(\cdot) )^T \nabla g(\cdot) + \frac{1}{2} \tr( \hat{\Msigma}^T \Hess g(\cdot) \,\hat{\Msigma} )+ \int_{\Ro } (g(\cdot+\V{y}) - g(\cdot)- \V{y}^T \nabla g(\cdot))\hat{\nu}(\diff \V{y}),
\end{equation}
\item [(ii)] The value function $V$  is given by  $V(\V{x}) = \V{x}^T \mathbf{B}\, \V{x} + \textbf{c}\cdot \V{x} + d$  for all $\V{x} \in \rd$, where   $\textbf{c} \in \rd$ and  $d \in \rr$ are   given by
 \begin{align}\label{D:abc}
 \textbf{c} := 2 \V{P}^T  \V{u}, \quad \quad d\,:=\, \frac{1}{q} \left (\,      2\V{u}^T \V{P}^T \V{u} + \hat{\delta} - \V{u}^T \V{P} \Theta^{-1} \V{P}^T \V{u}\, \right).
\end{align}

\end{itemize}
 \end{theorem}
 \begin{remark}
The proof, given in Section \ref{S:Qcontrol}, follows again a verification approach: we first show that $\hat{\Malpha}^\V{x}$ as defined in \eqref{alpha*} is an admissible policy for each $\V{x}\in \rd$ (see Lemma \ref{L:1} in Section \ref{SectionP}). We then prove that  $\phi (\V{x})  := \V{x}^T \mathbf{B}\, \V{x} + \textbf{c}\cdot \V{x} + d $ satisfies the assumptions of Theorem  \ref{VT0}. \Mcomment{Here we verify that the pointwise minimisation of the corresponding HJB equation yields the algebraic matrix equation \eqref{Riccati}.} 
 \end{remark}
 
 \begin{remarks}
 \end{remarks}
 \begin{itemize}[leftmargin=0.8cm]
 \item [i)] Notice that the optimal family of policies defined in \eqref{alpha*} is a linear function of the state $\V{x}$. This family depends on the solvability of the algebraic matrix Riccati equation \eqref{Riccati}. Although the dynamics of the controlled system are not linear, this example can be thought of as a generalisation of the standard linear quadratic regulator (LQR) problem,   \Mcomment{see, for example, the finite horizon case in \cite[Chapter VI, Section 5, p.165]{FlemingR1975}}. \McommentO{The essential assumptions in the LQR problem are $i)$ the controlled system dynamics are Gaussian and linear in $(\V{x},\V{a})$ (the space and control variables), $ii)$ the cost function is quadratic in such variables as well.}

\item [ii)] Various criteria to guarantee  the existence and uniqueness of a positive definite solution  to  the Riccati equation \eqref{Riccati} 
  are very well-known in the literature (see, for example,  \cite{Kalman1967,Wonham,Martensson1972}, and references therein). 
Furthermore, it is also known that such a solution can be expressed in terms of the eigenvectors of the $2n\times 2n$-matrix
\[
 \left \{ \begin{array}{cc}
-\frac{q}{2}\V{I} & \V{\Theta}^{-1} \\
-\V{\Lambda} & \frac{q}{2}\V{I}
\end{array} \right \},
\]
see  \cite[Theorem 1]{Potter1966}, \cite[Theorem 1]{Martensson1971}.
\end{itemize}

\textbf{Particular case.}  If the weight cost matrices   for the control problem \eqref{J}-\eqref{D:alphaO} are the diagonal matrices $\Lambda := \lambda \mathbf{I}$ and $\Theta := \theta \mathbf{I}$, with $\lambda \ge 0$ and $\theta > 0$,  then the coefficients of the corresponding value function  $V$ take the explicit values  
 \begin{align*}
\V{B} :&= \frac{\theta}{2} (p-q)\mathbf{I},\quad\quad
\V{c}:= \frac{8\lambda}{\theta (p+q)^2}\V{u},\quad\quad
d:= \frac{||\V{u}||^2}{q\theta (p+q)^2} (8\lambda -(p-q)^2) + \frac{\theta \tilde{\delta} (p-q)}{2q},
\end{align*}
where  $p:= \sqrt{q^2 + 4\lambda/\theta}$.

%% file: V7-Proofs-I.tex
%%%%%%%%%%%%%%%%%%%%%%%%%%%%%%%%%%%%%%               %%%%% PROOF: CLASS OF ADMISSIBLE CONTROLS
%%%%%%%%%%%%%%%%%%%%%%%%%%%%%%%%%%%%%%
%%%%%%%%%%%%%%%%%%%%%%%%%%%%%%%%%%%%%%               %%%%% PROOF: CONCATENATION
%%%%%%%%%%%%%%%%%%%%%%%%%%%%%%%%%%%%%%

\subsection{Proof of Lemma \ref{L:Concatenation}} \label{P:Concatenation}   

\begin{proof}

Define $\upx := \ap \oplus_t \bpa$.   Consider the admissible pairs $(\xp, \ap)$ and $(\ypa,\bpa)$ defined on the probability spaces $(\Omega^{\alpha}, \mathcal{F}^{\alpha}, (\mathcal{F}_t^{\alpha}), \mathbb{P}_{\xi}^{\alpha})$ and $(\Omega^{\beta}, \mathcal{F}^{\beta}, (\mathcal{F}_t^{\beta}), \mathbb{P}_{\eta_t^{\alpha}}^{\beta})$, respectively. Notice that  in the notation $\mathbb{P}_{\xi}^{\alpha}$ and $\mathbb{P}_{\eta_t^{\alpha}}^{\beta}$ we have made explicit the initial distributions $\xi$ and $
 \eta_t^{\alpha}$ of the corresponding control processes $\xp$ and $\ypa$, respectively.  Define a new filtered probability space $(\Omega, \mathcal{F}, (\mathcal{F}_t), \mathbb{P})$ by setting $\Omega := \Omega^{\alpha} \times \Omega^{\beta}$ endowed with the product $\sigma$-algebra $\mathcal{F} := \mathcal{F}^{\alpha} \otimes \mathcal{F}^{\beta}$ generated by the measurable rectangles. Define the probability measure $\mathbb{P}$ on $(\Omega, \mathcal{F})$ as the probability measure on $(\Omega^{\alpha}\times \Omega^{\beta},\mathcal{F}^{\alpha}\otimes \mathcal{F}^{\beta})$ given by 
\begin{equation}
\mathbb{P} (B ) := \int_{\Omega^{\alpha}} \mathbb{P}^{\beta}_{\etat} (B_{\omega_1}) \diff \mathbb{P}^{\alpha}_{\xi} (\omega_1), \quad \quad B \in \mathcal{F}^{\alpha} \otimes \mathcal{F}^{\beta},
\end{equation}
where, for each $\omega_1 \in \Omega^{\alpha}$,  $B_{\omega_1} := \{ \omega_2 \in \Omega^{\beta}\,:\, (\omega_1, \omega_2) \in B\}$ denotes the $\omega_1$-section of $B$. Note that $B_{\omega_1} \in \mathcal{F}^{\beta}$. 
We can now define the process $\zpx$ on $(\Omega, \mathcal{F}, \mathbb{P})$ as follows
\begin{equation}
\zpx_s (\omega_1, \omega_2) = \omega_1 (s)1_{[r,t)}(s) + \omega_2 (s)1_{[t,\infty)} (s),\quad s\in \rp,\,\, \omega_1 \in \Omega^{\alpha}, \omega_2 \in \Omega^{\beta}.
\end{equation}
By construction $(\zpx, \upx)$ is an $\rd\times A$-valued,  $(\mathcal{F}_s)$-adapted \cad process which agrees in law with $(\xp,\ap)$ on $[r,t)$ and with $(\ypa,\bpa)$ on $[t,\infty)$. 

Note now that the integrability of $\int_r^s \left |  L^{\upx_l} h\left (\zpx_l \right)\right | \diff l < +\infty$, for  $s>r$ and $h\in C_c^2(\rd)$, as well as the validity of condition \eqref{A3}, follow from the validity of such conditions for the admissible pairs $(\xp,\ap)$  and $(\ypa,\bpa)$ associated with $\upx$ and $\zpx$, respectively.  It thus remains to prove that 
\begin{equation}
M_t^{h,\upx}:=h(\zpx_t) - \int_r^t \left (L^{\upx_s} h\right) \left ( \zpx_{s-}\right) \diff s, \quad t\ge r,
\end{equation}
is an $(\mathcal{F}_t)-$local martingale under $\mathbb{P}$. Let us then prove that   $\mathbb{E} \left ( M_s^{h,\upx} \Big | \mathcal{F}_k \right ) = M_k^{h,\upx} $ for each $r\le k\le s$.  For this, we shall use that $M_{\cdot}^{h,\ap}$ and $M_{\cdot}^{h,\bpa}$ are $(\mathcal{F}_s^{\alpha})$- and $(\mathcal{F}_s^{\beta})$-local martingales, respectively.

$Case$ $1$. If $ k\le s\le t$, then $\mathbb{E} \left ( M_s^{h,\upx} \Big | \mathcal{F}_k\right ) = \mathbb{E} \left ( M_s^{h,\ap} \Big | \mathcal{F}_k^{\alpha}\right ) = M_k^{h,\ap}$. Similarly, if $t \le k\le s$, then $\mathbb{E} \left ( M_s^{h,\upx} \Big | \mathcal{F}_k\right ) = \mathbb{E} \left ( M_s^{h,\bpa} \Big | \mathcal{F}_k^{\beta}\right ) = M_k^{h,\bpa}$, as required.

$Case$ $2$. If $k \le t\le s$, then, by definition of $(\zpx,\upx)$ and by the law of iterated conditional expectation, we obtain that
\begin{align}
M_s^{h,\upx} &= M_s^{h,\bpa} - \int_r^t L^{\ap_l} h\left ( \xp_l \right ) \diff l \\ 
\mathbb{E} \left (M_s^{h,\bpa} \Big | \mathcal{F}_k  \right ) &= \mathbb{E} \left [ \mathbb{E} \left (M_s^{h,\bpa} \Big | \mathcal{F}_t^{\beta}   \right ) \Big |  \mathcal{F}_k \right ] =  \mathbb{E} \left [ M_t^{h,\bpa} \Big | \mathcal{F}_k   \right ] =  \mathbb{E} \left [ h\left ( \xp_t \right )\Big | \mathcal{F}_k   \right ].
\end{align}
Therefore,
\begin{align}
\mathbb{E} \left ( M_s^{h,\upx} \Big | \mathcal{F}_k\right ) &= \mathbb{E} \left (  M_s^{h,\bpa} - \int_r^t L^{\ap_l} h\left ( \xp_l \right ) \diff l \Big | \mathcal{F}_k \right ) \nonumber \\
&=  \mathbb{E} \left ( h\left ( \xp_t \right ) - \int_r^t L^{\ap_l} h\left ( \xp_l \right ) \diff l \Big | \mathcal{F}_k \right ) \nonumber \\
&=  \mathbb{E} \left (  M_t^{h,\ap}  \Big | \mathcal{F}_k^{\alpha} \right ) = M_k^{h,\ap},
\end{align}
as desired.
\end{proof}

%%%%%%%%%%%%%%%%%%%%%%%%%%%%%%%%%%%%%%               %%%%% PROOF: CLASS OF ADMISSIBLE CONTROLS
%%%%%%%%%%%%%%%%%%%%%%%%%%%%%%%%%%%%%%
%%%%%%%%%%%%%%%%%%%%%%%%%%%%%%%%%%%%%%               %%%%% PROOF: Preliminary Results to prove that X is a semimartingale
%%%%%%%%%%%%%%%%%%%%%%%%%%%%%%%%%%%%%%
\subsection{Preliminary results}
Let us introduce some additional notation and give some preliminary technical results.

Given a function $f \in C^2(\rd)$, define a sequence $\{f_K\}_{K} \subset C_c^2 (\rd)$ as follows.  For each $K > 1$, $K \in \mathbb{N}$, set
\begin{equation}\label{EtaK}
 f_K := f\, \zeta_{K}, \quad \text{where }\,\, \,\,\quad \zeta_K \in C_c^2 (\rd),\,\,\,\V{1}_{B(0,2K)} \le \zeta_K \le \V{1}_{B (0,3K)}.
 \end{equation}
 Note that $f_K \to f$ pointwise as $K \to \infty$.% $\zeta_K$ is a cut-off function\footnote{This can be constructed, as usual, by means of a convolution of a mollifier and an indicator function.}.
 
Given an admissible pair $(\xpx,\apx)$ with $\apx \in \Ax$, $p\ge 2$, we  define, for each $m< K$, $m \in \mathbb{N}$, the stopping times $\tau_m: = T_m \wedge S_m$, where 
\begin{align} \label{tau_n}
T_m := \inf \left \{  r \in \Rp : \left |  \xpx_r \right| > m\right \} \wedge m, 
 \end{align}
 and
 \begin{equation}\label{Sm}
S_m := \inf\left \{ r   \,: \, \int_0^r   Q^{p,\apx}_s \diff s\, > \,m\right \}, \quad m \ge 1,
\end{equation}
 with $Q^{p,\apx}_s$ as given in \eqref{A3} and the usual convention $\inf \varnothing = \infty$. Note that $\tau_m \to \infty$  $\as$  as $m \to \infty$ (thanks to the \cad property  of $\xpx$ and the continuity of  $r \mapsto \int_0^r   Q^{p,\apx}_s \diff s$). 
 
For each $ \V{a} =(\Msigma, \nu, \Mmu) \in A$, we will  rewrite 
  \begin{equation}\label{LasAG}
   \left ( L^{\V{a}} f \right ) (\cdot) = \left (A^{(\Mmu, \Msigma)} f  \right )(\cdot)  + \left ( G^{\nu} f\right) (\cdot).
   \end{equation} 
where%  the local operator $A^{(\Mmu,\Msigma)}$ and the non-local operator $G^{\nu}$ by
\begin{align}\label{AfGf-op0}
\left  (A^{(\Mmu, \Msigma)} f \right ) (\cdot) \,&:= \, (\V{u} +\Mmu)^T  \nabla f(\cdot) + \frac{1}{2} \tr( \Msigma^T \Hess f\, \Msigma)(\cdot) \\
\left( G^{\nu} f \right ) (\cdot) \, &: = \, \int_{\Ro} \left ( f(\cdot+ \V{y}) - f(\cdot) - \V{y}\cdot \nabla f(\cdot) \right ) \nu (\V{y}). \label{AfGf-op}
\end{align} 

\begin{lemma}\label{Lemma0}   Let $f \in C^2(\rd)$ be a function satisfying $|f(\V{x})| \le C(|\V{x}|^q +1)$ for some  $q\in [1,p]$, $p\ge  2$. Take $m \in \mathbb{N}$ and let $\{f_K\}_{K > m} $ be a sequence of functions approximating $f$ defined via \eqref{EtaK}.  Then, for any admissible pair $(\xpx,\apx)$, $\apx \in \Ax$, there exists a positive constant $C(m,f)$ (independent of $K$), such that, for each $K > m$,
 \begin{align}  
\left | (\Lasx f_K)(\xpx_{s-}) \right |  & \le  C (m,f) Q_s^{p,\apx} ,  \quad s \le  \tau_{m}.\label{EstLfK} 
\end{align}
In particular, the $\as$  convergence 
\begin{equation}\label{LK-Conv}
\int_0^{t\wedge \tau_m} (\Lasx f_K)(\xpx_{s-}) \diff s \,\,\to \,\, \int_0^{t\wedge \tau_m} (\Lasx f)(\xpx_{s-}) \diff s, 
\end{equation}
holds for all $t \in \rp$
 as $K \to \infty$. 
\end{lemma}
\begin{remark} Recall that we will always omit the superscript $\V{x}$ in $\xpx$ and $\apx$ whenever they appear inside the operator $\Ex$.
\end{remark}
\begin{proof}
 Take $\apx = (\Msigma, \nu, \Mmu) \in \Ax$,  $f$ and $\{f_K\}_K$ as in the statement. Since $f_K \in C_c^{2}(\rd)$,   the continuity of $f_K$ and the fact that $f_K = f$ on $[-2K, 2K]$ yields
\begin{equation}\label{A-Est}
 \left | \left (A^{(\Mmu_s, \Msigma_s)}  f_K \right)(\xpx_{s-}) \right |  \,\, = \,\,\left | \left (A^{(\Mmu_s, \Msigma_s)}  f\right)(\xpx_{s-}) \right | \,\, \le \,\, c_0(m,f)\left (|\Mmu_s| + ||\Msigma_s||^2 \right),
 \end{equation}
where $c_0 (m,f) := \max\{\, (\V{u}+1)\,\sup_{|\V{z}| \le m} |\nabla f (\V{z})|,\,\frac{1}{2}\sup_{|\V{z}| \le m} || \Hess f\, (\V{z})||\, \}$.

As for the non-local part, observe that for each $\V{x} \in \rd$ the integral term $\left (G^{\nu_s} f_K\right )(\V{x})$ can be split into two regions: $E:=\{\V{y} \in \Ro : |\V{y}|\le 1\}$ and $E^c :=\{\V{y}\in \Ro: |\V{y}| > 1\}$. For  $|\V{x}| \le m$, by Taylor's theorem, there exists $\theta \in (0,1)$ such that
\begin{align}  \left | f_K(\V{x}+ \V{y}) - f_K(\V{x}) - \V{y}\cdot \nabla f_K(\V{x}) \right | \,&= \,\frac{1}{2} | \sum_{i,j =1}^n \partial^2_{ij} f_K (\V{x} + \theta \V{y}) y_i y_j| \,  \le \, c_1(m,f) |\V{y}|^2,   \quad \V{y} \in E, \label{Gf1}
 \end{align}
 where $c_1 (m,f) :=  \frac{1}{4} |\V{y}|^2  \sum_{i,j =1}^n \sup_{|\V{z}| \le m + 1}\left |\partial^2_{ij} f (\V{z}) \right |$. Note the use of   inequality $2y_i y_j \le y_i^2 + y_j^2 \le |\V{y}|^2$, as well as the fact that  $c_1$ does not depend on $K$ as  (by construction) $f_K=f$ on $[-2K,2K] \subset [-m-1,m+1]$.  
 
 On the other hand, again using that $|f_K| \le |f|$ and $f$ has polynomial growth of  degree $q \in [1,p]$, we can find a positive constant $c_2' (m,f) > 0$ such that $ \big | f_K(\V{x}+ \V{y})\big| \le  c_2' (m,f) |\V{y}|^{p}$, for all $|\V{x}| \le m $ and $\V{y} \in E^c$. Thus
\begin{align}\label{Gf2}
  \left | f_K(\V{x}+ \V{y}) - f_K(\V{x}) - \V{y}\cdot \nabla f_K(\V{x}) \right | &\le  %c_2 (m,f) (c_m + 2 \sup_{|\V{z}| \le m}  \{ \left |f (\V{z}) \right | + \left | \nabla f(\V{z})\} \right |) |\V{y}|^2  \,\,=:\,\,
    c_2 (m,f) |\V{y}|^p,   \quad   \V{y} \in E^c,
 \end{align}
 where $c_2 (m,f) := \left (c'_2 (m,f) + \sup_{|\V{z}| \le m}  \{ \left |f (\V{z}) \right | + \left | \nabla f(\V{z})\} \right |)\right)$. 
 %Thus, \eqref{Gf1} - \eqref{Gf2} imply that,  
 %\begin{align}\label{Gf3}
  %\left | f_K(\V{x}+ \V{y}) - f_K(\V{x}) - \V{y}\cdot \nabla f_K(\V{x}) \right | &\le  
   % c_1 (m,f) |\V{y}|^2,   \quad   \V{y} \in \Ro,
 %\end{align}
 %where $c(m,f) := 2 \max\{c_1(m,f), c_2(m,f)\}$. 
 
 Since $f_K = f$ on $[-2K, 2K]$ and $\left | \xpx_{s-}\right | \le m < K$ $\as$ for  all $s\le \tau_m$, the estimates \eqref{Gf1} - \eqref{Gf2} imply that 
 \begin{equation} \label{GK-Est}
 |\left( G^{\nu} f_K \right ) (\xpx_{s-}) |\le c(m,f) \int_{\Ro} |\V{y}|^2 \vee |\V{y}|^p \nu_s (\diff \V{y})\nu_s (\diff \V{y}), \quad \text{for all }\,\, K\ge m,
 \end{equation}
 where $c(m,f) := 2 \max\{c_1(m,f), c_2(m,f)\}$. 
 %  This implies the result for all $|\xpx_{s-}|\le m$. %Note that the right-hand side of \eqref{Gf3} does not depend on $K$. 
 Estimates \eqref{A-Est} and \eqref{GK-Est}, together with \eqref{LasAG} yield
 \[ \left | (\Lasx f_K)(\xpx_{s-})  \right | \,\le\, c_0(m,f)\left (|\Mmu_s| + ||\Msigma_s||^2 \right)\,+c(m,f) \int_{\Ro} |\V{y}|^2 \vee |\V{y}|^p \nu_s (\diff \V{y})\nu_s (\diff \V{y}).\] Estimate \eqref{EstLfK} follows  by setting $C(m,f) := \max \{ c_0 (m,f), c(m,f)\}$ (recall definition of $Q_s^{p,\apx}$ in \eqref{A3}).

Now, to prove the convergence  \eqref{LK-Conv}, thanks to  \eqref{LasAG} and  the equality in \eqref{A-Est}, it is sufficient to prove the $\as$ convergence, for all $t\in \rp$, 
 \begin{align}  \label{IntGK-Conv}
   \int_0^{t\wedge \tau_m} \left (G^{\nu_s} f_K\right ) (\xpx_{s-}) \diff s \,&\to \,  \int_0^{t\wedge \tau_m} \left (G^{\nu_s} f\right ) (\xpx_{s-}) \diff s\quad \text{as} \,\, K\to \infty.
  \end{align}  
Since \eqref{GK-Est} holds for $f_K$ and $ \,\int_{\Ro} |\V{y}|^2 \vee |\V{y}|^p\nu_s (\diff \V{y})< +\infty$ (as $\nu_s \in \mathcal{M}_p$, recall  definition \eqref{D:Lm}),   the DCT implies 
 \begin{align}\label{GK-Conv}
     \left (G^{\nu_s} f_K\right ) (\xpx_{s-})  \,&\to \,   \left (G^{\nu_s} f\right ) (\xpx_{s-}) \quad \text{as} \,\, K\to \infty.
     \end{align}
 Moreover, since $\int_0^{t\wedge \tau_m} \diff s \,\int_{\Ro}  |\V{y}|^2\vee |\V{y}|^p\nu_s (\V{y})\diff \V{y}  < m$ (by definition of $S_m$ and because $\tau_m \le S_m$),  DCT implies \eqref{IntGK-Conv}, as required. 
  \end{proof}
 
 %%%%%%%%%%
 %%%%%%%%%%%
 %%%%%%%%%%

\begin{theorem}\label{LMtgle}
Let $p\ge 2$ and take any admissible pair $(\xpx,\apx)$, $\apx \in \Ax$ defined on the filtered probability space $(\Omega^{\alpha}, \mathcal{F}^{\alpha}, \Fa:=(\Ft^{\alpha}), \mathbb{P}^{\alpha})$. Then, for each bounded function $f \in C^2 (\mathbb{R}^d)$ with polynomial growth of degree  $q \in [1,p]$, the process $M^{f,\apx} = (M_t^{f,\apx})_{t \in \rp}$ defined by
\begin{equation}\label{Mfk0}
M_t^{f,\apx} := f(\xpx_t) - \int_0^t  \left (\Lasx f \right )(\xpx_{s-})\diff s, \quad t \in \rp,
\end{equation}
is a local $(\Fa,\Pa)$-martingale.
\end{theorem}
\begin{proof} 
 
  Let $\apx = (\Msigma, \nu, \Mmu) \in \Ax$, $ \xpx $ and $f$ be as in the statement.  Take $m \in \mathbb{N}$ and  consider a sequence $\{f_K\}_{K > m} $ defined via \eqref{EtaK}.
  Since $f_K \in C_c^{2}(\rd)$, condition (\textbf{H1}) implies that the process $M^{f_K}  = (M_t^{f_K})_{t\in \rp}$, where
\begin{equation}\label{Mtgle-f}
M_t^{f_K} := f_K (\xpx_t) - \int_0^t  \left (\Lasx f_K \right )(\xpx_{s-})\diff s, \quad t \in \rp,
\end{equation}
is a local $(\Fa,\Pa)$-martingale. 
 
Let $\tau_m: = T_m \wedge S_m$, where $T_m$ and $S_m$ are defined via \eqref{tau_n} and \eqref{Sm}, respectively. Note that the stopped process   $M^{f_K,\tau_m}_{\cdot} := M^{f_K}_{\,\cdot \wedge \tau_m}$ is also  a local $(\Fa,\Pa)$-martingale \cite[Corollary 3.6, Chapter II, p. 71]{RYor}.    Since (by  construction)  $f_K \left ( \xpx_{t\wedge \tau_m}\right) \to f (\xpx_{t\wedge \tau_m})$  as $K \to \infty$ and, further, \eqref{LK-Conv} in Lemma \ref{Lemma0} also holds, we obtain that $M^{f,\tau_m}_{t} =\lim_{K\to \infty}M_{t}^{f_K,\tau_m}$ $\as$ for all $t \in \rp$, where $M^{f,\tau_m}_{\cdot} := M^{f}_{\cdot \wedge \tau_m}$ is the stopped version of the process in \eqref{Mfk0}.

Let us now prove that, for each $m$,  the process $M^{f, \tau_m}$ is a  local $(\Fa,\Pa)$-martingalem.  Take $s,t \in \rp$, $s < t$ and  $B \in \mathcal{F}_s^{\alpha}$. It is sufficient to show that $ \Ex \left [  (M_t^{f,\tau_m} - M_s^{f,\tau_m})1_{B}\right ] =  0$. Since
\begin{align}\label{MtgleC}
   \left (M_t^{f,\tau_m} - M_s^{f,\tau_m} \right )1_{ B}= \,\, & \left (M_t^{f,\tau_m} - M^{f_K,\tau_m}_{t}\right )1_{ B} \,  
  +\,  \left (M^{f_K,\tau_m}_{t} - M^{f_K,\tau_m}_s\right)1_{ B}\,+ \,   \left (M^{f_K,\tau_m}_s - M_s^{f,\tau_m}\right)1_{ B},
\end{align}
by taking expectations and using that, for each $K$ and $m$, the second term vanishes because $M^{f_K,\tau_m}$ is a true $(\Fa,\Pa)$-martingale, we only need to prove that, for each $t$, $M^{f_K,\tau_m}_{t} \to M^{f,\tau_m}_t$ in $L^1 (\mathbb{P}^{\alpha})$ as $K \to \infty$.

Using \eqref{EstLfK} in Lemma \ref{Lemma0}, we get that, for fixed $t \in \rp$ and $m \in \mathbb{N}$, $\sup_{K} \left |M_t^{f_K,\tau_m} \right | \,\le\, |f| + m t C(m,f)$ (because $f$ is bdd, $|f_K| \le |f|$ and by definition of $\tau_m$).  Therefore, for each $t \in \rp$ and $m \in \mathbb{N}$, the family  of  r.v.'s $ \mathfrak{M}:= \left \{ M_t^{f_K,\tau_m} \, :\, K > m    \right \}$  is uniformly integrable (UI) (see, for example, \cite[Chapter 1, p.8]{Protter}).  
 
Therefore, there exists a process $M^m$ such that $\as$ $M^{f_K}_{ \,t \wedge \tau_m} \to M_t^{m}$  for all $t\in \rp$, as $K \to \infty$.   It follows that $M_t^{m}$ is integrable for each $t \in \rp$ and, further,  $M^{f_K}_{ \,t \wedge \tau_m} \to M_t^{m}$ in $L^1 (\mathbb{P}^{\alpha})$ as $K \to \infty$. The previous implies then that the  stopped process $M^{f,\tau_m}$ is an $(\Fa,\Pa)$-adapted martingale. Hence, by Theorem 50 in \cite[Chapter I, p. 38]{Protter}, we conclude that $M^f $ is an $(\Fa,\Pa)$-local martingale, as required.
\end{proof}

%%%%%%%%%%%%%%%%%%%%%%%%%%%%%%%%%%%%%%               %%%%% PROOF: CLASS OF ADMISSIBLE CONTROLS
%%%%%%%%%%%%%%%%%%%%%%%%%%%%%%%%%%%%%%
%%%%%%%%%%%%%%%%%%%%%%%%%%%%%%%%%%%%%%               %%%%% PROOF:  Characteristics of X
%%%%%%%%%%%%%%%%%%%%%%%%%%%%%%%%%%%%%%
  \subsection{Proof of Proposition \ref{X-charact}} 

\begin{proof}

Let $(\xpx, \apx)$ be an admissible pair. Observe first that $\V{B}$ (resp. $\V{C}$) is predictable as, by definition, it is a Lebesgue integral of the locally integrable processes $\Mmu$ and $\nu$ (resp. $\Msigma$). \Mcomment{By  (\textbf{H2}), for each $t\ge r$, $ \int_r^t Q_s^{p,\ap}\diff s$ is the  Lebesgue integral of a locally integrable process, hence the dominated convergence theorem (DCT)  implies that $\as$  the paths  $t \mapsto \int_r^t Q^{p,\ap}_s\diff s$ are continuous. Thus, the process  $\int_0^{\cdot}Q_s^{p,\ap}\diff s$ is predictable.} 
As for the random measure $\eta$, by \cite[Definition 1.6, Chapter II.1a, p. 66]{JacodS1987} we need to prove that, for any predictable function $W(\omega, s,\V{x})$ on $\Omega^{\alpha}\times \rp \times \Ro$, the integral process $W\ast \eta$ is also predictable, where 
\begin{equation}\label{ast}
W\ast \eta_t (\omega) := \int_{(0,t]\times \Ro} W (\omega, s, \V{x}) \eta (\omega; \diff s, \diff \V{x})
\end{equation}
if $\int_{[0,t]\times \Ro}| W (\omega, s, \V{x}) |\eta (\omega; \diff s, \diff \V{x})$ is finite, and equal to $+\infty$ otherwise. Since  $\eta (\diff s, \diff \V{y}) = \diff s \otimes \nu_s(\diff \V{y})$,  for each $t$ and $\omega$, the integral $W\ast \eta_t (\omega)$ is a Lebesgue integral of the product of two predictable processes: $W$ and $\Msigma$. Hence,   $W\ast \eta$ is predictable.

Therefore, by   \cite[Theorem II. 2.42 p. 86]{JacodS1987}, we only need to show that, for each bounded function $f \in C^2 (\mathbb{R}^d)$, the process
\begin{align*}
N^{f,\apx}_t \,\,:= \,\,&f(\xpx_t) - f(\xpx_0) - \int_0^t \sum_{i=1}^d \partial_i f (\xpx_{s-}) \diff B_s^j - \frac{1}{2} \int_0^t \sum_{i,j =1}^d \partial^2_{ij} f(\xpx_{s-}) \diff C_{ij}(s)\\
&- \int_{[0,t]\times\Ro} f(\xpx_{s-} + \V{y}) -f (\xpx_{s-}) - h(\V{y}) \cdot \nabla f(\xpx_{s-}) \eta(\diff s, \diff \V{y})
\end{align*}
is a local martingale. By the definition of $\Lasx$ and $\eta$,  we have the equality
\begin{equation}\nonumber
N_t^{f,\apx} = f(\xpx_t) - f(\xpx_0) - \int_0^t  \left (\Lasx f \right )(\xpx_{s-})\diff s, \quad t \in \Rp.
\end{equation}
Therefore, the result follows from Theorem \ref{LMtgle}.
\end{proof}

%%%%%%%%%%%%%%%%%%%%%%%%%%%%%%%%%%%%%%               %%%%% PROOF: CLASS OF ADMISSIBLE CONTROLS
%%%%%%%%%%%%%%%%%%%%%%%%%%%%%%%%%%%%%%
%%%%%%%%%%%%%%%%%%%%%%%%%%%%%%%%%%%%%%               %%%%% PROOF:  Estimates of qth-moments
%%%%%%%%%.   a) local martingale f(X), $f$ polynomial growth $p>2$
%%%%%%%%%.   b) Estimates
%%%%%%%%%%%%%%%%%%%%%%%%%%%%%%%%%%%%%%
 \subsection{Proof of Proposition \ref{qth-moments}}
 For the proof of this result, we will need the following auxiliary lemma.
\begin{lemma}\label{L2-f2}
Theorem \ref{LMtgle} is also valid for any function $f \in C^2 (\rd)$ satisfying the polynomial growth $|f(\V{x})| \le C (|\V{x}|^q + 1)$ for some $C >0$ and $q \in [2,p]$. 
\end{lemma}
\begin{proof}
Let $(\xpx, \apx)$ be an admissible pair with $\apx \in \Ax$ and let $f \in C^2(\rd)$  be as in the statement. Define 
\[ W(\omega,s,\V{y}):=  f(\xpx_{s-}+\V{y}) - f(\xpx_{s-}) -\sum_{i=1}^n \partial_i f(\xpx_{s-}) y_i, \quad (\omega,s,\V{y}) \in \Omega^{\alpha}\times \rp \times \Ro. \]
Notice that all processes $f(\xpx_{-}+\V{y})$, $ f(\xpx_{-})$ and $\partial_i f(\xpx_{-})$ are left-continuous with right limits, so they are locally bounded and predictable. The latter implies that,  for each $\V{y}$, the process $W$ is also predictable.  Since  (by Theorem \ref{X-smtgle}) $\xpx$ is a semimartingale, the generalised \ito formula  \cite[Theorem 4.57, Chapter I.4e, p. 57]{JacodS1987} implies that $f(\xpx)$ is also a semimartingale satisfying
\begin{align}\nonumber
f(\xpx_t) %&= f(\V{x}) + \sum_{i=1}^n \int_{0+}^t& \partial_i f(\xpx_{s-})\diff \xpi_s +\frac{1}{2} \sum_{1\le i,j\le n} \int_{0+}^t \partial_{ij} f(\xpx_{s-}) \diff \lrangle{X^{i,c}}{X^{j,c}}  + \sum_{s\le t} \left \{ f(\xpx_s) - f(\xpx_{s-}) -\sum_{i=1}^n \partial_i f(\xpx_{s-}) \Delta \xpi_s \right \}\\
 &= f(\V{x}) + \sum_{i=1}^n \int_{0+}^t \partial_i f(\xpx_{s-})\diff \xpi_s +\frac{1}{2} \sum_{1\le i,j\le n} \int_{0+}^t \partial_{ij} f(\xpx_{s-}) \diff \lrangle{X^{i,c}}{X^{j,c}}  + W\ast \eta^X_t \\
 %&= f(\V{x}) + \sum_{i=1}^n \int_{0+}^t \partial_i f(\xpx_{s-})\diff M^i + \sum_{i=1}^n \int_{0+}^t \partial_i f(\xpx_{s-})(u^i +\mu_i(s))\diff s \nonumber  \\ &+\frac{1}{2} \sum_{1\le i,j\le n} \int_{0+}^t \partial_{ij} f(\xpx_{s-}) (\Msigma_s^T \Msigma_s)_{ij} \diff s   + W\ast ( \eta^X - \eta ) + W \ast \eta \\
 &= f(\V{x}) + \sum_{i=1}^n \int_{0+}^t \partial_i f(\xpx_{s-})(u^i +\mu_i(s))\diff s +\frac{1}{2} \sum_{1\le i,j\le n} \int_{0+}^t \partial_{ij} f(\xpx_{s-}) (\Msigma_s^T \Msigma_s)_{ij} \diff s   + W \ast \eta_t  + N_t \nonumber\\
 &= f(\V{x}) + \int_0^t \Lasx f\left (\xpx_{s-} \right)  +  N_t 
\end{align}
where $N_t :=  \sum_{i=1}^n \int_{0+}^t \partial_i f(\xpx_{s-})\diff M^i_t  + W\ast ( \eta^X - \eta )_t $ is a local martingale (recall that $M^i$ is the local martingale in the decomposition of $\xpi$ and $\eta$ is the compensator  of the random measure $\eta^X$). Second equality follows from Theorem \ref{X-smtgle} and  the third equality from the definition of $\Lasx$ and the fact that  $W \ast \eta_t (\omega) = \int_{[0,t]\times \Ro} W(\omega,s,\V{y}) \diff s \otimes \nu_s (\omega;\diff \V{y}) $ (see definition in  \eqref{ast}). Let us observe that  $f$ having a  polynomial growth of degree $q \le p$ with $p\ge 2$ is a key assumption to guarantee that $W \ast \eta_t$ is well-defined. The latter holds because $\nu_s$ takes values in $\mathcal{M}_p$ and thus it is a measure with finite second moments inside the unitary ball $B_1$ and finite $p$th-moments in $B_1^c$. This concludes the proof.
\end{proof}
%%%%%%%%%%%%%%
%%%%%%%%%%%%%%  b) Estimates
%%%%%%%%%%%%%%

\begin{proof} (of Proposition \ref{qth-moments})

$(i)$ Equality \eqref{X-decomp} is the canonical representation (relative to $h$) for special semimartingales  (see \cite[II.2c, Theorem 2.34, p.84]{JacodS1987}) and follows from Proposition  \ref{X-charact} and \cite[II.2c, Corollary 2.38, p.84]{JacodS1987}.

$(ii)$ Since $\left | \Xc_s \right |^q \le C \sum_{i=1}^n \left | \Xic_s \right |^q$ and (by Corollary \eqref{X-charact}) the equality $\lrangle{\Xic}{\Xic}_t = \int_0^t \sum_{k=1}^n \sigma_{ik}^2 (s)\diff s $ holds, the Burkholder-Davis-Gundy inequality and the fact that $\sum_{k=1}^n |a_i|^r \le c \left ( \sum_{k=1}^n |a_i|\right)^r$ for $r>1$ and some constant $c>0$,  imply
\begin{align*}
\Ex \sup_{0 \le s \le t} \left |\Xc_s \right |^q &\le C \Ex \left [ \sum_{i=1}^n \left ( \sum_{k=1}^n  \int_0^t \sigma^2_{ik} (s) \diff s\right)^{q/2} \right ] \\
&\le C \Ex \left [  \left (  \sum_{i=1}^n\sum_{k=1}^n  \int_0^t \sigma^2_{ik} (s) \diff s\right)^{q/2} \right ] = C \Ex \left (  \int_0^t ||\Msigma_s||^2 \diff s\right )^{q/2},
\end{align*}
as required.

$iii)$ To deal with the running maximum of the discontinuous martingale part of $\xpx$, we consider the controlled process $\V{Y}$ obtained by taking the policy $\beta = (0,\nu,-\V{u}) \in \Ax$, where $\nu$ is the same process in the control $\apx = (\Msigma, \nu,\Msigma)$ and $\V{u}$ is the vector corresponding to the operator $L^a$ defined in \eqref{D:La}. Therefore,  the statement $(i)$ proved above implies that 
\begin{equation}
 \Y_t = \V{x} +  \int_0^t \int_{\Ro} \V{y} \left (\ny -\eta \right )(\diff s, \diff \V{y}),
\end{equation}
where $\ny$ is the integer-valued random measure associated with the jumps of $\Y$ and the random measure $\eta$ is its predictable compensator. Hence, $\Y$ is a local martingale and, by Corollary \eqref{X-charact}, $\eta (\omega,\diff s, \diff \V{y}) = \diff s \otimes \nu_s (\omega,\diff \V{y})$.  Thus, to obtain the estimate for the process $\Xd$, we only need to estimate  $\left |\Y_t \right |^q$ for $\V{x} = 0$. 

Define
\begin{equation} \label{F}
 F(\omega,s,\V{y}):=  |\Y_{s-}+\V{y}|^q - |\Y_{s-}|^q -\V{y}^T \nabla |\Y_{s-}|^q, \quad (\omega,s,\V{y}) \in \Omega^{\alpha}\times \rp \times \Ro,
 \end{equation}
and 
\begin{equation} \label{G}
 G(\omega,s,\V{y}):= \V{y}^T \nabla |\Y_{s-}|^q \quad (\omega,s,\V{y}) \in \Omega^{\alpha}\times \rp \times \Ro.
 \end{equation}
Observe that the processes $|\Y_{-}|^q$ and $\nabla |\Y_{-}|^q$ are left-continuous with right limits. The previous implies that,   for each $\V{y}$, both processes $F$ and $G$ are also predictable.  

\ito's formula applied to  $h:\V{y} \mapsto |\V{y}|^q$ implies that 
\begin{align}
|\Y_t|^q &=  G\ast (\ny -\eta)_t + F\ast \ny_t = |\V{x}|^q + G\ast (\ny -\eta)_t + F\ast (\ny -\eta )_t + F \ast \eta_t,\label{Yq-2}
\end{align}
where we have used that $\eta$ is the compensator of $\ny$. Notation $\ast$ stands for the stochastic integral defined in \eqref{ast}. 

Observe now that the process $N$ defined by
 \begin{align*}
 N_t &:= (F + G)\ast (\ny -\eta)_t \\
 & =  \int_0^t  \int_{\Ro} \left ( |\Y_{s-} + \V{y}|^q - |\Y_{s-}|^q \right)  (\ny -\eta) (\diff s,\diff \V{y}), 
 \end{align*}
 is a local martingale. Without loss of generality, let us  assume that $ N=|F + G|\ast (\ny -\eta)$ is a true martingale (otherwise one can proceed by considering an appropriate localising sequence). 
 
 Hence, since $\sup_{0\le s \le t} |N_s| \le   |F + G|\ast (\ny -\eta)_t $, we obtain that $ \Eob \left (\sup_{0\le s\le t} |N_s| \right ) = 0$ and, thus, the equality \eqref{Yq-2} implies
  \begin{equation}\label{E:RunY-q}
  \Eob \left ( \sup_{0\le s \le t} |\Y_s|^q\right ) \le \Eob  \left (  \sup_{0\le s \le t}F \ast \eta_s  \right ).
  \end{equation}  
 To estimate the right hand side above, we can now proceed as in the proof of the Kunita's inequalities for \levy-type stochastic integrals given in \cite[Theorem 4.4.23, p. 265]{a}.  Namely, using the definition of $F$ and Taylor's theorem  one can find $\theta \in (0,1)$ such that
 \begin{align*}
 F \ast \eta_t &= \int_0^t \diff s\int_{\Ro} \left (  \left | \Y_{s-}  + \V{y} \right |^q  +   \left | \Y_{s-}\right |^q +  \sum_{i=1}^n q\left | \Y_{s-}\right |^{q-2}  \Y_{s-}  y_i \right ) \nu_s (\diff \V{y}) \\
&\le \frac{1}{2}\int_0^t \diff s\int_{\Ro}    \sum_{1\le i,j \le n} \partial_{ij}^2 \left | \Y_{s-} + \theta \V{y} \right |^{q} |y_i y_j|  \nu_s (\diff \V{y}) \\
&\le \frac{1}{4} \int_0^t \diff s\int_{\Ro}  \sum_{1\le i,j \le n}  \left (  \delta_{ij} q \left | \Y_{s-}+ \theta \V{y}\right |^{q-2}
 + q(q-2)|\Y_{s-} +\theta y_i||\Y_{s-} + \theta y_j|\left | \Y_{s-}\right |^{q-4} \right )  |\V{y}|^2  \nu_s (\diff \V{y})  \\
 &\le  C \int_0^t \diff s\int_{\Ro}  \left ( |\Y_{s-}|+ | \V{y}|\right )^{q-2}  |\V{y}|^2  \nu_s (\diff \V{y}) \\
 &\le  C \int_0^t \diff s\int_{\Ro} \left ( \left | \Y_{s-}\right |^{q-2}+ \left | \V{y}\right |^{q-2} \right )  |\V{y}|^2  \nu_s (\diff \V{y})\\
 &\le  C \left \{ \int_0^t \diff s\int_{|\V{y}|\le 1} \left ( \left | \Y_{s-}\right |^{q-2}+ 1 \right )  |\V{y}|^2  \nu_s (\diff \V{y})  +  \int_0^t \diff s\int_{|\V{y}|\ge 1} \left ( \left | \Y_{s-}\right |^{q-2}+  |\V{y}|^{q-2} \right )  |\V{y}|^2  \nu_s (\diff \V{y})\right\} 
 \end{align*}
 Therefore, on taking expectations we get $
 \Eob  \left (  \sup_{0\le s \le t}F \ast \eta_s  \right ) \,\le\,  H_1 \, + \, H_2$,  where
  \begin{align*}
H_1 &:=  C \,\Eob \left \{ \sup_{0\le s\le t} \frac{1}{\gamma} \left | \Y_{s-}\right |^{q-2}  \int_0^t \diff s\int_{\Ro} \gamma |\V{y}|^2  \nu_s (\diff \V{y})  \right \}, \\
H_2 &:= C\,\Eob \left \{
  \int_0^t \diff s\int_{\Ro} \left | \V{y}\right |^{2}\vee \left | \V{y}\right |^{q}  \nu_s (\diff \V{y}) \right \},
\end{align*}
for any  $\gamma >1$. Using \holder's inequality with the conjugate values $p' = q/(q-2)$ and $q' = q/2$, it follows that
\begin{align*}
H_1  
 \le   &  \,\frac{C}{\gamma} \left \{\Eob \sup_{0\le s\le t}  \left | \Y_{s-}\right |^{q} \right \}^{1-2/q}  \left \{ \Eob \left ( \int_0^t \diff s\int_{\Ro}  \gamma |\V{y}|^2   \nu_s (\diff \V{y})  \right )^{q/2} \right \}^{2/q} \\
 \le   &  \,\frac{(q-2) C}{\gamma q } \Eob \left ( \sup_{0\le s\le t}  \left | \Y_{s-}\right |^{q} \right ) +  \frac{2 C \gamma^{q/2}}{q}  \Eob \left ( \int_0^t \diff s\int_{\Ro}  \gamma |\V{y}|^2   \nu_s (\diff \V{y})  \right )^{q/2} \\
  \le   &   \frac{2 C \gamma^{q/2}}{q}  \Eob \left ( \int_0^t \diff s\int_{\Ro}  \gamma |\V{y}|^2   \nu_s (\diff \V{y})  \right )^{q/2}
 \end{align*}
whenever  $\gamma$ is  chosen to satisfy  $(q-2) C <\gamma q $. Using the previous estimates into \eqref{E:RunY-q} yields
\begin{align*}
 \Eob \left ( \sup_{0\le s \le t} |\Y_s|^q\right ) \,\le \, C_1 \left \{ \Eob \left ( \int_0^t \diff s\int_{\Ro}   |\V{y}|^2   \nu_s (\diff \V{y})  \right )^{q/2} + \Eob \left (
  \int_0^t \diff s\int_{\Ro} \left | \V{y}\right |^{2}\vee \left | \V{y}\right |^{q}  \nu_s (\diff \V{y}) \right ) \right \},
 \end{align*}
for some positive constant $C_1> 0$. Applying the previous result to $\Xd$ and rearranging terms, we obtain the inequality  required in \eqref{Xd-estim}. 
 \end{proof}

%% file: V7-Proofs-II.tex
%%%%%%%%%%%%%%%%%%%%%%%%%%%%%%%%%%%%%%               %%%%% PROOF: VERIFICATION RESULTS
%%%%%%%%%%%%%%%%%%%%%%%%%%%%%%%%%%%%%%
%%%%%%%%%%%%%%%%%%%%%%%%%%%%%%%%%%%%%%               %%%%% PROOF: Stationarity of the control problem
%%%%%%%%%%%%%%%%%%%%%%%%%%%%%%%%%%%%%%

\subsection{Proof of Lemma \ref{L:TimeHom}} \label{P:TimeHom}   
\begin{proof}
Take $\aprx \in \Arx$, then a simple change of variables yields
\begin{align}\nonumber
J^{\alpha} (r,\V{x}) &= \mathbb{E} \left [ \int_0^{\infty} e^{-\int_0^u q\left ( \xpo_{m+r},\apo_{m+r} \right ) \diff m} f \left (  \xpo_{u+r}, \apo_{u+r} \right ) \diff u \Big | \xpo_r = \V{x} \right ]  \\ 
&= \mathbb{E} \left [ \int_0^{\infty} e^{-\int_0^u q\left ( \ypo_{m},\apoT_{m} \right ) \diff m} f \left (  \ypo_{u}, \apoT_{u} \right ) \diff u \Big | \ypo_0= \V{x} \right ] \, = \, J^{\tilde{\alpha}} (\V{x}),\label{Eq:JHom}
\end{align}
where $\left (\ypox, \apoTx \right ) = \left (\xpx_{\,\cdot\,+ r}, \apx_{\,\cdot\,+r} \right )$. Using that $\aprx \in \Arx$, the definition of $ \apoTx$ implies that $\apoTx \in \Ax$ with corresponding controlled process $\ypox$. Indeed, the construction of the pair $\left (\ypox, \apoTx \right ) $ is obtained from the corresponding canonical process $\left (\xpx, \apx \right )$ by shifting appropriately.   The validity of \eqref{D:Mphi} and \eqref{A3} follow straightforwardly from the corresponding conditions on  $\left (\xpx, \apx \right )$. Therefore,   taking the infimum over $\Malpha^{r,\V{x}} \in \Arx$  in the first equality of \eqref{Eq:JHom} and  then taking the infimum over $ \hat{\alpha} \in \Ax$ in the second equality of \eqref{Eq:JHom}, yields $v(r,\V{x}) = v (0,\V{x}) = V(\V{x})$, as required.
\end{proof}

%%%%%%%%%%%%%%%%%%%%%%%%%%%%%%%%%%%%%%               %%%%% PROOF: VERIFICATION RESULTS
%%%%%%%%%%%%%%%%%%%%%%%%%%%%%%%%%%%%%%
%%%%%%%%%%%%%%%%%%%%%%%%%%%%%%%%%%%%%%               %%%%% PROOF: Dynamic Programming Principle
%%%%%%%%%%%%%%%%%%%%%%%%%%%%%%%%%%%%%%
\subsection{Proof of Lemma \ref{L:DPP}} \label{P:DPP}   
\begin{proof} 
Denote by $W(\V{x})$ the right-hand side in \eqref{E:DPP}. 
Let us first prove the inequality $V\le W$. Let $(\xpx, \apx)$ be an admissible pair and define $\etat := \mathbb{P}^{\Malpha} \circ \left ( \xpx_t  \right )^{-1}$. Take $\epsilon > 0$ and  let $\bpa$ be an $\epsilon$-optimal control in  $\Apya$ with corresponding control process $\ypa$.   Define $\uxp :=\apx \oplus_t \bpa$ as given in \eqref{D:oplus}. Since,  by Lemma \ref{L:Concatenation}, $\uxp$ is also an admissible control in $\Ax$, there exists $(\Omega^u, \mathcal{F}^u, (\mathcal{F}_t^u), \mathbb{P}^u)$ in which the  corresponding admissible pair $(\zxp,\uxp)$ is defined. 

Set $ J^{u} (\cdot) := J(\V{Z}^{\V{u}^{\,\cdot,\, \oplus_t}},\V{u}^{\,\cdot,\, \oplus_t})$. Then  $V(\V{x}) \le J^u (\V{x})  = A +B$, where 
\begin{align}\nonumber
A &:= \Eu \left [ \int_0^t e^{-\int_0^s q\left (\zpxo_l, \upxo_l \right ) \diff l} f \left ( \zpxo_s, \upxo_s \right ) \diff s \,  \right ],\,\,  \\
B&:= \Eu \left [ \int_t^{\infty} e^{-\int_0^s q\left (\zpxo_l, \upxo_l \right ) \diff l} f \left ( \zpxo_s, \upxo_s \right ) \diff s \,  \right ]. 
\end{align}
 Since the control process $\zxp$ satisfies that $\zxp = \xpx$  on $[0,t)$, we have 
\begin{align}\label{Eq:A}
A = \Ex \left [ \int_0^t e^{-\int_0^s q\left (\xpo_l, \apo_l \right ) \diff l} f \left ( \xpo_s, \apo_s \right ) \diff s \,  \right ],
\end{align}
whereas the equality $\zxp = \ypa$ on $[t,\infty)$ and  properties of conditional expectation yield
\begin{align}
B &=  \Eu \left [ e^{-\int_0^t q\left (\zpxo_l, \upxo_l \right ) \diff l} \int_t^{\infty} e^{-\int_t^s q\left (\zpxo_l, \upxo_l \right ) \diff l} f \left ( \zpxo_s, \upxo_s \right ) \diff s \,  \right ] \nonumber \\
&=\Eu \left [ e^{-\int_0^t q\left (\zpxo_l, \upxo_l \right ) \diff l} \Eu \left ( \int_t^{\infty} e^{-\int_t^s q\left (\zpxo_l, \upxo_l \right ) \diff l} f \left ( \zpxo_s, \upxo_s \right ) \diff s \,\Big | \mathcal{F}^u_t  \right )  \right ] \nonumber\\
%&= \Eu \left [ e^{-\int_0^t q\left (\zpxo_l, \upxo_l \right ) \diff l} \,\Eu \left (\Eu  \left ( \int_t^{\infty} e^{-\int_t^s q\left (\zpxo_l, \upxo_l \right ) \diff l} f \left ( \zpxo_s, \upxo_s \right ) \diff s \,\Big | \mathcal{F}^u_t  \right ) \Big | \,\xpo_t   \right ) \right ] \nonumber \\
%&=\Eu \left [ e^{-\int_0^t q\left (\zpxo_l, \upxo_l \right ) \diff l} \,\Eu \left ( \int_t^{\infty} e^{-\int_t^s q\left (\zpxo_l, \upxo_l \right ) \diff l} f \left ( \zpxo_s, \upxo_s \right ) \diff s \Big | \,\xpo_t   \right ) \right ] \nonumber \\
&=\Eu \left [ e^{-\int_0^t q\left (\zpxo_l, \upxo_l \right ) \diff l} \,\mathbb{E}^{\bpa}  \left ( \int_t^{\infty} e^{-\int_t^s q\left (\ypa_l,\, \bpa (l) \right ) \diff l} f \left ( \ypa_s, \bpa(s) \right ) \diff s \Big | \,\xpo_t   \right ) \right ] \nonumber \\
&=\Ex \left [ e^{-\int_0^t q\left (\xpo_l, \apo_l \right ) \diff l} \,J \left ( \tilde{\V{Y}}^{\tilde{\Mbeta}^{\alpha}}, \tilde{\Mbeta}^{\alpha}   \right ) \right ],\label{Eq:DPPRev}
\end{align}
where $\left (\tilde{\V{Y}}^{\tilde{\Mbeta}^{\alpha}}_{\,\cdot\,}, \tilde{\Mbeta}^{\alpha}_{\,\cdot \,}   \right ) : = \left ( \ypa_{\,\cdot \, + t},\bpa ({\,\cdot \,+ t}) \right)$. Similarly as we did in the proof of Lemma  \ref{L:TimeHom}, we obtain that the process $\tilde{\Mbeta}^{\alpha}$ is admissible and belongs to $\mathcal{A}^p_{\eta_t^{\alpha}}$. Moreover, it is not difficult to see that $\tilde{\Mbeta}^{\alpha}$ is an $\epsilon$-optimal policy as well. Hence,  \eqref{Eq:A} and \eqref{Eq:DPPRev},together with   the equality $\gamma_{t}^{\apx} := \exp\{-\int_0^s q\left (\xpo_l, \apo_l \right ) \diff l\}$, imply that 
\begin{align}
V(\V{x})\,<\,A +\Ex \left [ e^{-\gamma_{t}^{\apx} } \,\,  V\left (\xpo_t   \right ) \right ]+ \epsilon \, = \,W(\V{x}) + \epsilon.
\end{align}
%In the fourth equality we used the homogeneity in time of the payoff function and the last equality followed from the $\epsilon$-optimality of $\beta_t^{\alpha}$. Recall now  
Letting $\epsilon \downarrow 0$ and then taking the infimum over all policies in $\Ax$ yield the desired inequality  $V \le W$. 

To prove the reverse inequality, take an $\epsilon$-optimal policy $\apx \in \Ax$. Again, properties of conditional expectation yield
\begin{align}
V(\V{x}) +\,\epsilon \, &> \, J^{\alpha} (\V{x}) \\
& = \Ex \left [\left ( \int_0^t +  \int_t^{\infty} \right ) e^{-\gamma_{s}^{\apo} } f\left (\xpo_s,\apo_s \right ) \diff s \right ] \nonumber \\
&=  \Ex \left [ \int_0^t  e^{-\gamma_{s}^{\apo} } f\left (\xpo_s,\apo_s \right ) \diff s \right ] 
 + \Ex \left [  e^{-\gamma_{t}^{\apo} } \Ex \left [ \int_t^{\infty} e^{-\int_t^s q\left (\xpo_l, \apo_l \right ) \diff l} f\left (\xpo_s,\apo_s \right ) \diff s  \Big | \mathcal{F}_t^{\alpha} \right ]  \right ]\nonumber \\
 &= \Ex \left [ \int_0^t  e^{-\gamma_{s}^{\apo} } f\left (\xpo_s,\apo_s \right ) \diff s \right ] 
 + \Ex \left [  e^{-\gamma_{t}^{\apo} } \Ex \left [ \int_t^{\infty} e^{-\int_t^s q\left (\xpo_l, \apo_l \right ) \diff l} f\left (\xpo_s,\apo_s \right ) \diff s  \Big | \xpo_t \right ]  \right ] \nonumber \\
&= \Ex \left [ \int_0^t  e^{-\gamma_{s}^{\apo} } f\left (\xpo_s,\apo_s \right ) \diff s \right ] 
 + \Ex \left [  e^{-\gamma_{t}^{\apo} } J \left ( \tilde{\V{X}}^{\tilde{\Malpha}}, \tilde{\Malpha} \right)\right ]   \nonumber \\
 &\ge \Ex \left [ \int_0^t  e^{-\gamma_{s}^{\apo} } f\left (\xpo_s,\apo_s \right ) \diff s \right ] 
 + \Ex \left [  e^{-\gamma_{t}^{\apo} } V \left ( \xpo_t  \right)\right ] = W (\V{x}),   \nonumber
\end{align}
where $\left (\tilde{\V{X}}^{\tilde{\Malpha}}_{\,\cdot\,}, \tilde{\Malpha}_{\,\cdot \,}   \right ) : = \left ( \xpa_{\,\cdot \, + t},\bpa ({\,\cdot \,+ t}) \right)$ is an admissible pair with $\tilde{\Malpha} \in \mathcal{A}^p_{\eta^{\alpha}_t}$,  $\eta^{\alpha}_t$ being the law of $\xpx_t$. By letting $\epsilon \downarrow 0$, we get $V \ge W $, as required. 
\end{proof}
%%%%%%%%%%%%%%%%%%%%%%%%%%%%%%%%%%%%%%               %%%%% PROOF: VERIFICATION RESULTS
%%%%%%%%%%%%%%%%%%%%%%%%%%%%%%%%%%%%%%
%%%%%%%%%%%%%%%%%%%%%%%%%%%%%%%%%%%%%%               %%%%% PROOF: Dynamic Programming Principle - Submartingale inequality
%%%%%%%%%%%%%%%%%%%%%%%%%%%%%%%%%%%%%%
\subsection{Proof of Lemma  \ref{P0-S}}. \label{P:P0-S}
\begin{proof}%(of Lemma \ref{P0-S})
We need to prove that for any admissible $\apx \in \Ax$,  the inequality $\Ex \left [ S_t^{V,\Malpha} \Big |  \mathcal{F}^{\alpha}_s  \right ] \ge S_s^{V,\Malpha}$ holds for all $t \ge s \ge 0$.  

Define $\theta_{\alpha}(r,t) := e^{-\int_r^t q\left (\xpx_s, \apx_s \right )\diff s}$ for $0\le r\le t$. Notice that  $\theta_{\alpha}(0,t) = e^{-\gamma_t^{\apx}}$. Now, fix some arbitrary admissible control $\apx$, then
\begin{align}
\Ex \left [ S_t^{V,\Malpha} \Big | \mathcal{F}^{\alpha}_s \right ] &\,\,\,= \,\,\,\Ex \left [  \left ( \int_0^s  +  \int_s^t\right ) \theta_{\alpha}(0,s) f\left(\xpo_r, \apo_r\right) \diff r + \theta_{\alpha}(0,t) V\left (\xpo_t \right ) \right ] \nonumber \\
&\,\,\,= \,\,\,  \int_0^s   \theta_{\alpha}(0,r) f\left(\xpo_r, \apo_r\right) \diff r \,\,\,+\nonumber \\
& \quad \quad \quad + \,\,\,  \theta_{\alpha}(0,s) \mathbb{E} \left [  \int_s^t  \theta_{\alpha}(s,r)  f\left(\xpo_r, \apo_r\right) \diff r + \theta_{\alpha}(s,t) V\left (\xpo_t \right ) \,\,\Big | \,\, \mathcal{F}^{\alpha}_s\right ] \nonumber \\
&\,\,\,= \,\,\,  \int_0^s   \theta_{\alpha}(0,r) f\left(\xpo_r, \apo_r\right) \diff r \,\,\,+\nonumber \\
& \quad \quad \quad + \,\,\,  \theta_{\alpha}(0,s) \mathbb{E} \left [  \int_s^t  \theta_{\alpha}(s,r)  f\left(\xpo_r, \apo_r\right) \diff r + \theta_{\alpha}(s,t) V\left (\xpo_t \right ) \,\,\Big | \,\, \xpo_s\right ]. \label{V-subF}
\end{align}
The last equality follows by conditioning on $\xpx_t$ and then by using the law of iterated conditional expectation.  Let $\hat{\Malpha}$ be the restriction of $\apx$ on $[s,\infty)$. Then, it is not difficult to see that $\hat{\Malpha}$ is an admissible policy in $\mathcal{A}_{s,\eta_s^{\alpha}}$ where $\eta_s^{\alpha} := \Pa \circ \left ( \xpx_s\right )^{-1}$, i.e.  $\eta_s^{\alpha}$ is the law of $\xpx_s$. Hence, by the DPP (Lemma \ref{L:DPP}) we obtain that
\[ \mathbb{E} \left [  \int_s^t  \theta_{\alpha}(s,r)  f\left(\xpx_r, \apx_r\right) \diff r + \theta_{\alpha}(s,t) V\left (\xpx_t \right ) \,\,\Big | \,\, \xpx_s\right ] \ge V\left ( \xpx_s \right ),\]
which plugged into  \eqref{V-subF} yields
\begin{align}\label{Dpp0}
\Ex \left [ S_t^{V,\Malpha} \Big | \mathcal{F}^{\alpha}_s \right ] &\,\,\,= \,\,\,  \int_0^s   \theta_{\alpha}(0,r) f\left(\xpo_r, \apo_r\right) \diff r + \theta_{\alpha}(0,s)V\left ( \xpx_s \right ) = S_s^{V,\apx},
\end{align}
which in turn implies that $S^{V,\apx}$ is a $\Pa$-submartingale. On the other hand, if $\apx$ is an optimal policy,  then \eqref{Dpp0} ensures an equality in \eqref{Dpp0} which then yields $\Ex \left [ S_t^{V,\Malpha} \Big | \mathcal{F}^{\alpha}_s \right ] = S_s^{V,\apx}$. Therefore,   $S^{V,\apx}$ is a $\Pa$-martingale for any optimal policy with finite payoff.
\end{proof}
%%%%%%%%%%%%%%%%%%%%%%%%%%%%%%%%%%%%%%               %%%%% PROOF: VERIFICATION RESULTS
%%%%%%%%%%%%%%%%%%%%%%%%%%%%%%%%%%%%%%
%%%%%%%%%%%%%%%%%%%%%%%%%%%%%%%%%%%%%%               %%%%% PROOF: Verification Result 1: Submartingale conditions
%%%%%%%%%%%%%%%%%%%%%%%%%%%%%%%%%%%%%%
 \subsection{Proof of Lemma  \ref{VT0-SubPhi}}\label{P:VT0-SubPhi}
 
\begin{proof}
 $i)$ Take an arbitrary policy $\apx \in \Ax$. Assume that $\Jx < + \infty$  as, otherwise, the inequality $\phi(x) \le \Jx$ follows immediately.  Since $S^{\phi,\apx}$ is a submartingale by assumption, $\phi(\V{x})\, \le\, \Ex \left [ S_t^{\phi,\apo}\right ]$. Therefore, 
 \begin{equation} \label{phi-t}
\phi(\V{x}) \le \Ex \left [ \int_0^t e^{-\gamma_{s}^{\apo}}f(\xpo_s, \apo_s) \diff s\right ] + \Ex \left [ e^{-\gamma_{t}^{\apo}} \phi \left (\xpo_t\right) \right ].
\end{equation}
Note that, as $t\to \infty$, the first expectation in \eqref{phi-t}  converges  to $\Jx$ (by the MCT). Hence, letting $t \to \infty$ in \eqref{phi-t} and using the transversality condition \eqref{E:TC}  imply that $\phi (\V{x}) \le \Jx$.  We have used the fact that $\liminf (a_n  + b_n) = \liminf a_n + \liminf b_n$ whenever one of the sequences is convergent.  

$ii)$ Take $\hat{\Malpha}^{\V{x}}$ be an optimal policy. By assumption (\textbf{SC}), $\phi(\V{x}) = \mathbb{E}_{\V{x}}^{\hat{\Malpha}} \left [ S_t^{\hat{\apo},\phi}\right ]$ and, thus,  the same arguments above yield $ \phi (\V{x})  = J^{\hat{\Malpha}} (\V{x})$,  which in turn implies that  $V(\V{x}) =\phi(\V{x})$, as required.

$iii)$ Using statement $i)$ and the definition of $V$, as well as  condition (\textbf{nC}), it follows that
$\phi (\V{x}) \le V(\V{x}) \le J^{\alpha^n} (\V{x}) < \phi (x) + \frac{1}{n}$, for all $n \ge 1$. Hence, letting $n\to \infty$, we get $\phi = V$, as required. 
 \end{proof}
%%%%%%%%%%%%%%%%%%%%%%%%%%%%%%%%%%%%%%               %%%%% PROOF: VERIFICATION RESULTS
%%%%%%%%%%%%%%%%%%%%%%%%%%%%%%%%%%%%%%
%%%%%%%%%%%%%%%%%%%%%%%%%%%%%%%%%%%%%%               %%%%% PROOF: Verification Result 2
%%%%%%%%%%%%%%%%%%%%%%%%%%%%%%%%%%%%%%
\subsection{Proof of Theorem \ref{VT0}} \label{P:VT0}
\begin{proof} 

$i)$    Take an arbitrary policy $\apx \in \Ax$. Assume that $\Jx < + \infty$ as, otherwise, the inequality $\phi(x) \le \Jx$ follows immediately. Then there exists  a complete, filtered probability space $(\Omega^{\alpha}, \mathcal{F}^{\alpha}, (\Ft^{\alpha}), \mathbb{P}^{\alpha})$ and an $(\Ft^{\alpha})$-adapted pair process $(\xpx,\apx)$  defined on it such that the process   $\xpx:=(X^1,\ldots, X^n)^T$,  started at $\V{x}=(x_1,\ldots, x_n)$,  is \cad.  

Let $\phi \in C^2(\rd)$ be as in the statement and set
 \begin{equation}\label{Sax00}
  S_{t}^{\phi,\apx} := \int_0^{t } e^{-\gamma_{s}^{\apx}}\,
 f(\xpx_s,  \apx_s) \diff s + e^{-\gamma_{t}^{\apx}} \phi \left (\xpx_t\right),
  \end{equation}

     Since $\phi$ has polynomial growth of degree $q \le \max \{2,p\}$, Lemma \ref{L2-f2}  guarantees that  \[ \phi (\xpx_t)
  =  \int_0^t \Lasx \phi (\xpx_{s-}) \diff s  + M_t^{\phi}, \] 
  for some local martingale  $M^{\phi} = (M_t^{\phi})_{t \in \rp}$. Hence, the integration by parts formula \cite[Corollary 2, p. 68]{Protter} yields 
\begin{align}\label{ephik00}
e^{-\gamma_{t}^{\apx}} \phi \left (\xpx_t\right)&=  \phi(\V{x}) + \int_{0}^{t} e^{-\gamma_{s}^{\apx}}  \left (  \Lasx\phi (\xpx_{s-}) - q \left (\xpx_{s}, \apx_{s} \right ) \phi (\xpx_{s-}) \right ) \diff s+ N_t, \end{align}
where $N=(N_t)_{t \in \Rp}$ is the local martingale given by $N_{t} := \int_{0}^{t} e^{-\gamma_{s}^{\apx}} \diff M^{\phi}_s$.  It is not difficult to see that, for $\V{a} \in A$, the mapping $\V{x} \mapsto L^\V{a}\phi (\V{x})$ is continuous (this follows from the fact that  $\phi \in C^2 (\rd)$ has polynomial growth of degree $q \le p$ and each $\nu \in \mathcal{M}_p$ has finite second moments in the unitary ball $B_1$ and $p$th-moments outside $B_1$). Thus, since  the paths of $\xpx$ are \cad  (so they have at most a countable number of discontinuities), the integral in \eqref{ephik00} is $\as$ equal to the one but with $s$ instead of $s-$.  

Substituting \eqref{ephik00} into \eqref{Sax00} yields
\begin{equation}\label{Sax-f0}
 S_{t}^{\phi,\apx} := \phi(x) +\int_0^{t } e^{-\gamma_{s}^{\apx}}\left (\, f (\cdot,  \apx_s) + (L^{\apx_s} \phi) (\cdot) -  q \left (\cdot, \apx_s \right ) \phi (\cdot)\, \right ) (\xpx_s)\diff s + N_t.
  \end{equation}
Thus the  process $S^{\phi,\apx}$ is a local submartingale as the integral term in \eqref{Sax-f0} is non-negative thanks to assumption (\textbf{HJB}).   
  Let  $\{T_m\}_{n \ge 0}$ be  a localising sequence for the local martingale $N$. 
 Then, for each $m$, the stopped process $(S_{t\wedge T_m}^{\phi,\apx})_{t \in \Rp}$ is a submartingale and thus $   \phi(\V{x}) \le \Ex \left [S_{t\wedge T_m}^{\phi,\apo} \right ]$. Hence,
\begin{equation} \label{phi-tm}
\phi(\V{x}) \le \Ex \left [ \int_0^{t\wedge T_m}e^{-\gamma_{s}^{\apo}} \, f(\xpo_s, \apo_s) \diff s\right ] + \Ex \left [ e^{-\gamma_{t\wedge T_m}^{\apo}} \phi \left (\xpo_{t\wedge T_m}\right) \right ].
\end{equation}
The uniform integrability condition (\textbf{UI}) yields $\lim_{m\to \infty}\Ex\left [ e^{-\gamma_{t\wedge T_m}^{\apo}} \phi \left (\xpo_{t\wedge T_m}\right) \right ]=  \Ex\left [ e^{-\gamma_{t}^{\apo}} \phi \left (\xpo_t\right) \right ]$, whereas the MCT  implies (by letting $m\to \infty$) that
\begin{equation} \label{phi-t}
\phi(\V{x}) \le \Ex \left [ \int_0^t e^{-\gamma_{s}^{\apo}} f(\xpo_s, \apo_s) \diff s\right ] + \Ex \left [ e^{-\gamma_{t}^{\apo}} \phi \left (\xpo_t\right) \right ].
\end{equation}
Hence, letting $t \to \infty$ in \eqref{phi-t}, the MCT and the transversality condition  (\textbf{TC}) yield $\phi (\V{x}) \le \Jx$, as required.

$ii)$ Suppose now that, for every $\V{x} \in \rd$,  there exists an admissible pair  $(\V{X}^{\hat{\Malpha}^{\V{x}}}, \hat{\Malpha}^\V{x} )$ such that the triplet $(\hat{\Msigma}, \hat{\nu}, \hat{\Mmu})$  satisfies \eqref{HJB-op}. To prove the optimality of $\hat{\Malpha}^{\V{x}}$, it remains to prove that $ \phi (\V{x})  = J^{\hat{\Malpha}}(\V{x})$.  Similar calculations than above imply that the equality
 \begin{equation}\label{D:S_Bellman00}
 S_t^{\phi,\hat{\Malpha}^{\V{x}}} = \int_0^t e^{-\gamma_s^{\hat{\Malpha}^\V{x}}}f(\V{X}_s^{\hat{\Malpha}^\V{x}}, \hat{\Malpha}^\V{x}_s) \diff s + e^{-\gamma_t^{\hat{\Malpha}^\V{x}}} \phi \left (\V{X}^{\hat{\Malpha}^\V{x}}_t\right) ,\quad t \in \Rp,
\end{equation}
can be rewritten as
\begin{align*}
S_t^{\phi,\hat{\Malpha}^{\V{x}}} &= \phi(\V{x}) +\int_0^{t} e^{-\gamma_{s}^{\hat{\Malpha}^{\V{x}}}}  \left [f (\V{X}_s^{\hat{\Malpha}^\V{x}}, \hat{\Malpha}^\V{x}_s) + L^{\hat{\Malpha}^\V{x}_s} \phi (\V{X}_{s-}^{\hat{\Malpha}^\V{x}}) -q \left (\V{X}_{s}^{\hat{\Malpha}^\V{x}},  \hat{\Malpha}^\V{x}_s \right ) \phi(\V{X}_{s-}^{\hat{\Malpha}^\V{x}}) \right] \diff s + \hat{M}_t^{\phi},
\end{align*}
where $\hat{M}^{\phi}$ is some local martingale. 
The \cad property of  $\V{X}_s^{\hat{\Malpha}^\V{x}}$ and the fact that   $\phi$ solves  \eqref{HJB-op} ensure the equality $S_t^{\phi,\hat{\Malpha}^{\V{x}}}= \phi(\V{x})  + \hat{M}_t^{\phi}$, which then implies that  $S_t^{\phi,\hat{\Malpha}^{\V{x}}}$ is a local martingale. By repeating the same arguments as before (localising and taking the corresponding limits), we obtain the equality $ \phi (\V{x})  = J^{\hat{\Malpha}}(\V{x})$, which implies both that  $\hat{\alpha}^\V{x}$ is optimal and that  $\phi(\V{x})$ is the value function.
   \end{proof}

%%%%%%%%%%%%%%%%%%%%%%%%%%%%%%%%%%%%%%               %%%%% PROOF: VERIFICATION RESULTS
%%%%%%%%%%%%%%%%%%%%%%%%%%%%%%%%%%%%%%
%%%%%%%%%%%%%%%%%%%%%%%%%%%%%%%%%%%%%%               %%%%% PROOF: Two lemmas: a)  polynomial growth of V, and b) condition TC .
%%%%%%%%%%%%%%%%%%%%%%%%%%%%%%%%%%%%%%

\subsection{Proof of Lemma \ref{P:lemma00}}\label{Pr:lemma00}

\begin{proof}
Let $C>0$, $\V{a}_0 = ( \Msigma, \nu, \Mmu ) \in A$, $p\ge 2$ and $f$ be as in the statement. Take the   admissible pair $(\V{X}^{\tilde{\Malpha}^x}, \tilde{\Malpha}^x)$, where $\V{X}^{\tilde{\Malpha}^x}$ is the \levy process, starting at $\V{x}\in \rd$, corresponding to the constant policy $\tilde{\Malpha}_t^x =\V{a}_0$ for all $t\in \rp$ (see Remark \ref{CteP}). The  assumption $|f(\V{x},\V{a}_0)| \le C (1 + |\V{x}|^q)$ for all $\V{x} \in \rd$ and the definition of $J^{\tilde{\Malpha}} (\V{x})$, yield 
\begin{equation}\label{EX2}
J^{\tilde{\Malpha}} (\V{x}) \le C \left ( \frac{1}{\delta} +\int_{0}^{\infty} e^{-\delta t} \Ex|\V{X}^{\tilde{\Malpha}^x}_t|^q   \diff t \right ),
\end{equation}
where $\delta >0$ is the lower bound of the function $q$ in the discount factor $\gamma_{t}^{\apx}$.  
Since  $\V{X}^{\tilde{\Malpha}^x}$ is a \levy process with jump intensity measure in $\mathcal{M}_p$ (recall definition in \eqref{D:Lm}), Proposition \eqref{qth-moments} ensures that, for each $t\in \rd$,  $\mathbb{E} [\left |\V{X}^{\tilde{\Malpha}^x}_t \right|^q ] \le K (|\V{x}|^q + t^q)$,  for some  $K>0$  depending on the fixed constants  $\mu$, $\sigma$ and the measure $\nu$.  Plugging the previous expression into \eqref{EX2} implies that $J^{\tilde{\Malpha}} (\V{x}) \le \tilde{C} (1 + |\V{x}|^q)$ for some constant $\tilde{C} > 0$. Therefore, taking the infimum over all admissible policies implies, by definition of the value function, that $V(\V{x}) \le \tilde{C} (1 +| \V{x}|^q)$, as required.
\end{proof}

\subsection{Proof of Lemma \ref{P:lemma01}}\label{Pr:lemma01}
We first recall the following.
\begin{lemma}\label{A-Lemma2}
If $b =\liminf_{t\to \infty} f(t)$, then for all $\epsilon >0$, there exists $t_0$ such that $f(t) > b - \epsilon$ for all $t \ge t_0$. 
\end{lemma}    
\begin{proof}
By definition, $b = \liminf_{t\to \infty} f(t) := lim_{t\to \infty} A_t$, where $A_t := \inf \{ f(s)\,:\, s \ge t\}$. Thus, by definition of limit, for all $\epsilon > 0$, there exists $t_0$ such that $|A_{t_0} - b| < \epsilon$, thus $b - \epsilon < A_{t_0}$. Since $A_t$ is an increasing sequence in $t$, then $b - \epsilon < A_{t_0} \le A_t$ for all $t \ge t_0$. Also, by definition of $A_t$, it follows that $A_t \le f(t)$, which in turns implies $b - \epsilon < f(t)$ for all $t\ge t_0$, as required.
\end{proof}

\begin{proof} (of Lemma \ref{P:lemma01})
Define the mapping $ g^{h,\apx}: t\mapsto \Ex \left [ e^{-\gamma_{t}^{\apo}} h(\xpo_{t}) \right ]$ for any  nonnegative function $h$ on $\rd$. 
Take $\phi \in C^2(\rd)$ and $f$ as in the statement.  Suppose that \eqref{E:TC} does not hold. That is, $\Jx < \infty$ and $\liminf_{t \to \infty} g^{\phi,\apo}_t = \gamma$ for some constant $\gamma > 0$. Then  \[ +\infty > \Jx \ge \Ex \left [\int_{t_0}^{\infty} e^{-\gamma_{t}^{\apo}} f(\xpo_t, \apo_t) \diff t \right ] \ge c \left (  \frac{e^{-b t_0}}{b} +  \int_{t_0}^{\infty} g^{|\cdot|^p, \apx}_t \diff t \right), \]
where $b >0$ is the upper bound of the function $q$ defining the discounting factor  ${\gamma_{t}^{\apx}}$. Notice the use of the  lower bound of $|f(x,a)|$ as well as Tonelli's theorem to interchange the integral and the expectation in the right hand side above. Since $\phi $ is of polynomial growth of degree $p\ge 2$, there exists $C > 0$ such that  $ \int_{t_0}^{\infty}  g^{\phi,\apx}_t \diff t  \le C ( \frac{e^{-b t_0}}{b}  + \int_{t_0}^{\infty} g^{|\cdot|^p,\apx}_t \diff t) $. Moreover, by Lemma \ref{A-Lemma2},  for $\epsilon = \gamma/2$,  there exists $t_0 \ge 0$ such that $g^{\phi,\apx}_t > \gamma/2$   for all $t \ge t_0$, and this implies that $ \int_{t_0}^{\infty}  g^{\phi,\apx}_t \diff t$ is not finite, which in turn implies  (by the inequalities above) that the payoff function $\Jx$ is not finite. The latter yields a contradiction and  we thus conclude that \eqref{E:TC} holds.
\end{proof}

%\begin{proof}
%Take $\phi \in C^2(\rd)$ and $f$ as in the statement. By assumption, there exist positive constants $c,C$ such that $c|\V{x}|^p \le f(\V{x}, \cdot) \le C|\V{x}|^p$. Hence, \[c\Ex \left [\int_{t_0}^{\infty} e^{-qt} |\V{X}_t^{\ap}|^p \diff t\right ]\le J(\xp,\ap) \le C\Ex \left [\int_{0}^{\infty} e^{-qt} |\V{X}_t^{\ap}|^p \diff t\right ].\]Let us suppose that \eqref{E:TC} does not hold. That is, $J(\xp,\ap) < \infty$ and $\liminf_{t \to \infty} g_{\ap}(t) = \gamma > 0$. Therefore,  for $\epsilon = \gamma/2$,  there exists $t_0 \ge 0$ such that $g_{\ap}(t) > \gamma/2$   for all $t \ge t_0$, but the latter implies that $ \mathbb{E}_\V{x} \left [\int_{t_0}^{\infty} e^{-qt} |\V{X}_t^{\ap}|^p \diff t\right ]$ is not finite, which in turn implies  (as $\phi $ is of polynomial growth of degree $p\le 2$) that the payoff function $J(\xp,\ap)$ is not finite. The latter yields a contradiction and thus \eqref{E:TC} holds.
%\end{proof}

%%%%%%%%%%%%%%%%%%%%%%%%%%%%%%%%%%%%%%               %%%%% PROOF: VERIFICATION RESULTS
%%%%%%%%%%%%%%%%%%%%%%%%%%%%%%%%%%%%%%
%%%%%%%%%%%%%%%%%%%%%%%%%%%%%%%%%%%%%%               %%%%% PROOF: Verification Result 3
%%%%%%%%%%%%%%%%%%%%%%%%%%%%%%%%%%%%%%
 \subsection{Proof of Theorem \ref{VT-EFinite} }\label{P:VT-EFinite}

\begin{proof} The proof follows the same lines as the one for Theorem \ref{VT0}.  The only change is made at justifying the limiting step in  \eqref{phi-tm} to obtain the inequality \eqref{phi-t}, which  now  is ensured by the DCT and the finiteness of the expectation of the running maximum of $\left | \xpx \right |^q$. The latter assertion holds true due to Proposition \ref{qth-moments}.
\end{proof}

%% file: V7-Proofs-III.tex
\subsection{Proof of Lemma \ref{Example2}}

\begin{proof} 

$i)$ The convexity of $f$ yields
\begin{align*}
\theta \psi (x) + (1-\theta) \psi (y) &= \mathbb{E} \left (\int_0^{\infty} e^{-(q+1) t} \left [  \theta f(x + B_t) + (1-\theta) f(y + B_t) \right ] \diff t \right ) \\
&\ge   \mathbb{E} \left ( \int_0^{\infty} e^{-(q+1) t}   f ( \theta x +  (1-\theta) y + B_t) \diff t \right ) \\
 &= \psi (\theta x +  (1-\theta) y ),
\end{align*}
establishing that $\psi$ is convex. Symmetry follows from the symmetry of $f$ and of the normal distribution. Finally, convexity and symmetry show that $\psi (x) = \frac{1}{2} \psi (x) + \frac{1}{2} \psi (-x) \ge \psi (0)$, establishing that zero gives the global minimum of $\psi$.

$ii)$ The polynomial growth of $f$ ensures that $\psi(x)$ is finite for each $x\in \rr$ and, further, it implies that $\psi$ has the same  polynomial growth. It is not difficult to see that $
\psi$ satisfies
\begin{equation}\label{Eq:Psi}
\frac{1}{2} \psi'' - (q+1) \psi + f = 0.
\end{equation}
We will now show that $\phi := \psi + c$ solves the HJB equation:
\begin{equation}\label{Eq:hjb2}
\inf_{\nu \in \mathcal{M}_{\le 1}} \left \{ \frac{1}{2}h ''(x) + \int \left (h(x+y) - h (x) \right )\nu(\diff y)  -q h(x) + f(x) \right \} = 0.
\end{equation}
Note that, for each $x \in \rr$ and $\nu \in \mathcal{M}_{\le 1}$, 
\begin{align}\nonumber
\frac{1}{2}\phi '' (x)+ \int \left (\phi (x+y) - \phi (x) \right )\nu(\diff y)  -&q \phi (x) + f(x)  \\ \nonumber &\ge \frac{1}{2}\psi '' (x) + (\psi (0) - \psi (x) )\nu(\rr)  -q \psi (x) - \psi (0) + f (x) \\ \nonumber
 &\ge  \frac{1}{2}\psi '' (x) + (\psi (0) - \psi (x)) -q \psi  (x) - \psi (0)+ f (x) \\&= \frac{1}{2} \psi '' (x) -\psi (x) (1+q) +f(x) \ge 0, \label{Ineq:Ex2}
\end{align}
where we used that $\psi (z) \ge \psi (0)$ for all  $z \in \rr$ (because of statement $i)$ above), and  $cq = \psi (0)$ by definition. The last inequality in \eqref{Ineq:Ex2} follows from \eqref{Eq:Psi}. Hence, taking the infimum over all $\nu \in \mathcal{M}_{\le 1}$ establishes that $\phi$ satisfies condition (\textbf{HJB}).
Observe now that  \eqref{UniformB2} and Proposition  \eqref{qth-moments} imply that $\mathbb{E} \left [ \left |X_{t}^{\alpha^x}\right |^p\right ] \le C\, (|x|^p + t^p)$ for some positive  constant $C = C(\beta,p)$, which then implies (\textbf{TC}). Furthermore,  $\sup_{0\le s \le t}\left |X_{s}^{\alpha^x}\right |^p \in L^1(\Pa)$ which establishes condition (\textbf{UI}).  Moreover, since the infimum \eqref{Eq:hjb2} is attained at $\hat{\nu} = \delta_{-x}$ for each $x \in \rr$, Theorem  \eqref{VT-EFinite}  establishes the equality $V = \phi$, as required. 

\Mqo{Existence of $X^x$ is clear (Levy type process with bounded coefficients)!}

 %As for condition (\textbf{UI}), due to the polynomial growth of $\psi$, say $p\ge 2$, it is enough to prove that, for each $t\ge 0$,  $ \sup_{0\le s \le t} \left |X_{t}^{\alpha^x}\right |^p \in L^1 (\mathbb{P}^{\alpha})$, which  holds by Proposition \eqref{qth-moments} and assumption \eqref{UniformB2}. Finally, one can verify using again Proposition  \eqref{qth-moments} that $\mathbb{E} \left [ \left |X_{t}^{\alpha^x}\right |^p\right ] \le C\, (|x|^p + t^p)$ for some positive  constant $C = C(\beta,p)$, which then guarantees the validity of condition (\textbf{TC}). 
\end{proof}

%%%%%%%%%%%%%%%%%%%%%%%
%%%%%%%%%%%%% PROOFS: 
%%%%%%%%%%%%%%%%%%%.  EXAMPLE 3
%%%%%%%%%%%%%%%%%%%%%%%

\subsection{Proof of Theorem \ref{Example3}.}
We will need the following result.
\begin{theorem}\label{T:Phi-Up}
Define the operator $G^q  h  =  \frac{1}{2} h'' - q h$. Then, $\phi$ satisfies 
 \begin{equation}\label{Eq:HJB-3a}
  G^q \phi+ f  = 0, \quad \text{ on } \,\,\, (0,\hat{b}),
 \end{equation}
 and 
 \begin{equation}\label{Eq:HJB-3b}
 G^q \phi + f  + \kappa - (\phi - \phi(0)) = 0, \quad \text{ on } \,\,\, (\hat{b}, \infty),
 \end{equation}
 and $\phi$ is increasing on $\rp$.
 \end{theorem}
 \begin{proof}(of Theorem \ref{T:Phi-Up})
 We first prove that $\phi$ satisfies \eqref{Eq:HJB-3a}-\eqref{Eq:HJB-3b}. Let $B^x$ be a Brownian motion started at $x$. Define the stopping times $\tau_b := \inf \{t \ge 0 : |B_t^{x}| = b \}$ and
 $\tau_0 := \inf \{t \ge 0 : B_t^{b,x} = 0 \}$, for each $x \in (-b,b)$. Observe that  $\tau_b = \tau_0$ in distribution.
 
  Using the strong Markov property of the Brownian motion,  $\phi$ in \eqref{Eq:Phi-b} can be rewritten as
 \begin{align*}
\phi (x) &= \mathbb{E}_x \left [  \int_0^{\tau_0} e^{-qt} \left ( f(B_t^{b,x}) + \kappa 1_{|B_t^{b,x}| \ge b} \right ) \diff t\right ] + \phi(0) \mathbb{E}_x \left [  e^{-q \tau_0}  \right ], \\
 &=\mathbb{E}_x \left [  \int_0^{\infty} e^{-\left(qt+\int_0^t 1_{\{|B_s^{b,x}| \ge b\}}\diff s\right )} \left ( f(B_t^{b,x}) + \kappa 1_{|B_t^{b,x}| \ge b} \right ) \diff t\right ] + \phi(0) \mathbb{E}_x \left [  e^{-q \tau_0}  \right ]
 \end{align*}
Using that
 \begin{align}
 \mathbb{E}_x \left [  e^{-q \tau_0}  \right ] =   \mathbb{E}_x \left [ \int_0^{\infty} 1_{\{|B_s^{b,x}| \ge b\}} e^{-\left(qt+\int_0^t 1_{\{|B_s^{b,x}| \ge b\}}\diff s\right )}  \diff t\right ], \nonumber
 \end{align}
 it follows that 
  \begin{align*}
\phi(x) &= \mathbb{E}_x \left [  \int_0^{\infty} e^{-\left(qt+\int_0^t 1_{\{|B_s^{b,x}| \ge b\}} \diff s\right )} \left ( f(B_t^{b,x}) + \right( \kappa + \phi(0)\left )1_{|B_t^{b,x}| \ge b} \right ) \diff t\right ],\quad x\in \rr
\end{align*}
and
  \begin{align*}
\phi(0) &= \frac{\mathbb{E}_x \left [  \int_0^{\infty} e^{-\left(qt+\int_0^t 1_{\{|B_s^{b,x}| \ge b\}}\right )} \left ( f(B_t^{b,x}) +  \kappa 1_{|B_t^{b,x}| \ge b} \right ) \diff t\right ]}{1- \mathbb{E}_x \left [  \int_0^{\infty} e^{-\left(qt+\int_0^t 1_{\{|B_s^{b,x}| \ge b\}}\right )} 1_{\{|B_s^{b,x}| \ge b\}} \right ]  }.\end{align*}

 Therefore, the stationary Feynman-Kac formula implies that $\phi$ solves
 \[ \frac{1}{2} \phi''(x) - (q + 1_{|x| \ge b}) \phi(x) + f(x) + (\kappa +\phi(0))1_{|x| \ge b} = 0,\]
 that is, 
 \[
 \left \{
 \begin{array}{ll}
\frac{1}{2} \phi''(x) - q  \phi(x) + f(x)  = 0 & x\in (-b,b)\\
\frac{1}{2} \phi''(x) - (q + 1) \phi(x) + f(x) + \kappa +\phi(0) = 0 & x\in (-b,b)^c 
 \end{array}
 \right .
 \]
The existence of $\hat{b}$ can  be justified as follows. Observe that $\mathbb{E}_x f (B_t) \uparrow \infty$ as $x\to \infty$, $\phi_b(x) \to \infty$ for any $b$, and $\phi_b(b) \ge c_q f (b-1)$ for some  constant $c_q$. Moreover, 
\begin{align}
\phi_b(0) &= \mathbb{E}_0 \int_0^{\tau_b} e^{-qt} f(B_t) \diff t + \mathbb{E}_0 e^{-q\tau_b} \phi_b(b) \\ &\le
\mathbb{E}_0 \int_0^{\infty} f(B_t) e^{-qt}\diff t  + c_b \phi_b(b),
\end{align}
where $0 < c_b <1$. Thus, $\phi_b (b) - \phi_b(0) \ge (1-c_b) \phi_b(b) - d \,\to\, \infty $ as $b\to \infty$ and, further,   $\phi_b(b) - \phi_b(0)$ is continuous as a function of $b$. 
%with the boundary conditions
%\begin{align*}
%\phi(b_+) &= \phi(b_-), \quad \phi(-b_+) = \phi(-b_-),\\
%\phi'(b_+) &= \phi'(b_-), \quad \phi'(-b_+) = \phi'(-b_-).
%\end{align*}

Now, to prove that $\phi$ is increasing, we proceed  in six stages:
 \begin{itemize}
 \item [(1)] Show that $\phi \in C^2 (\rr)$.
  \item [(2)] Show that $\phi(x) - \epsilon x \to \infty$ as $x\to \infty$ for some $\epsilon > 0$. 
 \item [(3)] Show that $\phi'$ has no negative minimum on the domain $(0, \hat{b})$ or on the domain $(\hat{b}, \infty)$.
 \item [(4)] Show that $\liminf_{x\to \infty} \phi' (x) \ge 0$.
 \item [(5)] Deduce that either $\phi' \ge 0$ on $\rp$ or
 \begin{itemize}
 \item [a)] $\phi'$ attains its unique negative minimum on $[0,\hat{b}]$ at $\hat{b}$ and
 \item [b)] $\phi'$ attains its unique negative minimum on $[\hat{b}, \infty]$ at $\hat{b}$.
 \end{itemize}
 \item [(6)] Deduce a contradiction from \eqref{Eq:HJB-3a} and \eqref{Eq:HJB-3b}.
 \end{itemize}
 Proof of:
 \begin{itemize}
 \item [(1)] Since $\phi$ is clearly positive and satisfies 
\eqref{Eq:HJB-3a} on $ (0,\hat{b})$,  it follows that $\phi$ is $C^2$ on $(0,\hat{b})$.  Similarly, since $\phi$ satisfies
 \eqref{Eq:HJB-3b} on  $(\hat{b}, \infty)$,  it follows that $\phi$ is $C^2$ on $(\hat{b}, \infty)$. A standard martingale argument based on the \ito-Tanaka formula shows that $\phi$ is $C^1$ on $\rr$. It then follows from the characterization of $\hat{b}$ that $\phi''$ does not have a discontinuity at $\hat{b}$.  
 \item [(2)]
 A simple argument show that, for $x \ge 1$, 
 \begin{equation}
 \phi (x) \ge e^{-(q+1)} \mathbb{E}_x \left [  \int_0^1 f(B_t) \diff t \right ] \ge e^{-(q+1)} \mathbb{P}_x \left (  \inf_{0\le t \le 1} B_t  \ge x-1\right ) \ge c_q f(x-1),
 \end{equation}
 for some $c_q > 0$. Since $f$ is convex, increasing on $\rp$, it is of at least linear growth on $\rp$, and so the result follows.
 \item [(3)] Denote $\phi'$ by $\psi$. It follows from differentiating \eqref{Eq:HJB-3a} and \eqref{Eq:HJB-3b} that
  \begin{equation}\label{Eq:HJB-3aD}
  \frac{1}{2} \psi'' - q \psi + f'  = 0, \quad \text{ on } \,\,\, (0,\hat{b}),
 \end{equation}
 and
  \begin{equation}\label{Eq:HJB-3bD}
  \frac{1}{2} \psi'' - (q+1) \psi + f'  = 0, \quad \text{ on } \,\,\, (\hat{b}, \infty).
 \end{equation}
 Since $f' > 0$ on $\rp$ the result follows from the strong minimum principle applied separately on each domain.
 \item [(4)] On $(\hat{b},\infty)$,  $ \frac{1}{2} \psi'' - (q+1) \psi  = -f' < 0$. It follows from the strong minimum principle that $\psi$ has no negative minimum on $(\hat{b},\infty)$. Consequently, if $m:= \liminf \psi < 0$, then once $\psi$ becomes negative it must decrease monotonically to $m$. But then $\lim \psi = -\infty$ which contradicts the positivity of $\phi$.
 \item [(5)] Note that $\psi (0) = 0$ by symmetry of $\phi$. So if $\psi$ has a negative minimum on $[0,\hat{b}]$ it follows from (3) that it must be attained at $\hat{b}$. Similarly for the negative minimum on $[\hat{b},\infty)$. 
 \item [(6)] Suppose that $\psi$ goes below $0$. Then from (5) we must have $0 < l < \hat{b} < r < \infty$ such that
 \[\{ x : \psi(x) < 0\} = (l,r)\]
 and $\psi$ is decreasing on $[l,\hat{b}]$ and increasing on $[\hat{b},r]$. It follows that $\phi'' = \psi'$ is negative on $(l,\hat{b})$ and zero at $\hat{b}$. Now define $h = q\phi - f$ and notice that $h(0) \ge 0$ since $f$ is minimised at $0$, while $h(\hat{b}) = \frac{1}{2} \phi''(\hat{b}) =0$. Note that $G^q h = -\frac{1}{2} f'' \le 0$ (since $f$ is convex) on $(0,\hat{b})$ so by the weak minimum principle the (negative) minimum of $h$ on $[0,\hat{b}]$ is attained at the boundary. But the boundary values are non-negative (since $h = \frac{1}{2} \phi'''$ on $[0,\hat{b}]$) so we deduce a contradiction.
 \end{itemize}
 \end{proof}

\begin{proof} (of Theorem \ref{Example3})
Observe that $\phi$ satisfies the HJB equation because $\phi$ solves \eqref{Eq:HJB-3a}-\eqref{Eq:HJB-3b}  and $\phi$ is increasing on $\rp$ (by Theorem \ref{T:Phi-Up}).  Condition (\textbf{TC}) is satisfied by Corollary \ref{C:TC-cond}. The validity of (\textbf{UI}) follows by Proposition \ref{qth-moments} as each $\nu \in \mathcal{M}_{\le 1}$. Therefore, the existence of the process $B^{b,x}$  satisfying \eqref{HJB-op} and Theorem \ref{VT0} imply the result.
\end{proof}

%%%%%%%%%%%%%%%%%%%%%%%
%%%%%%%%%%%%% PROOFS
%%%%%%%%%%%%%%%%%%%.  EXAMPLE 4 (Quadratic control)
%%%%%%%%%%%%%%%%%%%%%%%
\subsection{Proof of  Theorem \ref{P:1}}\label{S:Qcontrol}
For this proof, we need the following preliminary result.
\begin{lemma}\label{L:1}
The process $\hat{\Malpha}^\V{x} = (\hat{\Malpha}^\V{x}_t)_{t\in \Rp}$ as defined in \eqref{alpha*} 
 is an admissible policy in $\hat{\mathcal{A}}_{\V{x}}^2$. Furthermore, $J^{\hat{\alpha}} (\V{x}) < \infty$, for each $\V{x} \in \rd$. %  Furthermore, the corresponding admissible pair $(\hat{\V{X}}^\V{x}, \hat{\Malpha}^\V{x}) $ satisfies \eqref{E:TC}.	
 \end{lemma}
 \begin{proof} 
  First notice that (\textbf{H2}) holds directly as  $\hat{\Malpha}^\V{x}_s = (\hat{\Msigma},\hat{\nu}, \hat{\mu} (\hat{\V{X}}_s^\V{x}))$ for all $s$, and, further,    $\hat{\nu} \in \mathcal{M}$.  
    To prove  (\textbf{H1}),  it is enough to guarantee the existence of a filtered probability space $(\hat{\Omega}, \hat{\mathcal{F}}, (\hat{\mathcal{F}} )_{t\ge 0}, \hat{\mathbb{P}})$ supporting the process $\hat{\V{X}}^{\V{x}}$,  $\V{x} \in \rd$.   This is, however,  just a direct consequence of  \cite[Theorem~3.1]{SatoY1984}.
   Furthermore, Proposition 1.7 in  \cite[ Chapter 4]{EK} guarantees that, for every function $h$ in the domain of the (infinitesimal) generator $\hat{L}$  (and hence for every $h \in C_c^2 (\rd)$), 
 the process $M^h$ defined by $M^h_t:=h(\hat{\V{X}}_t^\V{x}) - \int_0^{t} (\hat{L} h)(\hat{\V{X}}_{s}^\V{x})\diff s$ is an $(\hat{\mathcal{F}})_t$-martingale. Thus,  the equality 
\begin{equation}\label{LhatL}
 (\hat{L} h)(\hat{\V{X}}_{s}^\V{x})  =  (L^{\hat{\Malpha}^\V{x}_s} h)(\hat{\V{X}}_{s}^\V{x}),\quad \hat{\mathbb{P}}-a.s.,  
 \end{equation}
implies \eqref{D:Mphi}.  Therefore, $\hat{\Malpha}^\V{x} \in \mathcal{A}_{\V{x}}$ and its associated controlled process   $\V{X}^{\hat{\Malpha}^{\V{x}}}$ is given by $\hat{\V{X}}^\V{x}$.

For the second part, let us recall that $|\V{G}\V{x}| \le ||\V{G}|| \,|\V{x}|$ for any matrix $\V{G} \in \Mn$, $\V{x}\in \rd$, whereas $|\V{x}^T\V{G}\V{x}| \le \lambda_n |\V{x}|$ for any positive definite matrix $\V{G}$ with  eigenvalues $0\le \lambda_1\le \ldots \le \lambda_n$. Hence, due to the quadratic form of the payoff function and the fact that   $\hat{\mu}$ is linear in $\V{x}$, to prove \eqref{E:TC} it is enough to show the inequality 
 \begin{align}\label{E:X2-est}
 \Ex \left [ |\hat{\V{X}}_t^\V{x}|^2\right] \le C (|\V{x}|^2  + 1),  \quad t \in \Rp,\,\,\V{x}\in\rd,
  \end{align}
 for some positive constant $C$ independent of both $t$ and $\V{x}$.  
 
By Theorem~2.12 in \cite{Duffie2003}, % (see also \cite{Masuda2004}) % \cite[Section 17, Chapter 3]{sato}) 
 $\hat{\V{X}}^\V{x}$ is a $\rd$-valued semimartingale  satisfying the stochastic differential equation
\begin{equation}\label{E:dX_tx}
\diff \hat{\V{X}}_t^\V{x} = -\V{Q} \hat{\V{X}}_t^\V{x} \diff t + (\V{u} + \V{v}) \diff t + \hat{\Msigma} \diff \V{W}_t + \diff \V{Z}_t,\quad\quad \hat{\V{X}}_0^\V{x} = \V{x},
\end{equation}
where $\V{W}=(\V{W}_t)_{t \in \Rp}$ is a standard $n$-dimensional Brownian motion  %with $[B^i, B^j]_t = t \, \delta_{ij}$ for any $t \in \Rp$;  the marginal processes $B_i (t)$ are mutually independent standard Brownian motions in $\rr$.
  and $\V{Z}=(\V{Z}_t)_{t \in \Rp}$ is  a $\rd-$valued pure-jump \levy martingale with quadratic variation  $[\V{Z}]$ taking values in $\Mn$, where the quadratic covariation entries $[Z_i,Z_j]$, $1\le i,j\le d$, are given by \[ [Z_i,Z_j]_t = t \int_{\Ro}y_i y_j \hat{\nu}(\diff y), \quad t \in \Rp.  \]

Using \ito's formula \cite[Theorem 33, Chapter 7, p. 81]{Protter}  one can verify that the solution $\hat{\V{X}}^\V{x} $ to \eqref{E:dX_tx} is given by  \begin{align}
\hat{\V{X}}^\V{x}_t  &=  e^{-\V{Q} t}\V{x} +   \int_0^t e^{-\V{Q} (t-s)} (\V{u}+ \V{v})  \diff s + \V{U}_t + \V{V}_t,  \label{E:X_tsol}
\end{align}
where 
\[  \V{U}_t:=  \int_0^t \hat{\Msigma }e^{-\V{Q} (t-s)}  \diff \V{W}(s) \quad \text{ and }\quad   \V{V}_t := \int_0^t e^{-\V{Q} (t-s)} \diff \V{Z}(s).  \]

Since both $\V{W}$ and $\V{Z}$ are martingales and, further, \[ \Ex \{[\V{U}]_t\}  \le \frac{n\,||\hat{\Msigma}||^2}{ ||\V{Q}||} < \infty \quad  \text{and} \quad  
   \Ex \{[\V{V}]_t\}  \le  \frac{n}{||\V{Q}||} \int_{\Ro} |\V{y}|^2\hat{ \nu}(\diff \V{y}) < \infty,\]
  \Out{Theorem~28 in  \cite[IV.2, p. 173]{Protter}  implies  that } it follows that $\V{U}$ and  $\V{V}$ are also  true martingales. Moreover,  by  Theorem~29 in   \cite[ p. 75]{Protter} and the generalised Ito isometry, it follows that  $\Ex \{\V{U}^2_t\} = \Ex \{[\V{U}]_t\}$ and $\Ex \{\V{V}^2_t\} = \Ex \{[\V{V}]_t\}$. Set $C:= \max \left \{ \,4,\,  |\V{u}+ \V{v}|^2 /||\V{Q}||+ n\,\delta^*/||\V{Q}|| \right \}$ and   $\delta^* := ||\hat{\Msigma}||^2 + \int_{\Ro} |\V{y}|^2 \hat{\nu} (\diff \V{y})$. Then, \eqref{E:X_tsol} and the above estimates imply  \eqref{E:X2-est}, which in turn implies  that (\textbf{H3}) holds and, hence, $\hat{\Malpha}^\V{x} \in \hat{\mathcal{A}}_{\V{x}}^2$. Finally, note that estimate  \eqref{E:X2-est} also ensures   $J(\hat{\V{X}}^{\V{x}},\hat{\Malpha}^\V{x}) < \infty$, as required.
 \end{proof}

  \begin{proof} (of Theorem \ref{P:1}) \newline \leavevmode
By Theorem \ref{VT-EFinite}, we need to prove that  the admissible pair $(\hat{\V{X}}^{\V{x}}, \hat{\Malpha}^{\V{x}})$ satisfies  \eqref{HJB-op} with $\phi(\V{x}) := \V{x}^T \mathbf{B} \V{x} + \textbf{c}\cdot \V{x} + d$, i.e. $\phi$ solves the HJB equation 
\begin{equation}\label{D:HJB11}
\inf_{\V{a}=(\Msigma, \nu,\Mmu) \in \V{A}} \left \{ L^\V{a} \phi (\V{x}) -q \phi(\V{x}) + (\V{x}^T \Lambda \V{x} + \Mmu^T\Theta \Mmu) \right\} = 0.
\end{equation}
We will see then that,  for each $\V{x}\in \rd$, the triplet $(\hat{\Msigma}, \hat{\nu},\hat{\mu}(\V{x}))$    is a minimiser of \eqref{D:HJB11}.

Since $\nabla \phi(\V{x}) = 2 \mathbf{B} \V{x} + \textbf{c}$ and $\Hess \phi(\V{x}) = 2\mathbf{B}$, it follows that 
\begin{equation}\label{La-Op}
L^\V{a} \phi(\V{x}) = (\V{u} + \Mmu)^T (2\mathbf{B}\, \V{x} + \textbf{c}) + \text{Tr} (\Msigma^T \mathbf{B}\, \Msigma) + \int_{\Ro} \V{y}^T \mathbf{B}\, \V{y} \,\nu (\diff \V{y}). 
\end{equation}
Thus, the minimal infinitesimal variance $\hat{\delta}$  in \eqref{D:delta}  yields 
\begin{equation}\label{E:HJB2}
 \V{u}^T (2\textbf{B}\V{x} + \textbf{c})  + \hat{\delta} -q (\V{x}^T \textbf{B} \V{x} + \textbf{c}^T \V{x} + d) + \V{x}^T \Lambda \V{x}  + \inf_{\Mmu\in \rd} \left \{ (2 \V{x}^T \textbf{B} +\textbf{c}^T) \Mmu  + \Mmu^T \Theta \Mmu \right\} = 0.
\end{equation}
 Let $g(\Mmu):=   (2 \V{x}^T \textbf{B} +\textbf{c}^T) \Mmu  + \Mmu^T \Theta \Mmu$. Then
 \begin{equation}\label{D:mu*}
   \mu^*(\V{x}) :=  - \, \frac{1}{2} \Theta^{-1} ( 2\textbf{B}\V{x} + \textbf{c})
   \end{equation}
minimises  $g$ for every $\V{x}\in \rd$ and, further, $g(\mu^*(\V{x}) )=   - ( 2\textbf{B}\V{x} + \textbf{c})^T \Theta^{-1} ( 2\textbf{B}\V{x} + \textbf{c}) / 4 $, which   (after  rearranging terms and substituting into equation \eqref{E:HJB2}) yields 
\begin{equation}\label{QRS}
\V{x}^T\left (\Lambda -q \textbf{B} - \textbf{B}^T \Theta^{-1} \textbf{B}  \right)  +   \left (2\V{u}^T \textbf{B} -  q \textbf{c}^T- \textbf{c}^T  \Theta^{-1} \textbf{B} \right ) \V{x}  + \left (\V{u}^T \textbf{c} +  \hat{\delta}  -q d - \frac{\textbf{c}^T \Theta^{-1} \textbf{c}}{4} \right)= 0.
\end{equation}
Observe that $ \Theta^{-1}$ exists as $\Theta$ is a positive definite matrix. Since equation \eqref{QRS} should hold for every $\V{x} \in \rd$, we can now verify that  $\V{B}$, $\textbf{c}$ and $d$ as defined in \eqref{D:abc} solve   the corresponding system of
 equations. Moreover, using the definitions of $\V{c}$ and $\V{v}$, we can see  that $\mu^*(\V{x})$ coincides with $ \hat{\mu}(\V{x}) = -\V{Q}\V{x} + \V{v}$, as required. 
 
Let us also   observe that the previous calculations imply that the function $\phi$ satisfies 
\begin{equation}\label{Lhat-HJB}
 \hat{L} \phi(\V{x}) - q \phi(\V{x}) + \V{x}^T \Lambda \V{x} + \Mmu^T \Theta \Mmu = 0,\quad \text{ for each } \V{x} \in \Rp,
 \end{equation}
 where  $\hat{L}$ is the operator defined in \eqref{D:Lhat}. Since, by Lemma \ref{P:lemma01}, the transversality condition \eqref{E:TC} holds,  Theorem \ref{VT-EFinite} implies the optimality of the family of admissible policies $\{\,\hat{\Malpha}^\V{x}\,:\, \V{x} \in \rd\,\}$, as desired.
   \end{proof}